\documentclass[11pt]{amsart}

\usepackage{graphicx} 
\usepackage{url,tabularx,array,geometry}
\usepackage{amssymb}
\usepackage{fancyhdr}
\usepackage{wrapfig}
\usepackage{epstopdf}
\usepackage{amsmath}
\usepackage[labelfont=bf]{caption}
\usepackage[hidelinks]{hyperref}
\usepackage{amsthm}
\usepackage{bbm}
\usepackage[T1]{fontenc}
\usepackage{yfonts}
\usepackage{breqn}
\usepackage{harpoon}
\usepackage[all]{xy}
\usepackage{tikz}
\usepackage{tikz-cd}
\usepackage{hyperref}
\usepackage{comment}
\usepackage{mathtools}
\usepackage{quiver}
\usepackage{multicol}
\usepackage{float}
\usepackage{blkarray}

\theoremstyle{plain}
\newtheorem{theorem}{Theorem}

\newtheorem{theoremN}{Theorem}[section]
\newtheorem{propositionN}[theoremN]{Proposition}

\newtheorem{lemmaN}[theoremN]{Lemma}
\newtheorem{conjectureN}[theoremN]{Conjecture}
\newtheorem{exampleN}[theoremN]{Example}

\theoremstyle{definition}
\newtheorem{definitionN}[theoremN]{Definition}
\theoremstyle{remark}

\newtheorem*{note}{Note}

\newtheorem*{example}{Example}

\usepackage[maxbibnames = 50]{biblatex}
\renewbibmacro{in:}{}
\DeclareFieldFormat{pages}{#1}
\title[Geometric Approach to Links-Quivers Correspondence I: Rational Tangles]{A Geometric Approach to the Links-Quivers Correspondence I: Rational Tangles}
\author{Jonathan A. Higgins}
\thanks{The
author was partially supported by NSF grant DMS-2405402.
}

\begin{document}

\address{Jonathan Higgins, Department of Mathematics, University of Illinois Urbana-Champaign, 1409 W Green St, Urbana, IL 61801, USA}
\email{jh110@illinois.edu}

\begin{abstract}
   The Links-Quivers Correspondence of \cite{KRSS17,K19} predicts that all the symmetric (or antisymmetric) colored HOMFLY-PT polynomials of a link can be recovered from a finite amount of data (a quiver) associated to the link. We give a new geometric proof of the Links-Quivers Correspondence modified for rational tangles (first proved in \cite{SW21}) and explicitly describe the corresponding quivers in terms of winding numbers in the punctured plane and its second configuration space.
\end{abstract}

\keywords {colored HOMFLY-PT polynomial, rational tangles, Links-Quivers, configuration spaces}

\maketitle

\tableofcontents

\section{Introduction}
\label{introsec}


For any link $L\subset \mathbb{R}^3$, there is an infinite family of associated invariants called the colored HOMFLY-PT polynomials. The Links-Quivers Correspondence, originally formulated by Kucharski, Reineke, Sto\v si\'c, and Sułkowski in \cite{KRSS17,K19}, predicts that these infinitely many invariants can be described by a finite amount of data associated with a symmetric quiver. Despite being relatively new, the conjecture has already resulted in many developments within mathematical physics and knot theory. For example, the symmetric quivers predicted by the conjecture are non-unique, so a major question associated with the conjecture is how these quivers can be related to each other, which has been studied systematically in \cite{CKNPSS25,EKL20,JKLN21,KLNS26}. See also \cite{EGGKPSS22,EKL23,JGKM23,K20,KRSS20,PSS18,SS25} for further developments and applications. In \cite{SW21}, Sto\v si\'c and Wedrich proved the conjecture for rational tangles and links inductively (and then extended this to the arborescent case in \cite{SW21ii}), but this left unanswered how the associated quivers relate to the geometry of these tangles and links. In this paper we will address this for rational tangles, and we will do so for rational links in the sequel \cite{JHpII26}. Particularly, we will determine the finite data predicted by the Links-Quivers Correspondence by using intersection models of Lagrangians in configuration spaces of the punctured plane. Such intersection models have been studied previously for quantum invariants by Lawrence \cite{L93}, Bigelow \cite{B07,B02}, and Anghel \cite{A22}, among others. Most relevant for our situation is the work of Wedrich \cite{W16}, who expressed the  $j$-colored HOMFLY-PT polynomial of a rational tangle in terms of the the $j$th configuration space. An important new feature of our work is that we can describe all the colored HOMFLY-PT polynomials of a rational tangle using only the first and second configuration spaces of the punctured plane. 

In \cite{RT90}, Reshitikhin and Turaev showed that we can get invariants of tangles in $B^3$ by coloring the strands with irreducible representations of simple Lie algebras. The $j$-colored HOMFLY-PT polynomial $\langle-\rangle_j$ encapsulates the invariants obtained by coloring all components with the representation $\Lambda^jV$, where $V$ is the vector representation of $\mathfrak{sl}_N$. (The dependence on \(N\) is given by setting $a=q^N$). The result is a Laurent polynomial for links, but, for tangles, it is a linear combination of basis elements in the skein module, or MOY-type webs (see \cite{MOY98}). 

In this paper, we will focus on rational tangles, which are a family of 4-ended tangles in bijection with $\mathbb{Q}\cup\{\infty\}$ (see Section \ref{rtanglesec} or \cite{Ku96}). The tangle associated with $u/v$ will be denoted $\tau_{u/v}$. For example, 
\[
\begin{tikzpicture}
    \node at (0,0) {$\tau_{3/1}=$};
    \node at (2,0) {\includegraphics[height=2cm]{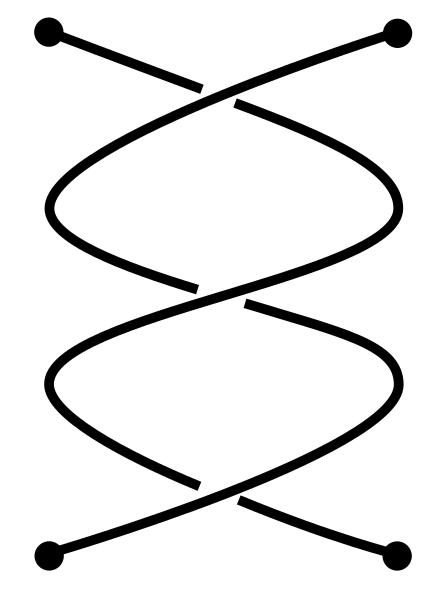},};
\end{tikzpicture}
\]
where the two strands are oriented upwards, and one can compute
\[
\begin{tikzpicture}
    \node at (0,0) {$\langle \tau_{3/1}\rangle_1=(-q^5+q^{3}-q)$};
    \node at (2.5,0) {\includegraphics[height=1cm]{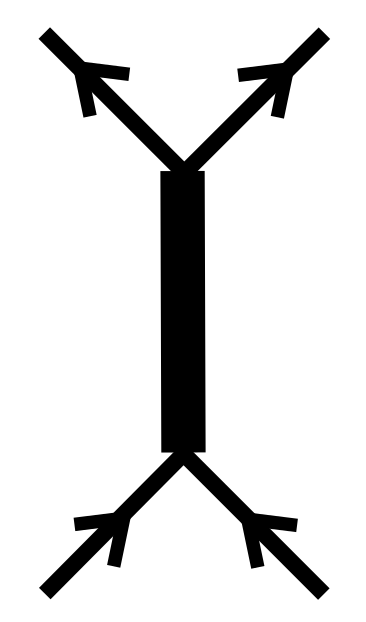}};
    \node at (3.2,0) {$+$};
    \node at (3.7,0) {\includegraphics[height=1cm]{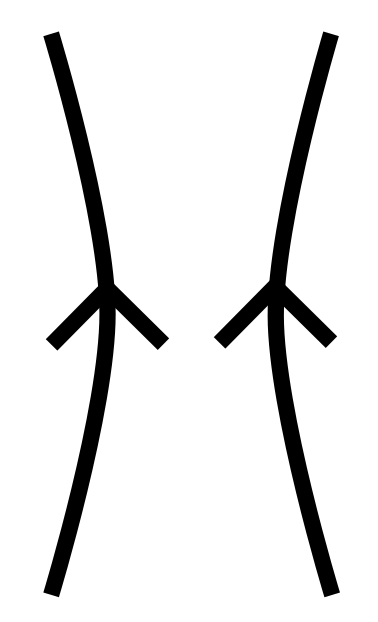}};
    
\end{tikzpicture}
\]
for the $1$-colored (or uncolored) HOMFLY-PT skein module evaluation of $\tau_{3/1}$.

Starting with the work of Khovanov \cite{Kh00}, much effort has gone into categorifying these invariants (see \cite{BN05,BN02,Kh07,KR08a,KR08b,MSV11,WW17}, among others). The categorified invariant of $\tau_{u/v}$ takes the form of a chain complex in a category whose generating objects are given by generating webs of the corresponding skein module. For example, the 1-colored (or uncolored) complex for $\tau_{3/1}$ is 
\begin{equation}
\label{t31cx}
\begin{tikzpicture}
    \node at (0,0) {$q^5a^0t^{-3}$};
    \node at (1,0) {\includegraphics[height=1cm]{web1.png}};
    \node at (1.7,0) {$\longrightarrow$};
    \node at (2.8,0) {$q^3a^0t^{-2}$};
    \node at (3.8,0) {\includegraphics[height=1cm]{web1.png}};
    \node at (4.5,0) {$\longrightarrow$};
    \node at (5.6,0) {$qa^0t^{-1}$};
    \node at (6.6,0) {\includegraphics[height=1cm]{web1.png}};
    \node at (7.3,0) {$\longrightarrow$};
    \node at (8.4,0) {$q^0a^0t^0$};
    \node at (9.4,0) {\includegraphics[height=1cm]{web0.png}};
\end{tikzpicture}
\end{equation}
where the objects shown represent the same basis webs appearing in the skein module evaluation for the tangle, and we have written the gradings multiplicatively in front of the objects in the complex. Note that ``$t$'' represents the homological grading.

These complexes are only defined up to homotopy and the category they are defined over depends on $N$. Nonetheless, Wedrich showed in \cite{W16} that the $j$-colored $\mathfrak{sl}_N$ homology of $\tau_{u/v}$ can be represented by a complex whose Poincar\'e polynomial stabilizes in $q,q^N,$ and $t$ for $N\gg 0$, so for large $N$, we can set $a=q^N$. For convenience, we will call these stabilized $\mathfrak{sl}_N$ complexes the colored HOMFLY-PT complexes, but we will only need the Poincar\'e polynomials, written in terms of $q,a,$ and $t$.

We use $\langle\langle\tau_{u/v}\rangle\rangle_j$ to denote the Poincar\'e polynomial of the $j$-colored HOMFLY-PT complex for $\tau_{u/v}$, which is the sum of all objects in the complex, written as shown in (\ref{t31cx}). Continuing our example with $\tau_{3/1}$, we get
\[
\begin{tikzpicture}
    \node at (0,0) {$\langle\langle \tau_{3/1}\rangle\rangle_1=(q^5t^{-3}+q^{3}t^{-2}+qt^{-1})$};
    \node at (3.3,0) {\includegraphics[height=1cm]{web1.png}};
    \node at (4,0) {$+$};
    \node at (4.5,0) {\includegraphics[height=1cm]{web0.png}.};
\end{tikzpicture}
\]
As should be expected with Poincar\'e polynomials, $\langle\tau_{u/v}\rangle_j$ and $\langle\langle \tau_{u/v}\rangle\rangle_j$ are related by
\[
\langle\tau_{u/v}\rangle_j=\langle\langle\tau_{u/v}\rangle\rangle_j\bigg|_{t\mapsto -1}.
\]

The Links-Quivers Correspondence (adjusted for tangles) says that 
the generating function $\mathcal{P}(\tau_{u/v})= \sum_{j\geq 0} \langle\langle \tau_{u/v}\rangle\rangle_j$ can be expressed in terms of a ``quiver,'' which we view as a quadratic form $Q$, along with three linear forms, all defined on a free $\mathbb{Z}$-module $\mathbb{Z}\mathcal{G}_{u/v}^1$ equipped with a preferred basis $\mathcal{G}_{u/v}^1$. Associated with $\tau_{u/v}$ is an arc Lagrangian $\alpha_{u/v}$ in the 3-punctured plane and two types of vertical Lagrangians passing between the punctures, as shown in Figure \ref{D52fig} for $\tau_{5/2}$. The preferred basis for $\mathbb{Z}\mathcal{G}_{u/v}^1$ is given by the intersections of $\alpha_{u/v}$ with the two vertical lines. 
If we use $\mathcal{D}(\tau_{u/v})$ to denote this picture in the 3-punctured plane consisting of the three Lagrangians for $\tau_{u/v}$, then we may state the main theorem as follows.

\begin{theorem}
\label{introtanglethm}
    Given a rational tangle $\tau_{u/v}$, $\mathcal{P}(\tau_{u/v})$ may be written as 
    \begin{multline}
    \label{calPtuvformula}
        \sum_{\bf{d}=(d_1,...,d_{u+v})\in \mathbb{N}^{u+v}}q^{S\cdot\textbf{d}+\textbf{d}\cdot Q\cdot \textbf{d}^t}a^{A\cdot \textbf{d}} t^{T\cdot\textbf{d}} \genfrac{[}{]}{0pt}{}{d_1+...+d_u}{d_1,...,d_u} \genfrac{[}{]}{0pt}{}{d_{u+1}+...+d_{u+v}}{d_{u+1},...,d_{u+v}}\\ \times X[d_1+...+d_{u+v},d_1+...+d_u],
    \end{multline}
    where $S,A,$ and $T$  may be computed via winding numbers about the three punctures in $\mathcal{D}(\tau_{u/v})$ of loops based at the Lagrangian intersections. Furthermore, $Q$ may be computed in a similar manner by passing to the second configuration space of the 3-punctured plane, where we also need to consider loops winding around the diagonal $\Delta$.
\end{theorem}

A precise statement about $S, A, T$, and $Q$ is given in Theorem \ref{bigthm}.

\begin{figure}
    \centering
    \includegraphics[height=4cm]{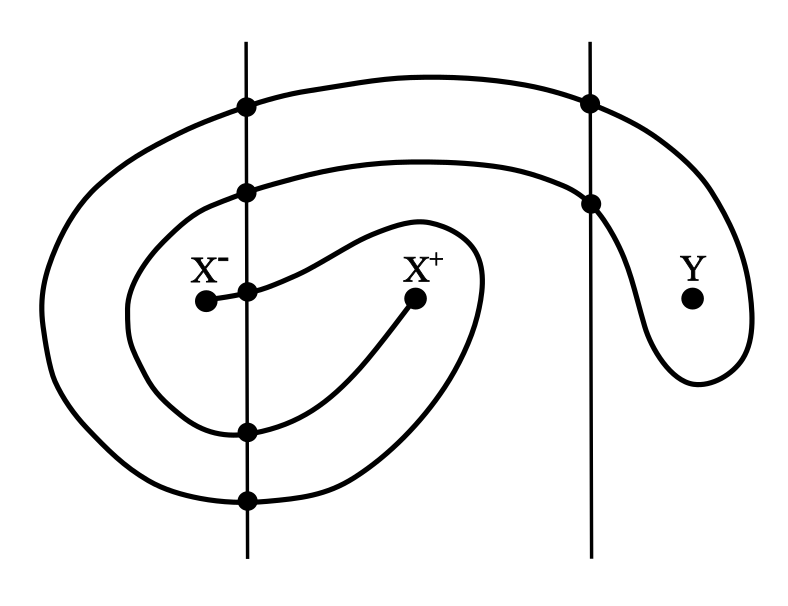}
    \caption{The diagram $\mathcal{D}(\tau_{5/2})$ used in Theorem \ref{introtanglethm} for the tangle $\tau_{5/2}$}
    \label{D52fig}
\end{figure}

One could also define $P(\tau_{u/v})=\mathcal{P}(\tau_{u/v})|_{t=-1}$, which is the generating function for tangles studied in \cite{SW21} by Sto\v si\'c and Wedrich (in a rescaled form). In particular, their rescaled form $P'(\tau_{u/v})$ involves multiplying $P(\tau_{u/v})$ by the quantum binomial that combines the two multinomials from (\ref{calPtuvformula}) into one, which is helpful for proving their theorem about rational knots, but it no longer relates to the geometry of $\tau_{u/v}$ in a clear way. Furthermore, this rescaling does not affect $Q$ or any of the linear forms. In \cite{SW21}, they prove by induction that $P'(\tau_{u/v})$ can be written in a way analogous to (\ref{calPtuvformula}), which they call its ``quiver form.'' The main difference is that, in $P'(\tau_{u/v})$ and $P(\tau_{u/v})$, one no longer has the linear form $T$ and $q^{S\cdot\textbf{d}}$ is replaced with $(-q)^{S\cdot \textbf{d}}$. Thus, Theorem \ref{introtanglethm} essentially provides a new way to show that $\mathcal{P}(\tau_{u/v})$ can be put in an appropriately defined quiver form by using the geometry of $\tau_{u/v}$.

\bigskip
\noindent
\textbf{Links-Quivers for Rational Links}
In \cite{JHpII26}, we will use Theorem \ref{introtanglethm} to give a new geometric interpretation of the Links-Quivers Correspondence for rational links. In its standard formulation, the Links-Quivers Correspondence claims that the generating function for the colored HOMFLY-PT polynomials of a link can be described in terms of the generating function for the cohomological Hall algebra of a symmetric quiver, after a change of variables. This generating function for a symmetric quiver of $m$ vertices given by an adjacency matrix $Q$ has the form
\begin{equation}
P_Q(x_1,...,x_m) = \sum_{\textbf{d}=(d_1,..,d_m)\in\mathbb{N}^m} (-q)^{-\langle \textbf{d},\textbf{d}\rangle_Q}\prod_{i=1}^m \frac{x_1^{d_1}...x_m^{d_m}}{(1-q^{-2})...(1-q^{-2d_i})},
\end{equation}
where $\langle \textbf{d},\textbf{e}\rangle_Q = \sum_{i,j}(\delta_{ij}- Q_{ij})d_ie_j$ is the Euler form for the quiver and $x_1,...,x_m$ are formal variables. The generating function encodes information about the representations of the quiver and is used to define its Donaldson-Thomas (DT) invariants. See \cite{R24} for a survey on the generating functions and DT invariants for symmetric quivers. If $P^{\bigwedge^j}_L(q,a)$ is the $\bigwedge^j$-colored HOMFLY-PT polynomial for a link $L$, then the Links-Quivers Correspondence could be phrased as follows.
\begin{conjectureN}[Links-Quivers Correspondence]
\label{LQCC}
   Given a link $L$ with $|L|$ components, if we set 
   \[
   P(L)=\sum_{j\geq 0}\left(\prod_{i=1}^j(1-q^{2i})\right)^{|L|-2}P^{\bigwedge^j}_L(q,a)x^j,
   \]
   then there is a symmetric quiver $Q_L$ such that $P(L)$ and $P_{Q_L}(\overline{x})$ are equal after a change of variables.
\end{conjectureN}

We will study $P(L)$ and this statement from a geometric perspective for rational links in \cite{JHpII26}, which will ultimately rely upon our work in this paper. The original proof of Conjecture \ref{LQCC} for rational links by Sto\v si\'c and Wedrich in \cite{SW21} consisted of showing that $P(L)$ can be put in a ``quiver form,'' similar to what was done for rational tangles. To do so, they applied a closure operation to the rescaled $P(\tau_{u/v})$, also put in quiver form. Thus, we see how showing that $\mathcal{P}(\tau_{u/v})$ can be written as in (\ref{calPtuvformula}) is essentially a reformulation of the Tangles-Quivers Correspondence for rational tangles.

\bigskip
\noindent
\textbf{Acknowledgments} This work was accomplished as part of the author's PhD research, so he would like to thank his advisor, Jacob Rasmussen, for his invaluable feedback and encouragement during the process of writing this paper. He is also grateful for helpful discussions with Ada Stelzer and Abbigail Price on related combinatorial problems. The author would also like to thank Marko Sto\v si\'c, Paul Wedrich, Piotr Kucharski, and Piotr Su\l{}kowski for helpful comments on an earlier draft.

\bigskip
\noindent
\textbf{Structure of Paper} Section \ref{bgrd} will contain some background information necessary for this paper. This will include some further details about quantum algebra, rational tangles, colored HOMFLY-PT skein module evaluations, prior geometric results for computing $\langle\langle\tau_{u/v}\rangle\rangle_j$, and quiver forms. In Section \ref{mainthmsec}, we set up and present the main theorem. We will also show that $Q$ is symmetric when computed by our formula. In section \ref{pfsec}, we will prove the theorem, and Section \ref{indreducesec} will be devoted to showing how we can reduce the number of indices used in the sum. Finally, we will provide some examples in Section \ref{exsec}.

\section{Background}
\label{bgrd}

We provide a review of the necessary definitions, notation, and known results relevant to this paper before introducing the new results.

\subsection{Some Quantum Algebra}
\label{qalgsec}

We will frequently use Pochhammer symbols to condense notation when performing calculations and stating results. In their general form, the \textit{Pochhammer symbols} are defined as $(x;y)_i=\prod_{j=0}^{i-1}(1-xy^j).$ In this paper, there are two particular types of Pochhammer symbols that we will encounter: the \textit{q-Pochhammer symbols}, given by 
\begin{equation*}
    (q^2;q^2)_i=\prod_{j=1}^{i}(1-q^{2j}),
\end{equation*}
 and a slight modification of them:
 \begin{equation*}
     (-t^{-1}q^2;q^2)_i=\prod_{j=1}^{i}(1+t^{-1}q^{2j}),
 \end{equation*}
which will be used when discussing Poincar\'e polynomials and reducing indices in Section \ref{indreducesec}. Note that these two types of Pochhammer symbols are related by setting $t=-1$.

Using the q-Pochhammer symbols, we can define quantum binomials and multinomials. Given $d_1+...+d_m=j$, we define the \textit{quantum multinomial} as 
\begin{equation*}
    \genfrac{[}{]}{0pt}{}{j}{d_1,...,d_m}=\frac{(q^2;q^2)_j}{(q^2;q^2)_{d_1}...(q^2;q^2)_{d_m}},
\end{equation*}
and the \textit{quantum binomial} is just the case where $m=2$. We will, however, adopt the notation of \cite{SW21} by writing 
\begin{equation*}
    \genfrac{[}{]}{0pt}{}{j}{d_1}_+=\genfrac{[}{]}{0pt}{}{j}{d_1,j-d_1}=\frac{(q^2;q^2)_j}{(q^2;q^2)_{d_1}(q^2;q^2)_{j-d_1}}.
\end{equation*}

It should be noted that the quantum multinomials simplify to give polynomial expressions in $q$. Next, we state an identity involving the Pochhammer symbols, which may be found in \cite{K19}.

\begin{lemmaN}
\label{algidentity}
    For $d_1,...,d_m\geq 0$, the following holds:
    \begin{equation*}
        \frac{(x^2;q^2)_{d_1+...+d_m}}{(q^2;q^2)_{d_1}...(q^2;q^2)_{d_m}}=\sum_{\substack{\alpha_1+\beta_1=d_1\\...\\\alpha_m+\beta_m=d_m}}\frac{(-x^2q^{-1})^{\alpha_1+...+\alpha_m}q^{\alpha_1^2+...+\alpha_m^2+2\sum_{i=1}^{m-1}\alpha_{i+1}(d_1+...+d_i)}}{(q^2;q^2)_{\alpha_1}...(q^2;q^2)_{\alpha_m}(q^2;q^2)_{\beta_1}...(q^2;q^2)_{\beta_m}}.
    \end{equation*}
\end{lemmaN}

One particular consequence of this lemma is 
\begin{equation*}
    \frac{(-t^{-1}q^2;q^2)_d}{(q^2;q^2)_d}=\sum_{\alpha+\beta=d}\frac{(t^{-1}q)^\alpha q^{\alpha^2}}{(q^2;q^2)_\alpha(q^2;q^2)_\beta}.
\end{equation*}

Finally, we present one quantum multinomial identity that we will need in Section \ref{mainthmsec}. Following the notation of \cite{AM20}, we first need to define the following.

\begin{definitionN}
\label{seqdef}
    Let $\mathcal{S}_j^k(d_1,...,d_k)$ denote the set of sequences in $[k]:=\{1,...,k\}$ of length $j$ involving $d_1$ 1's, $d_2$ 2's,..., and $d_k$ k's.
\end{definitionN}

Now, if $\sigma=(\sigma_1,...,\sigma_j)$ is a sequence in $\mathcal{S}_j^k(d_1,...,d_k)$, then
\[
\text{inv}(\sigma):=|\{(a,b)\,:\, a<b\,\, \text{and}\,\,\sigma_a>\sigma_b\}|.
\]
Equipped with this notation, we have the following combinatorial result from \cite{AM20}.

\begin{propositionN}[Proposition 1.7, \cite{AM20}]
\label{qalgcombprop}
If $j,k,d_1,...,$ and $d_k$ are non-negative integers such that $d_1+...+d_k=j$, then
\begin{equation}
    {j\brack d_1,...,d_k} = \sum_{\sigma\in \mathcal{S}_j^k(d_1,...,d_k)}q^{2\,\text{inv}(\sigma)}.
\end{equation}
\end{propositionN}

It should be noted that the proposition in \cite{AM20} does not involve the $2$ in the exponent, but this is because the authors define everything in terms of $q$ rather than $q^2$.

\subsection{Rational Tangles}
\label{rtanglesec}


\begin{figure}
    \raisebox{0pt}{\includegraphics[height=3cm, angle=0]{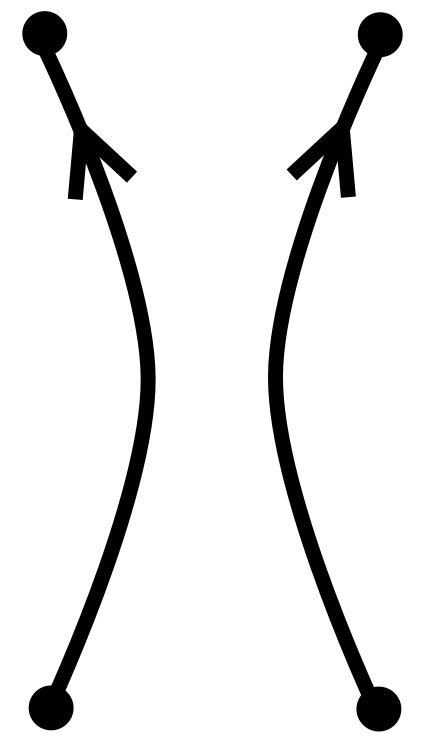}}\qquad \qquad \raisebox{0pt}{\includegraphics[height=3cm, angle=0]{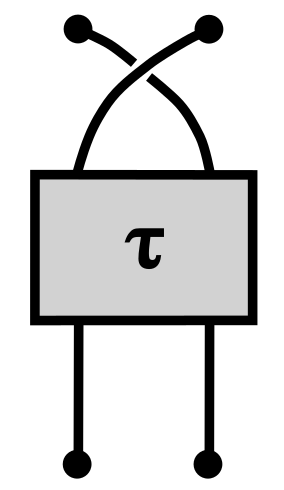}} \qquad \qquad \raisebox{0pt}{\includegraphics[height=3cm, angle=0]{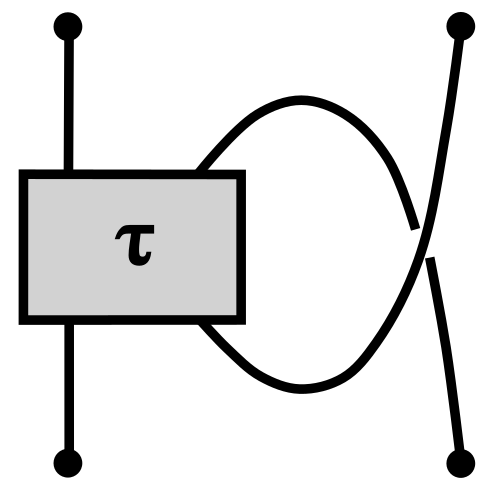}}
    \caption{Left: the trivial tangle $\tau_{0/1}$. Center: top twist rule. Right: right twist rule. For twist rules, $\tau$ denotes the tangle being twisted, and the orientation is determined by the orientation of endpoints on $\tau$.}
    \label{trivialandtwists}
\end{figure}

\textit{Positive rational tangles} are oriented tangles with four ends, which can be built inductively from the trivial tangle by adding a finite sequence of top and right twists. The trivial tangle, denoted $\tau_{0/1}$, is defined as the tangle with two untangled upward-oriented strands, as shown on the left in Figure \ref{trivialandtwists}. The twist rules are also shown in Figure \ref{trivialandtwists}. These twists will change the orientation of the strands, which we will want to keep track of. Following \cite{W16}, we will use the notation of UP, OP, and RI to denote the orientations at the boundary of the arcs, as shown below.

\[
\vcenter{\hbox{\includegraphics[height=3cm,angle=0]{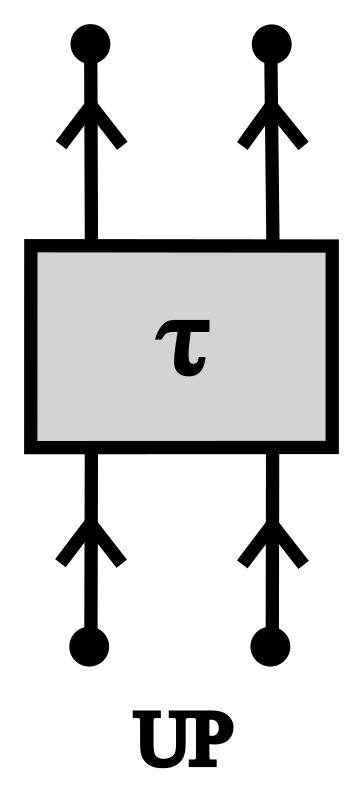}}} \qquad \qquad 
\vcenter{\hbox{\includegraphics[height=3cm,angle=0]{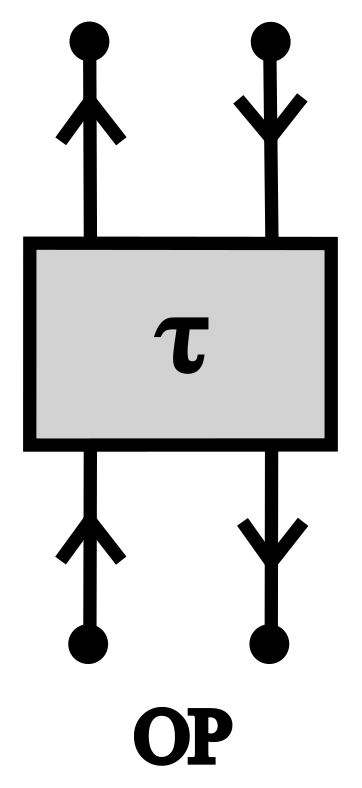}}} \qquad \qquad
\vcenter{\hbox{\includegraphics[height=3cm,angle=0]{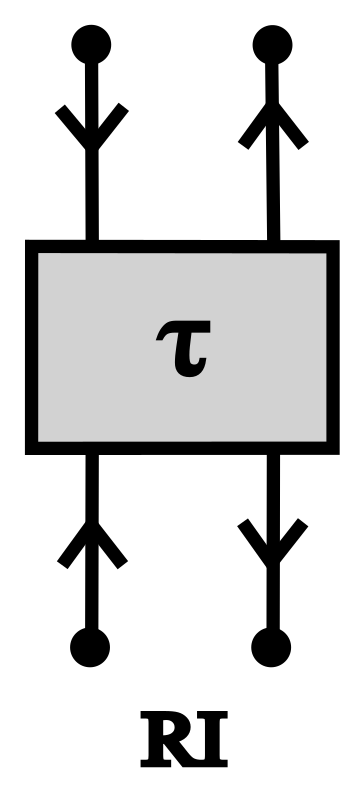}}}
\]

Perhaps the best way to think of these tangles is by placing them in the standard embedding of the unit ball $B^3$ in $\mathbb{R}^3$, with the four endpoints being located at $(0,\pm\frac{\sqrt{2}}{2},\pm\frac{\sqrt{2}}{2})$, which can be thought of as the NE, NW, SE, and SW points on the ball projected to the plane $x=0$. We may consider the double branched cover of the boundary sphere with the four endpoints being the branch points. Topologically, the double branched cover is a torus, so we write $\Sigma(S^2,4)\simeq S^1\times S^1$. Now, consider the four arcs connecting the branch points, as shown in Figure \ref{trivial}. The top and bottom arcs are colored blue and the two side arcs are colored red. Note that we may define the double branched cover so that the blue arcs lift to meridians of the torus $\Sigma(S^2,4)$ and the red arcs lift to longitudes. We take the top blue arc and the red arc on the right, whose lifts are generators $[m]$ and $[l]$ for $H_1(\Sigma(S^2,4))\cong \mathbb{Z}\times \mathbb{Z}$.

The trivial tangle $\tau_{0/1}$ is simply the tangle with its two arcs oriented from SE and SW to NE and NW, respectively. Observe that there is a disk $B^2$ separating the two tangle components, as seen in Figure \ref{trivial} on the right. The boundary of the disk is a simple closed curve $\gamma\subset\partial B^3=S^2$, and it lifts to two disjoint longitudes on the torus $\Sigma(S^2,4)$. Let $\tilde{\gamma}$ be one of these longitudes. With respect to the picture we have been setting up, the top twist is given by applying a Dehn twist along the top blue arc and the right twist is given by applying a Dehn twist along the red arc on the right. These twists will affect how the curve $\gamma$ winds around the boundary sphere, which will, consequently change the lift $\tilde{\gamma}$.

\begin{figure}
\centering
\raisebox{0pt}{\includegraphics[height=3.5cm, angle=0]{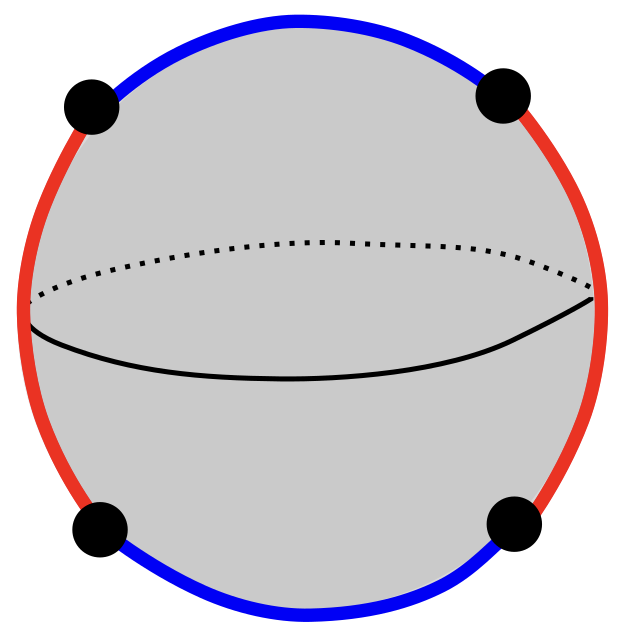}} \qquad \qquad
\raisebox{0pt}{\includegraphics[height=3.5cm, angle=0]{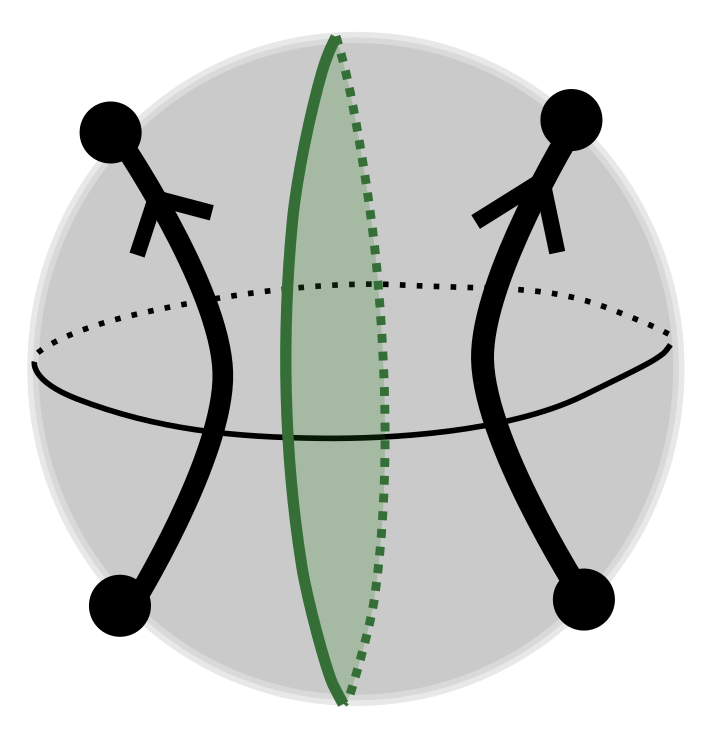}}
\caption{Left: the four colored arcs connecting the branch points on the boundary $S^2$. Right: the disk separating the trivial tangle in $B^3$.}
\label{trivial}
\end{figure}

It is easy to see that the Dehn twists just mentioned lift to Dehn twists along the lifted meridian and longitude in $\Sigma(S^2,4)\simeq S^1\times S^1$, and that $\tilde{\gamma}$ is a closed 1-chain determined by these lifted Dehn twists applied to the meridional lift of the trivial tangle. The homology class can be written as $[\tilde{\gamma}]=u[m]+v[l]\in H_1(S^1\times S^1)$ with $\text{gcd}(u,v)=1$, and taking the ratio $u/v$ gives us a way to distinguish the corresponding rational tangles. We write $\tau_{u/v}$ for the tangle given by this lifted curve $\tilde{\gamma}$. It is well-known that rational tangles are distinguished (up to isotopy) by the corresponding fraction $u/v$.

Now, we must revisit the effects of the Dehn twists in order to determine the corresponding fraction for the rational tangle. At the level of the tangle itself, we will use $T$ to denote the top twist and $R$ to denote the right twist. As we already said, the lift of $T$ gives a Dehn twist $\tau\in\text{Mod}(S^1\times S^1)$, the mapping class group of the torus, and the lift of $R$ is another Dehn twist $\rho \in \text{Mod}(S^1\times S^1)$ such that 
\begin{multicols}{2}
\begin{align*}
\tau:\begin{cases}
    m\mapsto m\\
    l\mapsto m+l
\end{cases}
\end{align*}

\begin{align*}
 \rho:\begin{cases}
    m\mapsto m+l\\
    l\mapsto l,
\end{cases} 
\end{align*}
\end{multicols}
\noindent so $\tau(u[m]+v[l])=(u+v)[m]+v[l]$ and $\rho(u[m]+v[l])=u[m]+(u+v)[l]$, which gives
\[
T\tau_{u/v}=\tau_{(u+v)/v}
\qquad\qquad
R\tau_{u/v}=\tau_{u/(u+v)}.
\]

Thus, we have a way of determining the fraction $u/v$ given a finite sequence of top twists and right twists applied to the trivial tangle $\tau_{0/1}$. We know that we may take the first twist to be a top twist $T$ because $R\tau_{0/1}=\tau_{0/1}$. Now, we first consider the case where $u/v\geq1$, which is the case where the final twist is a top twist $T$. This means that we may write $T^{a_n}R^{a_{n-1}}...R^{a_2}T^{a_1}\tau_{0/1}=\tau_{u/v}$, where the rules for $T$ and $R$ give 
\[
\frac{u}{v}=a_n+\frac{1}{a_{n-1}+\frac{1}{a_{n-2}+...+\frac{1}{a_2+\frac{1}{a_1}}}}.
\]
Otherwise, if $0<u/v<1$, the final twist is a right twist $R$, so we have $R^{a_n}T^{a_{n-1}}...R^{a_2}T^{a_1}\tau_{0/1}=\tau_{u/v}$ with
\[
\frac{u}{v}=\frac{1}{a_n+\frac{1}{a_{n-1}+...+\frac{1}{a_2+\frac{1}{a_1}}}}.
\]
We may record these sequences of $a_i$ by $[a_n,a_{n-1},...,a_2,a_1]$; in the first case above, $n$ is odd, whereas it is even in the second case. 

Now that we have carefully defined rational tangles and have determined how to distinguish them, we will describe another way of visualizing them that will be helpful for the content of this paper. In particular, for the tangle $\tau_{u/v}$, viewed as existing in $B^3$, we can push it onto the boundary sphere. In other words, starting with the trivial tangle $\tau_{0/1}$, we can push the two strands to be parallel to the two red arcs on the boundary sphere, so we know both lift to longitudes of $\Sigma(S^2,4)\simeq S^1\times S^1$, and then we apply the same sequence of top and right twists (as Dehn twists on the boundary sphere) that were used to construct $\tau_{u/v}$. Let $\alpha$ and $\alpha'$ be the two arcs on $S^2$ that we get for $\tau_{u/v}$ after applying these twists, with $\alpha$ being the arc with endpoint SE and $\alpha'$ being the arc with endpoint SW (the twist rules guarantee that these give the two distinct arcs). It is not difficult to see that their lifts $\tilde{\alpha}$ and $\tilde{\alpha'}$ on the torus satisfy $[\tilde{\alpha}]=[\tilde{\alpha'}]=[\tilde{\gamma}]\in H_1(\Sigma(S^2,4)).$ Thus, we could have used any one of these to determine the tangle $\tau_{u/v}$.

Since $S^2\simeq\mathbb{C}\cup\{\infty\}$, the 4-punctured sphere is homeomorphic to the 3-punctured plane. Thus, if we drop the arc $\alpha'$, then applying the stereographic projection with respect to the SW point results in the arc $\alpha$ being projected to the 3-punctured plane. 

\begin{definitionN}
    Given a rational tangle $\tau_{u/v}$, let $\alpha_{u/v}$ be the corresponding arc $\alpha$ in the 3-punctured plane obtained by the procedure just described.
\end{definitionN}

Next, we need a way to determine the shape of this arc for particular $u/v$; the following lemma tells us how to do so. The lemma is taken from Section 4.1 of \cite{W16}, adjusted to our conventions.

\begin{lemmaN}
\label{bridgerecipe}
    Up to isotopy, $\alpha_{u/v}$ may be drawn as follows:
    Draw the intervals $[-2,-1]$ and $[1,2]$ on the real axis and partition them into $u$ parts of equal size. If $u\geq v$, starting at the points $\pm1$, label the ends of the partitions on both line segments by $0,1,...,2u-1\in \mathbb{Z}/(2u)\mathbb{Z}$ in a clockwise manner, as in Figure \ref{4_3_tangle} on the top left. Starting at the 0-labeled point $\text{par}(u)\in\{\pm1\}$, the parity of $u$, draw the arc $\alpha_{u/v}$ by proceeding to the point labeled $v$ on the opposite side, then going to $2v$ on the same side you started (mod $2v$), and proceed until you hit $\{-2,-1,1,2\}$ again. The three punctures in the plane are the points at $x=\pm1, 2$ on the real axis. If $u<v$, reverse the roles of $u$ and $v$, label the partition points in a counter-clockwise manner, and proceed in a similar manner as in the previous case, but start at the 0-labeled point $-\text{par}(v)\in\{\pm1\}$. (Equivalently, reflect the picture for $\tau_{v/u}$ across the imaginary axis.) In this case, the three punctures in the plane are at $x=\pm1,-2$ on the real axis.
\end{lemmaN}

\begin{figure}
   \begin{tikzpicture}
       \node at (0,0) {\includegraphics[height=3cm]{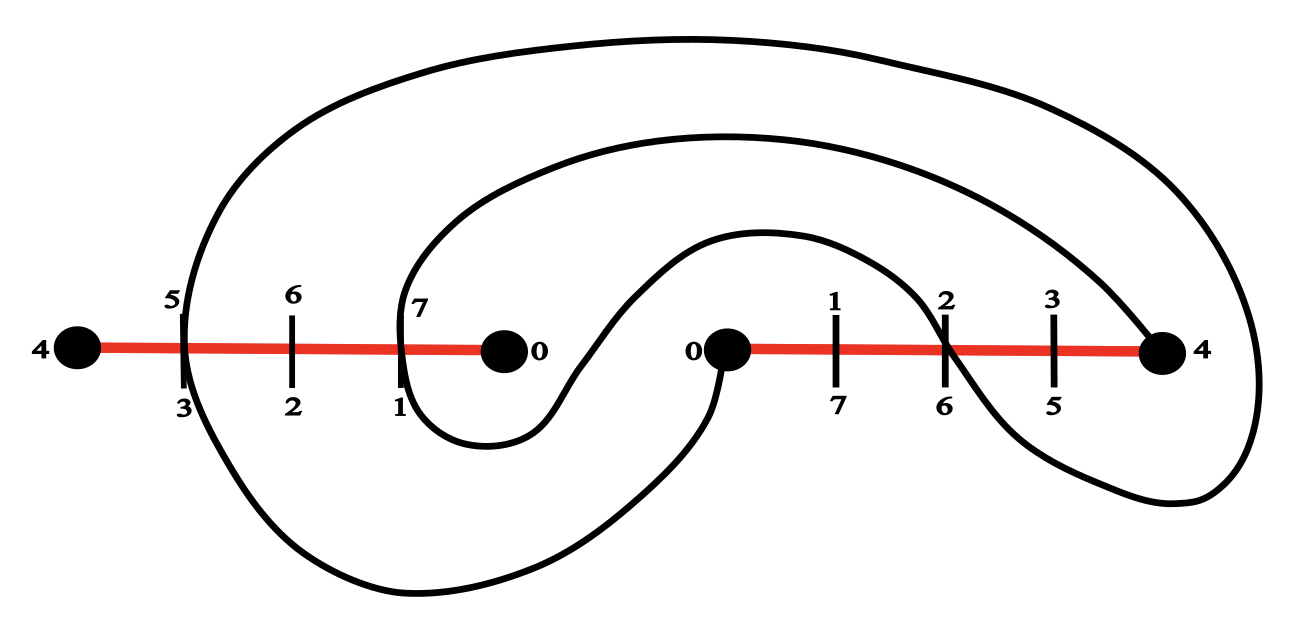}};
       \node at (3.8,0) {$\longrightarrow$};
       \node at (7,0) {\includegraphics[height=3cm]{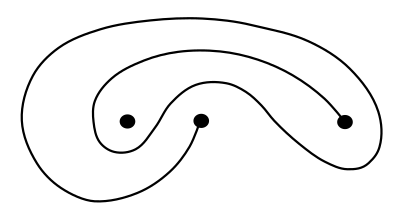}};
        \node at (0,-4) {\includegraphics[height=3cm]{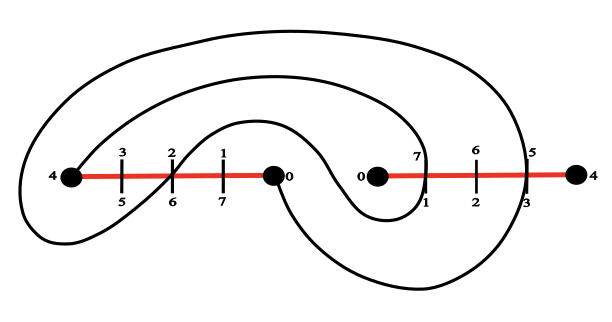}};
       \node at (3.8,-4) {$\longrightarrow$};
       \node at (7,-4) {\includegraphics[height=3cm]{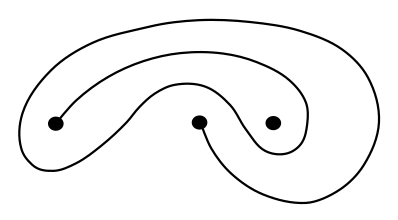}};
   \end{tikzpicture}
\caption{An application of Lemma \ref{bridgerecipe} to $\tau_{4/3}$ (top) and $\tau_{3/4}$ (bottom)}
\label{4_3_tangle}
\end{figure}

Furthermore, up to isotopy, we can make this picture as "tight" as possible to fix it uniquely. We can ensure that $\alpha$ has no self-intersections and has minimal intersections with the intervals $[-2,-1]$ and $[1,2]$ by pushing off bigons if necessary (strictly following the recipe in Lemma \ref{bridgerecipe} should avoid this, though). Additionally, in the $u\geq v$ case, we can guarantee that $\alpha_{u/v}$ is disjoint from $(-\infty,-2]$, which allowed us to throw out $x=-2$, and it is disjoint from $[2,\infty)$ if $u<v$, so we threw out $x=2$ in this case.

Now, we want to define two types of vertical Lagrangians in the 3-punctured plane. Assuming we have followed the procedure from Lemma \ref{bridgerecipe} to obtain $\alpha_{u/v}$, if $u\geq v$, let $l_A$ denote a vertical Lagrangian which intersects the real axis on $(-1,1)$ and let $l_I$ be a parallel vertical Lagrangian intersecting the interval $(1,2)$ and having the minimal number of intersections with $\alpha_{u/v}$. In the case where $u<v$, we let $l_A$ be a vertical Lagrangian intersecting $(-2,-1)$ having the minimal number of intersections with $\alpha_{u/v}$, and we let $l_I$ be a parallel Lagrangian intersecting the interval $(-1,1)$ on the real axis.  

\begin{note}
    Given the notation we have set up, we have
    \[
    |l_A\cap \alpha_{u/v}|=u
    \qquad\text{and}\qquad
    |l_I\cap \alpha_{u/v}|=v
    \]
    regardless of how $u$ and $v$ are related.
\end{note} 

\begin{definitionN}
\label{in_activedef}
    We will call $l_A$ and $l_I$ the \textit{active} and \textit{inactive axes}, respectively, and the intersection points $l_A \cap \alpha$ \textit{active intersection points} and $l_I\cap \alpha$ the \textit{inactive intersection points}.
\end{definitionN}

\begin{figure}
    \begin{tikzpicture}
        \node at (0,0) {$\mathcal{D}(\tau_{4/3})=$};
        \node at (3,0) {\includegraphics[height=3cm]{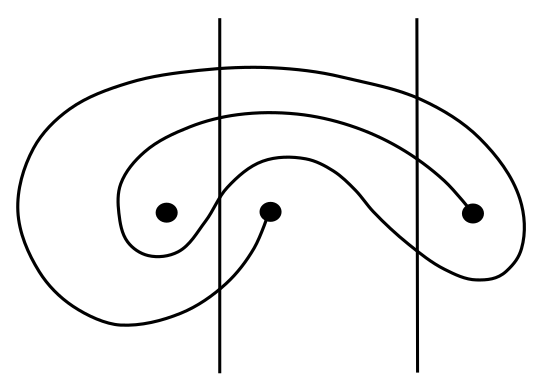}};
        \node at (7,0) {$\mathcal{D}(\tau_{3/4})=$};
        \node at (10,0) {\includegraphics[height=3cm]{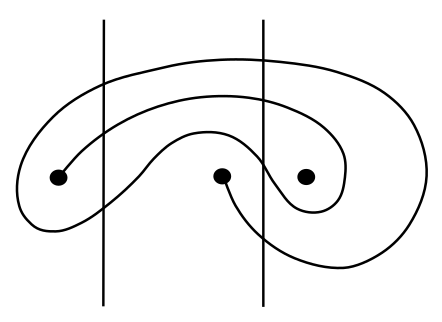}};
    \end{tikzpicture}
    \caption{Examples of the $\mathcal{D}(\tau_{u/v})$ notation}
    \label{DT43fig}
\end{figure}

In general, we may draw these verticals without labeling them, but it is important to remember that $l_A$ is always on the left. We are also now ready to define $\mathcal{D}(\tau_{u/v})$, which we referred to in the introduction. Examples are shown in Figure \ref{DT43fig}.

\begin{definitionN}
    Given $\tau_{u/v}$, let $\mathcal{D}(\tau_{u/v})$ be the picture in the 3-punctured plane consisting of $\alpha_{u/v}$, $l_A$, and $l_I$.
\end{definitionN}

In the following lemma, also taken from \cite{W16}, we describe how $\mathcal{D}(\tau_{u/v})$ is built inductively by seeing what the twist operations do to the verticals $l_A$ and $l_I$.

\begin{lemmaN}
\label{twistfx}
$\mathcal{D}(\tau_{u/v})$ can be built inductively using $l_A$ and $l_I$ as follows. Starting with the diagram corresponding to $\tau_{0/1}$, shown below, apply top twists by bending $l_A$ towards $l_I$ and straightening it out as shown below on the right. Right twists are performed in a similar way by bending $l_I$ towards $l_A$ and then straightening.
\[
\begin{tikzpicture}
\node at (-5,0) {$\mathcal{D}(\tau_{0/1})=$};
\node at (-3,0) {\includegraphics[height=2cm]{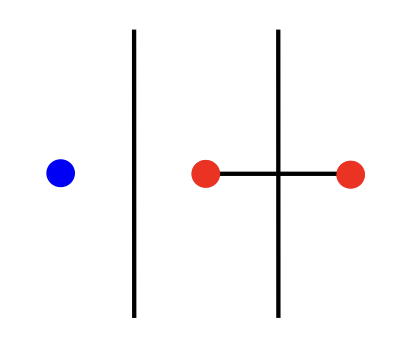}};
\node at (1,0) {$\mathcal{D}(T\tau_{1/1})=\mathcal{D}(\tau_{2/1})\colon$};
\node at (4,0) {\includegraphics[height=2cm]{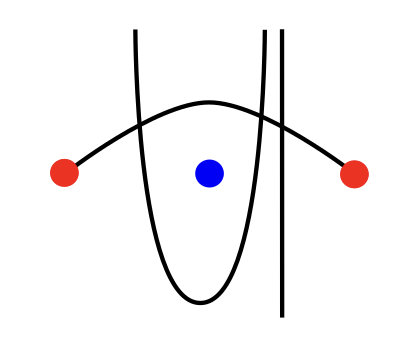}};
\node at (5.5,0) {$\longrightarrow$};
\node at (7,0) {\includegraphics[height=2cm]{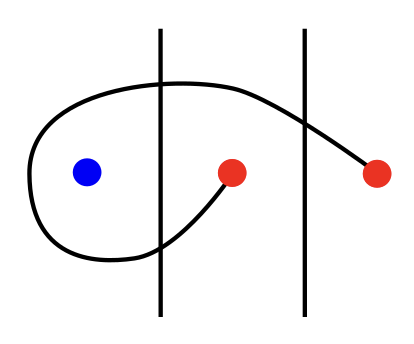}};
\end{tikzpicture}
\]
\end{lemmaN}

This is easy to see by thinking about the Dehn twists on the sphere. Additionally, we will need to keep track of the three punctures for our calculations later. We will follow the conventions of \cite{W16} and call them $X^+,X^-,$ and $Y$, where $Y$ is always the point that is not an endpoint of the arc (so it is the blue point above), and $X^+$ and $X^-$ (the red points above) are determined by the orientation of the tangle. In the case of $\tau_{0/1}$, the three points, labeled from left to right, are $Y, X^-,$ and $X^+$, which we encode by writing $Y|X^-|X^+$. We can keep track of which point is which by considering how the top and right twists permute them. The following lemma (Corollary 4.3 in \cite{W16}) tells us how the twists affect the orientation of $\tau_{u/v}$ and the configurations of these three points. In the lemma, we refer to this pair of data as a \textit{state}.

\begin{lemmaN}
\label{twistdiagram}
    The following diagram shows the effects of top and right twists on the orientations of the arcs and the configurations of the points $X^-,X^+$, and $Y$ (read left to right). The underlined state represents the state of the trivial tangle $\tau_{0/1}$.
  
\[\begin{tikzcd}[sep=small]
   & \underline{(UP,\,\,\,\, Y\vert X^-\vert X^+)} \arrow[rr,"T"]\arrow[ld] && (UP,\,\,\,\, X^-\vert Y\vert X^+) \arrow[ll]\arrow[rd,"R"] &\\
   (OP,\,\,\,\, Y\vert X^+\vert X^-) \arrow[ru,"R"]\arrow[rd]&&&& (OP,\,\,\,\,X^-\vert X^+\vert Y) \arrow[lu]\arrow[ld,"T"]\\
    & (RI,\,\,\,\, X^+\vert Y\vert X^-) \arrow[lu,"T"]\arrow[rr] && (RI,\,\,\,\, X^+\vert X^-\vert Y) \arrow[ll,"T"]\arrow[ru] &
\end{tikzcd}\]
\end{lemmaN}

Continuing with our examples of $\tau_{4/3}$ and $\tau_{3/4}$, one can easily check that $\tau_{4/3}=TR^2T\tau_{0/1}$, so $\tau_{4/3}$ is given by the state $(UP, Y|X^-|X^+)$, and $\tau_{3/4}=RT^3\tau_{0/1}$, so $\tau_{3/4}$ is given by the state $(OP, X^-|X^+|Y)$.

\subsection{$\langle-\rangle_j$ and $\langle\langle-\rangle\rangle_j$ for rational tangles}
\label{SMEsec}




In \cite{W16}, Wedrich studied colored $\mathfrak{sl}_N$ invariants for rational tangles via diagrams related to $\mathcal{D}(\tau_{u/v})$, which we will also need for the next section. Such invariants are given diagramatically in terms of webs, or oriented planar graphs with integer flow, where all internal vertices are trivalent. Such webs encode the information of linear maps between tensor products of the miniscule representations of $\mathfrak{sl}_N$ (see \cite{CKM14}). 

Given $\tau_{u/v}$ with orientation $X\in\{UP,OP,RI\}$, its $j$-colored HOMFLY-PT polynomial, written $\langle\tau_{u/v}\rangle_j$, can be thought of as the $j$-colored $\mathfrak{sl}_N$ invariant with $N\gg 0$ and $a=q^N$. It is a linear combination over $\mathbb{C}[q^{\pm 1},a^{\pm 1}]$ of webs of the form $X[j,k]$, for $0\leq k\leq j$, where the webs $X[j,k]$ are shown diagrammatically in Figure \ref{basiswebs} for each possible $X$. We adopt the convention that $0$-labeled edges are deleted.

\begin{figure}
    \raisebox{0pt}{\includegraphics[height=2.5cm, angle=0]{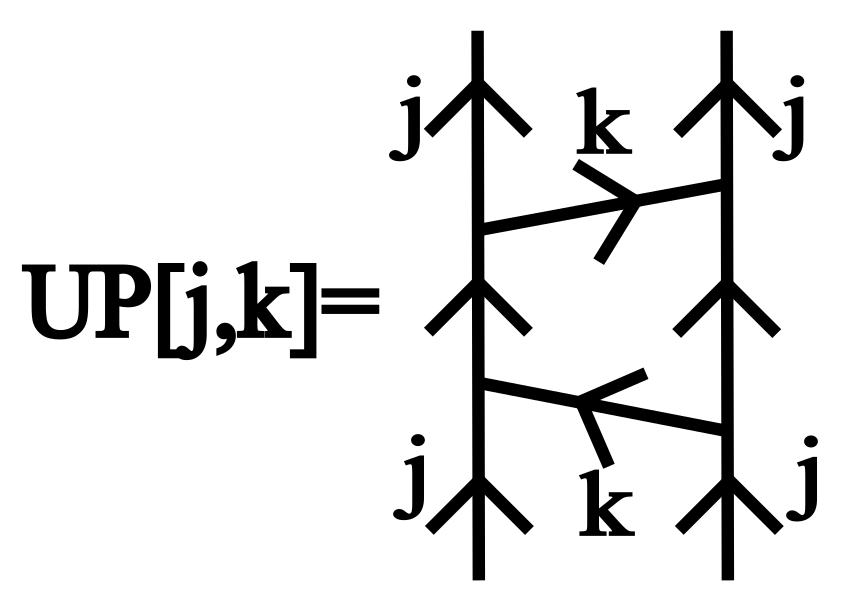}}\qquad \qquad \raisebox{0pt}{\includegraphics[height=2.5cm, angle=0]{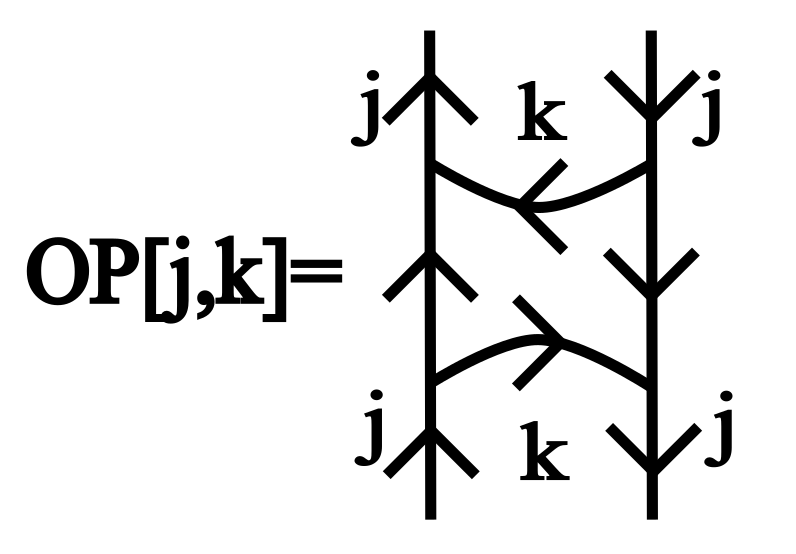}} \qquad \qquad \raisebox{0pt}{\includegraphics[height=2.5cm, angle=0]{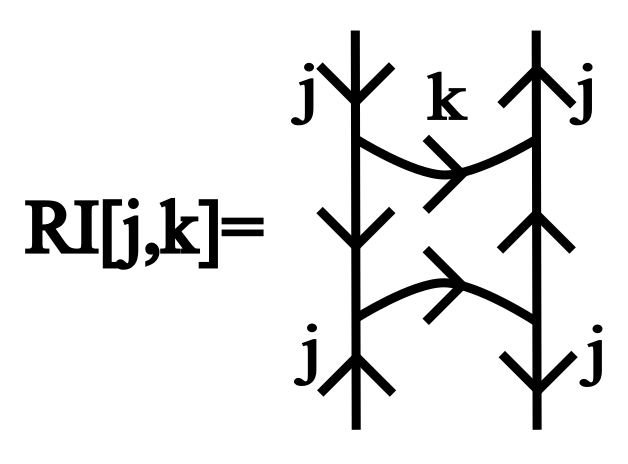}}
    \caption{The three types of basis webs}
    \label{basiswebs}
\end{figure}

As discussed in Section \ref{introsec}, $\langle -\rangle_j$ is related to the Poincar\'e polynomial $\langle\langle-\rangle\rangle_j$ of the $j$-colored HOMFLY-PT complex by
\[
\langle \tau_{u/v}\rangle_j=\langle\langle \tau_{u/v}\rangle\rangle_j\bigg|_{t\mapsto-1}
\]
for any $\tau_{u/v}$. In \cite{W16}, Wedrich used quantum skew Howe duality to prove how one can compute $\langle\langle\tau_{u/v}\rangle\rangle_j$ inductively. In particular, one starts with
\[
\begin{tikzpicture}
    \node at (0,0) {$\langle\langle\tau_{0/1}\rangle\rangle_j = UP[j,0]=$};
    \node at (2.5,0) {\includegraphics[height=1.5cm]{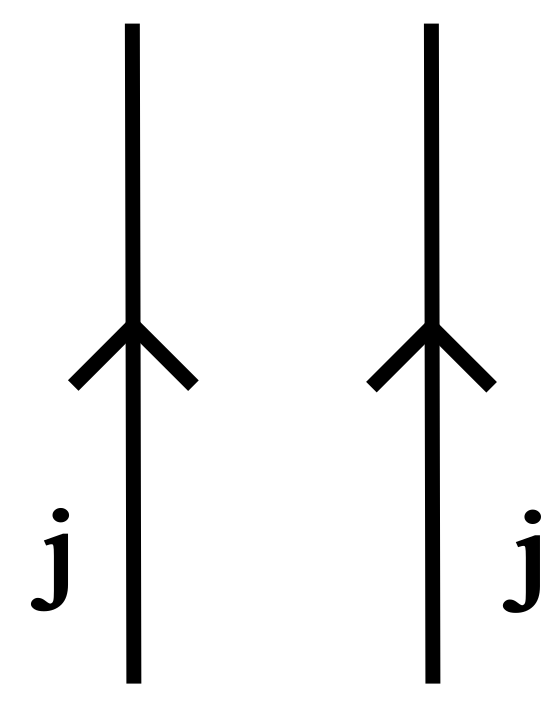}};   
\end{tikzpicture}
\]
and then one applies the same sequence of top and right twists used to build $\tau_{u/v}$ from $\tau_{0/1}$. One resolves crossings by the formula
\[
\begin{tikzpicture}
\node at (-2.5,0) {\includegraphics[height=1.5cm, angle=0]{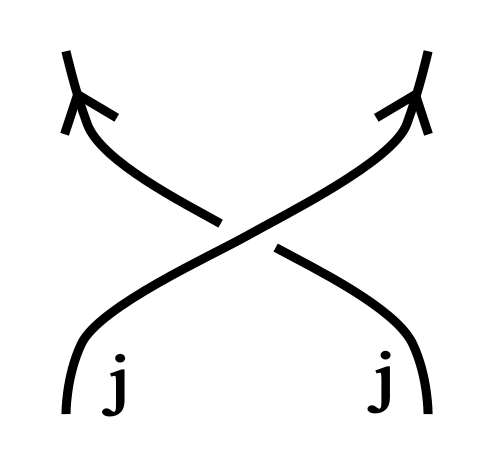}};
\node at (0,0) {\scalebox{1.25}{$=\sum_{k=0}^jt^{-k}q^k$}};
\node at (2.2,0) {\includegraphics[height=1.5cm, angle=0]{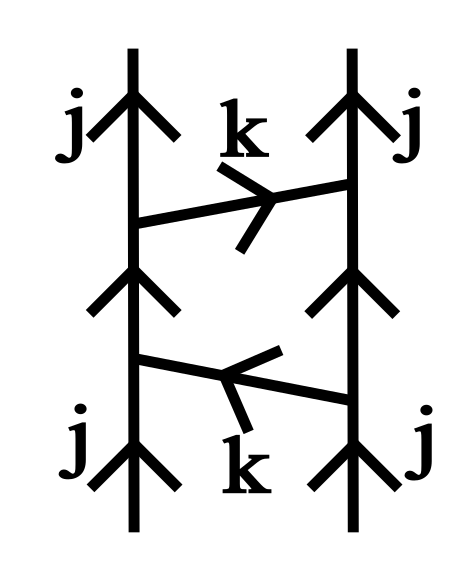}};
\node at (7,0) {$(q\leftrightarrow q^{-1} \,\, \text{if negative crossing}),$};
\end{tikzpicture}
\]
which translates to the following rules for applying $T$ and $R$ to $X[j,k]$ for each $X\in\{UP,OP,RI\}$, as proved in \cite{W16}:

\begin{enumerate}
  \item $\text{TUP}[j,k]\cong \sum_{h=k}^j t^{-h}q^{k^2+h}{h\brack k}_+\text{UP}[j,h]$ \label{TUPrule}
\item  $\text{TOP}[j,k]\cong \sum_{h=k}^j t^{-h}a^kq^{k(k-2j)+h}{h\brack k}_+\text{RI}[j,h]$
\item $\text{TRI}[j,k]\cong \sum_{h=k}^j t^{-h}a^hq^{q^2+h(1-2j)}{h\brack k}_+\text{OP}[j,h]$
\item $\text{RUP}[j,k]\cong \sum_{h=0}^k t^{-h}a^hq^{k(2j-k)+h(1-2j)}{j-h\brack k-h}^-\text{OP}[j,h]$
\item $\text{ROP}[j,k]\cong \sum_{h=0}^k t^{-h}a^kq^{-k^2+h}{j-h\brack k-h}^-\text{UP}[j,h]$
\item $\text{RRI}[j,k]\cong \sum_{h=0}^k t^{-h}q^{-k(k-2j)+h}{j-h\brack k-h}^-\text{RI}[j,h].$
\end{enumerate}
The negative superscript for some of the quantum binomials represents a rescaled version. In particular,
\[
{a\brack b}^-={a\brack b}_+ \bigg|_{q\mapsto q^{-1}}.
\]
It follows that the $X[j,k]$ form a sort of basis for the types of webs that appear in $\langle \tau_{u/v} \rangle_j$ and $\langle\langle \tau_{u/v}\rangle\rangle_j$ for any $\tau_{u/v}$. Furthermore, if we wanted to keep everything diagrammatic, one could write the first equation above as
\begin{equation*}
\begin{tikzpicture}
    \node at (-1.5,0) {\includegraphics[height=2cm]{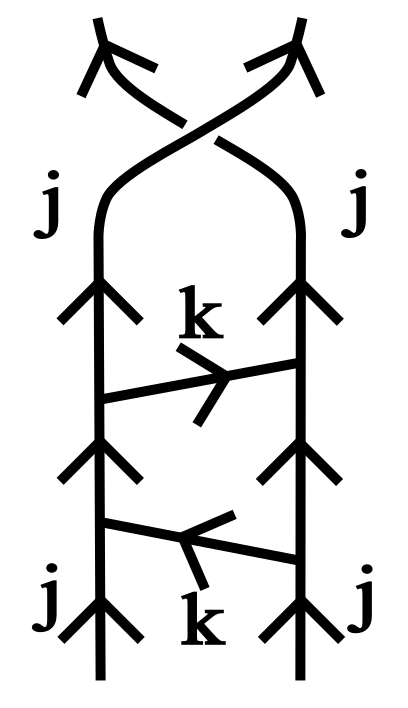}};
    \node at (1,0) {$=\sum_{h=k}^j t^{-h}q^{k^2+h}{h\brack k}_+$};
    \node at (3.5,0) {\includegraphics[height=1.5cm]{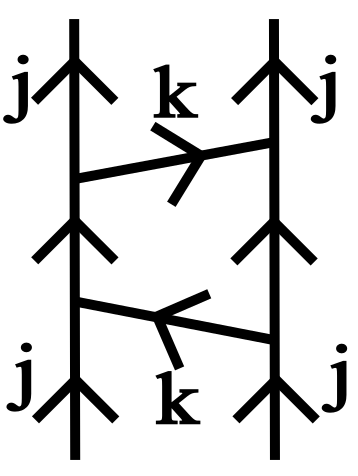}};
\end{tikzpicture}
\end{equation*}
and one could, similarly, write diagrammatic forms of each of the equations above. 

When computing $\langle\langle \tau_{u/v}\rangle\rangle_j$, it should be noted that the twist rules above need to be applied to each web in the Poincar\'e polynomials of each intermediate tangle during the inductive process. Finally, we will revisit the example of $\langle\langle\tau_{3/1}\rangle\rangle_1$ from Section \ref{introsec} to demonstrate how to apply the twist rules.

\begin{example}
  Recall that we had written
  \[
\begin{tikzpicture}
    \node at (0,0) {$\langle\langle \tau_{3/1}\rangle\rangle_1=(q^5t^{-3}+q^{3}t^{-2}+qt^{-1})$};
    \node at (3.3,0) {\includegraphics[height=1cm]{web1.png}};
    \node at (4,0) {$+$};
    \node at (4.5,0) {\includegraphics[height=1cm]{web0.png}};
\end{tikzpicture}
\]
in Section \ref{introsec}, which can be expressed in terms of the $X[j,k]$ as
\[
\langle\langle \tau_{3/1}\rangle\rangle_1=(q^5t^{-3}+q^{3}t^{-2}+qt^{-1})UP[1,1]+UP[1,0].
\]
We will now show this following the inductive procedure with the twist rules. Note that $\tau_{3/1}=T^3\tau_{0/1}$, so we need to apply $T$ to $UP[1,1]$ and $UP[1,0]$ at each of the three steps. By applying the rule for $TUP[j,k],$ we get
\begin{align*}
    &\langle\langle\tau_{0/1}\rangle\rangle_1=UP[1,0]\\
    &\langle\langle\tau_{1/1}\rangle\rangle_1=\langle\langle T\tau_{0/1}\rangle\rangle_1=qt^{-1}\,UP[1,1]+UP[1,0]\\
    &\langle\langle\tau_{2/1}\rangle\rangle_1=\langle\langle T\tau_{1/1}\rangle\rangle_1=q^3t^{-2}\,UP[1,1]+qt^{-1}\,UP[1,1]+UP[1,0]\\
    &\qquad\,\,\,\,\,\,\,\,\,\,\,= (q^3t^{-2}+qt^{-1})\,UP[1,1]+UP[1,0]\\
    & \langle\langle\tau_{3/1}\rangle\rangle_1=\langle\langle T\tau_{2/1}\rangle\rangle_1=(q^5t^{-3}+q^3t^{-2})\,UP[1,1]+qt^{-1}\,UP[1,1]+UP[1,0]\\
    &\qquad\,\,\,\,\,\,\,\,\,\,\,= (q^5t^{-3}+q^{3}t^{-2}+qt^{-1})UP[1,1]+UP[1,0],   
\end{align*}
as expected.
\end{example}


\subsection{Prior Results for Computing \(\langle\langle\tau_{u/v}\rangle\rangle_j\)}
\label{compmethods}

In \cite{W16}, Wedrich provides an elegant geometric method for computing $\langle\langle-\rangle\rangle_j$ for any rational tangle $\tau_{u/v}$, up to some universal shift of gradings, by calculating the grading differences between any pair of generators. Given $\tau_{u/v}$ and a coloring $j$, we can calculate $\langle\langle\tau_{u/v}\rangle\rangle_j$ using variations of what we have called $\mathcal{D}(\tau_{u/v})$. Recall that $\alpha_{u/v}$ comes with two special vertical Lagrangians: $l_A$ and $l_I$. Furthermore, the objects of the chain complex are of the form $X[j,k]$, where $k$ is called the \textit{weight}. Now we relate these two ideas.

\begin{definitionN}
    The generators of weight $k$ in the $j$-colored HOMFLY-PT complex for $\tau_{u/v}$ are defined in terms of weighted and unweighted verticals. We get the \textit{weighted verticals} in the colored HOMFLY-PT complex by taking $k$ parallel copies of $l_A$, which we denote by $w_1,...,w_k$, and the \textit{unweighted verticals} are $j-k$ parallel copies of $l_I$, denoted $u_1,...,u_{j-k}$. 
\end{definitionN}

We can use this new geometric setup to define our ``variation'' of $\mathcal{D}(\tau_{u/v})$.

\begin{definitionN}
\label{Dkjmkdef}
   Let $\mathcal{D}^{k,j-k}(\tau_{u/v})$ be the picture in the 3-punctured plane consisting of $\alpha_{u/v},w_1,...,w_k,$ and $u_1,...,u_{j-k}$. 
\end{definitionN}

In \cite{W16}, Wedrich proves that the generators of weight $k$ in the $j$-colored HOMFLY-PT complex are given by the $j$-tuples of intersection points 
\[
(x_1,...,x_j)\in (w_1\cap \alpha)\times...\times(w_k\cap \alpha)\times(u_1\cap \alpha)\times...\times(u_{j-k}\cap \alpha),
\]
 in $\mathcal{D}^{k,j-k}(\tau_{u/v})$, where we have dropped the $u/v$ subscript for $\alpha$, and then he describes how to compute the grading differences of the generators corresponding to these intersection points. If $\overline{x}=(x_1,...,x_h,...,x_j)$ and $\overline{x}'=(x_1,...,x_h',...,x_j)$ give two generators of the same weight that differ only at the $h$th coordinate, then one can compute the grading difference between $\overline{x}$ and $\overline{x}'$ as follows:

Take the path along $\alpha$ from $x_h$ to $x_h'$, denoted by $\alpha'$, and then the path $\beta$ from $x_h'$ back down to $x_h$ along the $h$th vertical. Consider now the loop $\alpha'\cdot \beta$ based at $x_h$. The enclosed region by this loop, called $D$, is a singular 2-chain over $\mathbb{Z}$, and the grading differences between $\overline{x}$ and $\overline{x}'$ are computed by studying intersections with this $D$. There are two types of contributions to the grading difference, known as the additive and non-additive parts. 

The \textit{additive part} is given by how $D$ intersects the points $\{X^+,X^-,Y\}$. The contributions to the grading difference arising from the intersection of these points with a simple disk $D_s$ coming from $D$ is given by
\[
D_s\cdot X^+=\frac{a^2}{tq^{4j-2}}, \qquad D_s \cdot X^-= D_s \cdot Y=\frac{q^2}{t}.
\]

The \textit{non-additive part} is given by how $D$ intersects the points in $\{x_1,...,x_{h-1},x_{h+1},...,x_j\}$ that $\overline{x}$ and $\overline{x}'$ have in common. The contribution from a simple disk $D_s$ coming from $D$ intersecting an $x_i$ from this set is 
\[
D_s \cdot x_i = \begin{cases}
    q^4 &  \text{if} \,x_i\in \text{int}(D_s),\\
    q^2 & \text{if} \, x_i \in \partial D_s, \\
    0   & \text{otherwise}.
\end{cases}
\]

Now, we have a way to determine the grading differences between any two generators of the colored HOMFLY-PT complex with weight $k$ (which, thus correspond to objects $X[j,k]$ for $X$ one of $UP, OP,$ or $RI$). Thus, we just need a way to compare generators of different weights. 

Suppose that $\overline{x}$ is a generator of weight $k$ and $\overline{y}$ is a generator of weight $k-1$ such that $\overline{x}$ and $\overline{y}$ are the same outside of the $k$th coordinate, and the $k$th coordinate of $\overline{y}$ is obtained from the $k$th coordinate of $\overline{x}$ by sliding the $k$th vertical from the weighted to the unweighted side. Then, we say that $\overline{x}$ and $\overline{y}$ are related by a \textit{simple slide}, and the grading difference between the generators is determined by which of the three special points from $\{X^+,X^-,Y\}$ the $k$th weighted vertical intersects while sliding to the unweighted side. We will use $Z$ here and throughout the rest of the paper to denote this point in between the other two special points. Then, the grading difference between $\overline{x}$ and $\overline{y}$ is given by

\[\begin{cases}
    t/q & \text{if}\, Z\in\{Y,X^-\},\\
    tq^{2j-1}/a &\text{if} \,Z= X^+.
\end{cases}\]

These rules allow us to determine the grading differences between any two generators in the colored $\mathfrak{sl}_N$ complex with $a=q^N$, which is sufficient to determine the Poincar\'e polynomial, up to some universal scaling of the gradings. 

Wedrich also gives another way one can think of these generators and computing their grading differences, to compare with the constructions of Bigelow and Manolescu found in \cite{B02} and \cite{M04}, respectively. We will now provide a description of this other method.

Consider the arc $\alpha_{u/v}$ as a 1-dimensional submanifold of $\mathbb{C}\setminus\{X^+,X^-,Y\}$. We will now set $M=\mathbb{C}\setminus\{X^+,X^-,Y\}$ to simplify notation and use $\text{Conf}^j(M)=\text{Sym}^j(M)\setminus \Delta$, where $\Delta$ is the big diagonal, to denote the $j$th (unordered) configuration space of $M$. Now, we define
\[
\mathbb{L}_{u/v}^j:=\text{Sym}^j(\alpha_{u/v})\setminus \Delta\subset\text{Conf}^j(M),
\]
and 
\[
\mathbb{L}_k^j := \pi(w_1\times...\times w_k\times u_1\times ...\times u_{j-k})\subset \text{Conf}^j(M)
\]
as two $j$-dimensional (over $\mathbb{R}$) submanifolds of the $2j$-real-dimensional $\text{Sym}^j(M)$, where we think of each of the parallel verticals shown in $\mathbb{L}_k^j$ as living in their own copies of $\mathbb{C}$ and 
\[
\pi:M^j\setminus\Delta\to \text{Conf}^j(M).
\]
Now, the generator $\overline{x}=(x_1,...,x_j)$ of weight $k$, using the notation from earlier in this section, can be thought of as the appropriate point in $\mathbb{L}_{u/v}^j \cap \mathbb{L}_k^j$. In fact, the whole set of generators for the $j$-colored $\mathfrak{sl}_N$ complex of $\tau_{u/v}$, which we denote by $\mathcal{G}^j_{u/v}$, is given by 
\[
\mathcal{G}^j_{u/v}=\bigcup_{k=0}^j (\mathbb{L}_{u/v}^j \cap \mathbb{L}_k^j).
\]

The grading differences between the points in $\mathcal{G}^j_{u/v}$ 
of the same weight $k$ can be computed by the graded intersection numbers of lifts of $\mathbb{L}_{u/v}^j$ and $\mathbb{L}_k^j$ to the covering space of $\text{Conf}^j(M)$ specified by a surjective homomorphism $\Theta_j:\pi_1(\text{Conf}^j(M))\to\mathbb{Z}^3$
 (the $\mathbb{Z}^3$ corresponds to the $q$-, $a$-, and $t$-gradings). We can describe this homomorphism explicitly with the help of some simpler ones. For any $V\in \{X^+,X^-,Y\}$, we can define a homomorphism
\[
\Psi_V^j:\pi_1(\text{Conf}^j(M))\xrightarrow{\iota_*} \pi_1(\text{Conf}^j(\mathbb{C}\setminus V)) \to \pi_1(\text{Conf}^{j+1}(\mathbb{C}))\cong\text{Br}_{j+1}\xrightarrow{\text{ab}} \mathbb{Z},
\]
where the first map is an inclusion, the second map comes from adding the points in $V$ to the unordered $j$-tuple, and the final map is the standard abelianization homomorphism. Additionally, we can define
\[
\Phi^j:\pi_1(\text{Conf}^j(M))\xrightarrow{\iota_*} \pi_1(\text{Conf}^j(\mathbb{C}))\cong\text{Br}_j\xrightarrow{\text{ab}} \mathbb{Z}.
\]
Our convention is that we compute with these homomorphisms by first determining the appropriate braid, which is drawn by viewing the loop from above the diagram, and then computing the abelianization---the number of positive crossings minus the number of negative crossings. The convention for obtaining the braid associated with a loop will be clarified shortly with an example.

Now, for $\overline{x}, \overline{y} \in \mathbb{L}_{u/v}^j \cap \mathbb{L}_k^j$ (two generators of the same weight), the $q$-, $a$-, and $t$-grading differences between $\overline{x}$ and $\overline{y}$ are given by $\Theta_j([\gamma_{\overline{x},\overline{y}}])\in \mathbb{Z}^3$, where $\gamma_{\overline{x},\overline{y}}$ is a loop in $\text{Conf}^j(M)$ that starts at $\overline{x}$, moves to $\overline{y}$ along $\mathbb{L}_{u/v}^j$, and then back down to $\overline{x}$ via $\mathbb{L}_k^j$. More precisely, if $\overline{x}=(x_1,...,x_j)$ and $\overline{y}=(y_1,...,y_j)$, then $\gamma_{\overline{x},\overline{y}}$ is given by a $j$-tuple of paths going from $x_i$ to $y_{\sigma(i)}$ along $\mathbb{L}_{u/v}^j$ for some $\sigma \in S_j$, then from $y_{\sigma(i)}$ to $x_{\sigma(i)}$ along the $\sigma(i)$th vertical in $\mathbb{L}_k^j.$ 

\begin{definitionN}
    Given two generators $\overline{x}$ and $\overline{y}$ of the $j$-colored HOMFLY-PT complex for a rational tangle $\tau_{u/v}$, let $q(\overline{x})-q(\overline{y})$ denote the difference of $q$-gradings between the two generators. Similarly, we define $a(\overline{x})-a(\overline{y})$ and $t(\overline{x})-t(\overline{y})$ for the $a$- and $t$-grading differences.
\end{definitionN}

Now, we formulate precisely what these grading differences are in terms of the homomorphisms just given, which we get by comparing with Wedrich's method of intersection numbers and 2-chains. Putting everything together, we see
\begin{align}
\label{qdiffeqn}
    &q(\overline{x})-q(\overline{y})=\left(2\Phi^j+(1-2j)\left(\Psi_{X^+}^j-\Phi^j\right)+\sum_{W\in\{X^-,Y\}}\left(\Psi_W^j-\Phi^j\right)\right)([\gamma_{\overline{x},\overline{y}}])\\
    \label{adiffeqn}
    &a(\overline{x})-a(\overline{y})=\left(\Psi_{X^+}^j-\Phi^j\right)([\gamma_{\overline{x},\overline{y}}])\\
    \label{tdiffeqn}
    &t(\overline{x})-t(\overline{y})=-\frac{1}{2}\sum_{W\in\{X^+,X^-,Y\}}\left(\Psi_W^j-\Phi^j\right)([\gamma_{\overline{x},\overline{y}}]).
\end{align}

This all follows from the fact that $2\Phi^j$ gives the non-additive part and, for $V\in\{X^+,X^-,Y\}$, $\frac{1}{2}\left(\Psi_V^j-\Phi^j\right)$ computes the winding number about $V$. These three homomorphisms on the right give the coordinate homomorphisms of $\Theta_j$.

\begin{example}
    Now, we give an example for comparing the grading differences between two generators $\overline{x},\overline{y}\in \mathbb{L}_{3/1}^2 \cap \mathbb{L}_2^2$ of weight $2$ for $\tau_{3/1}$. Note that $\tau_{3/1}$ has state $(UP, X^-|Y|X^+)$. For our generators, we will take
    \[
    \begin{tikzpicture}
        \node at (0,0) {$\overline x=$};
        \node at (2.5,0) {\includegraphics[height=4cm]{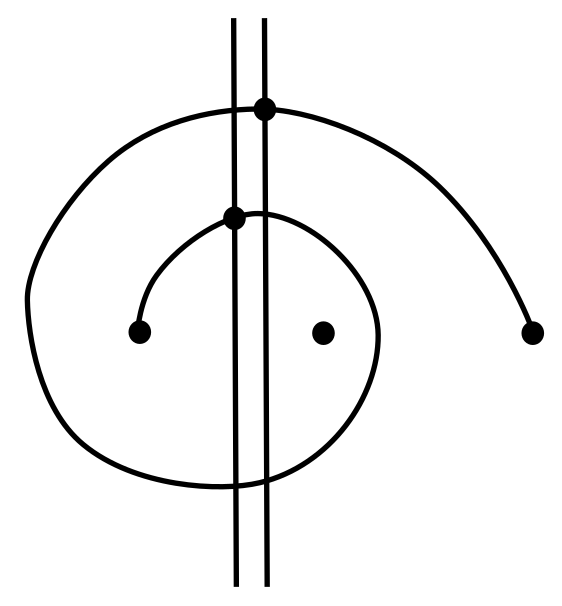}};
        \node at (7,0) {$\overline y=$};
        \node at (9.5,0) {\includegraphics[height=4cm]{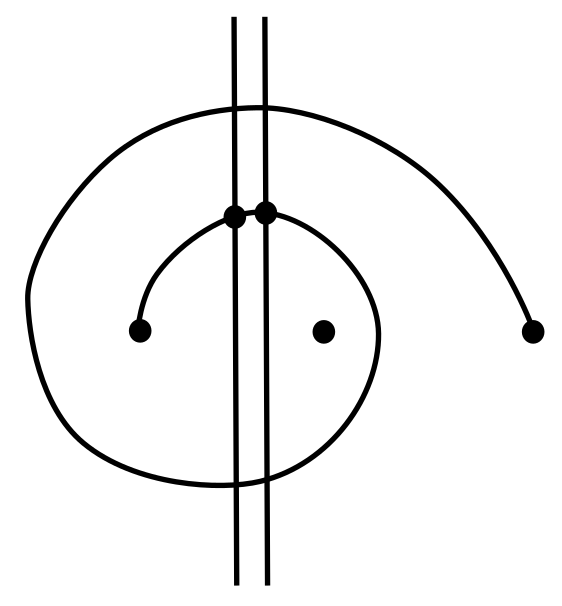}};
    \end{tikzpicture}
    \]
    where we can think of these pictures as living in $\text{Conf}^2(M)$. To compute the grading differences, we use the loop $\gamma_{\overline x,\overline y}$ that goes from $\overline x$ to $\overline y$ along $\mathbb{L}_{3/1}^2$ and then back to $\overline x$ along $\mathbb{L}_2^2$. The loop $\gamma_{\overline x, \overline y}$ is shown below with the eye symbol representing the perspective we use to draw the corresponding braids. In particular, the braid is obtained by viewing the loop from above and in the same plane (imagine tilting the page so that the top is near your face and the picture is facing down), so that when a path moves to the left/right on the page, it also moves to the left/right from this perspective. Recall that loops in $\text{Conf}^2(M)$ are given by two disjoint loops or two paths that swap starting/ending points. Below, the red point denotes a constant loop.
    \[
    \begin{tikzpicture}
        \node at (0,0) {$\gamma_{\overline x,\overline y}=$};
        \node at (4,0) {\includegraphics[height=5cm]{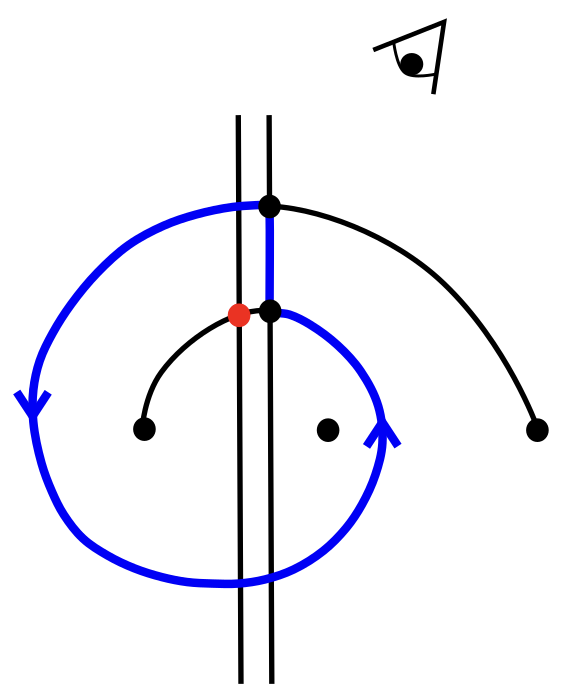}};
    \end{tikzpicture}
    \]
    Thus, for example, viewing the loop from the specified perspective, we see that the blue loop first goes to the left and in front of $X^-$, and then it goes behind $X^-$ and $Y$, moving to the right, and then it turns left again, passing in front of $Y$ and returning to the basepoint. Below are the braids we get from $\gamma_{\overline x,\overline y}$ for computing $\Phi^2, \Psi_{X^+}^2,\Psi_{X^-},$ and $\Psi_Y^2$, from left to right. We have colored in blue and red the strands corresponding to the colored component loops for $\gamma_{\overline x,\overline y}$, and the black strands represent the appropriate punctures.
    \[
    \begin{tikzpicture}
        \node at (1,0) {\includegraphics[height=3cm]{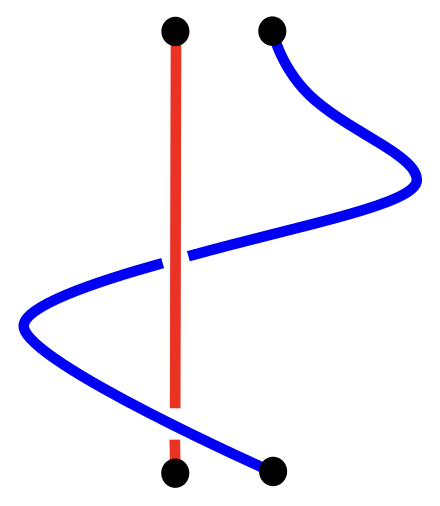}};
        \node at (4,0) {\includegraphics[height=3cm]{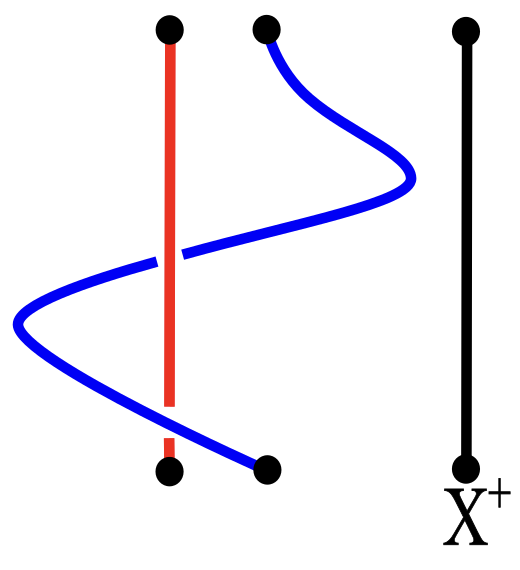}};
        \node at (8,0) {\includegraphics[height=3cm]{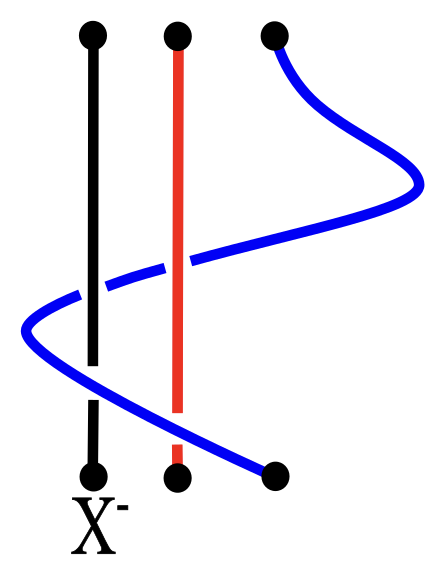}};
        \node at (11,0) {\includegraphics[height=3cm]{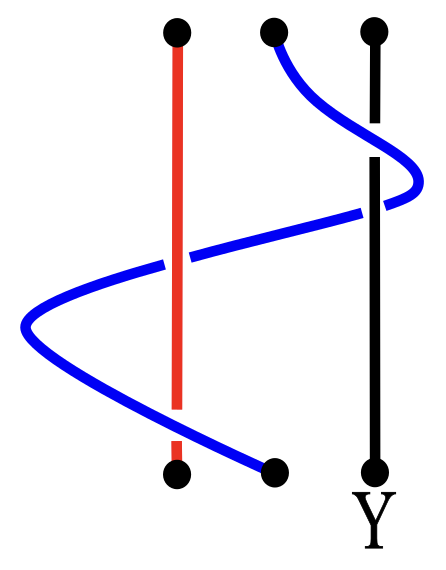}};
    \end{tikzpicture}
    \]
    From here, we see $\Phi^2([\gamma_{\overline x,\overline y}])=-2, \Psi_{X^+}^2([\gamma_{\overline x,\overline y}])=-2, \Psi_{X^-}^2([\gamma_{\overline x,\overline y}])=-4$, and $\Psi_Y^2([\gamma_{\overline x,\overline y}])=-4$. Now, we can use Equations (\ref{qdiffeqn})-(\ref{tdiffeqn}) to compute $q(\overline x)-q(\overline y)=-8$, $a(\overline x)-a(\overline y)=0$, and $t(\overline x)-t(\overline y)=2$.
\end{example}


\subsection{Links-Quivers for Rational Tangles}
\label{quiversec}

In Section \ref{introsec}, we briefly discussed how our tangles analogue of the Links-Quivers Correspondence was that we can write $\mathcal{P}(\tau_{u/v})=\sum_{j\geq 0}\langle\langle\tau_{u/v}\rangle\rangle_j$ in ``quiver form,'' which we precisely define now.

\begin{definitionN}
    For a rational tangle $\tau_{u/v}$, $\mathcal{P}(\tau_{u/v})$ is in \textit{quiver form} if it can be written as
    \begin{multline*}
    \mathcal{P}(\tau_{u/v})=
        \sum_{\bf{d}=(d_1,...,d_{u+v})\in \mathbb{N}^{u+v}}q^{S\cdot\textbf{d}+\textbf{d}\cdot Q\cdot \textbf{d}^t}a^{A\cdot \textbf{d}} t^{T\cdot\textbf{d}} {d_1+...+d_u\brack d_1,...,d_u}{d_{u+1}+...+d_{u+v}\brack d_{u+1},...,d_{u+v}}\\ \times X[d_1+...+d_{u+v},d_1+...+d_u],
\end{multline*}
where $S,A,$ and $T$ are linear forms and $Q$ is a quadratic form on a free $\mathbb{Z}$-module, denoted $\mathbb{Z}\mathcal{G}_{u/v}^1$. 
\end{definitionN}
Also, recall from Section \ref{introsec} that $\mathbb{Z}\mathcal{G}_{u/v}^1$ is equipped with a preferred basis $\mathcal{G}_{u/v}^1$ given by the Lagrangian intersections in $\mathcal{D}(\tau_{u/v})$.

It will be helpful to establish some notation regarding these linear and quadratic forms.
We see that the data for $\mathcal{P}(\tau_{u/v})$ is given by the orientation $X$ of $\tau_{u/v}$, the three linear forms $S,A,$ and $T$, which can be written as vectors, and the quadratic form $Q$, which can be written as a symmetric matrix. Following the conventions of \cite{SW21}, we will sometimes express this data by 
\[
X, [S|A|T], Q,
\]
where the three vectors are put as the columns of a matrix. Then, whenever we need to partition the intersection points (or, equivalently, the basis elements of $\mathbb{Z}\mathcal{G}_{u/v}^1$), one can use block representations such as 
\[
X, 
    \left[\begin{array}{ c| c | c  }
    S_1 & A_1 & T_1\\
    \hline
    S_2 & A_2 & T_2
  \end{array}\right],  \left[\begin{array}{ c |c }
    
    Q_{11} & Q_{12} \\
    \hline
    Q_{21} & Q_{22}
  \end{array}\right]
\]
to represent the data, where the subscript 1 corresponds to the block for one of the partitioning sets, and the 2 corresponds to the other. Thus, the blocks $Q_{12}$ and $Q_{21}$ can be thought of as the part of $Q$ that relates basis elements between the two sets. For example, one can use the partitioning of the indices by the active and inactive ones to define blocks such as the ones above. 

In Section \ref{indreducesec}, we will present a way to express $\mathcal{P}(\tau_{u/v})$ in a different form involving fewer indices in the sum, which corresponds to restricting $S,A,T,$ and $Q$ to a free submodule of $\mathbb{Z}\mathcal{G}_{u/v}^1$ with basis given by a subset of $\mathcal{G}_{u/v}^1$. The purpose for this will be apparent in \cite{JHpII26}, where we will need to use generating functions with compressed indices. As a consequence of reducing the number of indices, however, we will need one additional linear form $K$, which consists only of $1$'s and $0$'s. This modified way of writing $\mathcal{P}(\tau_{u/v})$ is analogous to what is referred to as the ``almost quiver form'' by Sto\v si\'c and Wedrich in \cite{SW21}, so we adopt this terminology. 
\begin{definitionN}
    Given a rational tangle $\tau_{u/v}$, we say that $\mathcal{P}(\tau_{u/v})$ is in \textit{almost quiver form} if it may be written as
    \begin{multline*}
\sum_{\bf{d}=(d_1,...,d_{m+n})\in \mathbb{N}^{m+n}}q^{S\cdot\textbf{d}+\textbf{d}\cdot Q\cdot \textbf{d}^t}a^{A\cdot \textbf{d}} t^{T\cdot\textbf{d}} (-t^{-1}q^2;q^2)_{K\cdot \textbf{d}}{d_1+...+d_m\brack d_1,...,d_m}{d_{m+1}+...+d_{m+n}\brack d_{m+1},...,d_{m+n}}\\ \times X[d_1+...+d_{m+n},d_1+...+d_m],
\end{multline*}
where $m\leq u$, $n\leq v$, $S,A,T,$ and $Q$ are restrictions of the same linear and quadratic forms from the quiver form, and $K$ is the new linear form.
\end{definitionN}

The data for $\mathcal{P}(\tau_{u/v})$, written in almost quiver form, may be given by 
\[
X, [K|S|A|T], Q,
\]
which we will use in Section \ref{indreducesec}.

In \cite{JHpII26}, we will need the generating functions $P(\tau_{u/v})=\sum_{j\geq 0}\langle\tau_{u/v}\rangle_j$ for the colored HOMFLY-PT polynomials (or skein module evaluations) of $\tau_{u/v}$. Consequently, we will present the important results for $P(\tau_{u/v})$ when appropriate, but they will follow trivially from the corresponding results for $\mathcal{P}(\tau_{u/v})$ by setting $t=-1$. For now, we will just state what the quiver and almost quiver forms look like for $P(\tau_{u/v})$.

\begin{definitionN}
    Given $\tau_{u/v}$, we say that $P(\tau_{u/v})$ is in \textit{quiver form} if it is written as 
    \begin{multline*}
        \sum_{\bf{d}=(d_1,...,d_{u+v})\in \mathbb{N}^{u+v}}(-q)^{S\cdot\textbf{d}}q^{\textbf{d}\cdot Q\cdot \textbf{d}^t}a^{A\cdot \textbf{d}} {d_1+...+d_u\brack d_1,...,d_u}{d_{u+1}+...+d_{u+v}\brack d_{u+1},...,d_{u+v}}\\ \times X[d_1+...+d_{u+v},d_1+...+d_u],
\end{multline*}
where $S,A,$ and $Q$ are the same as in the definition of the quiver form for $\mathcal{P}(\tau_{u/v})$. We say $P(\tau_{u/v})$ is in \textit{almost quiver form} if it is written as
\begin{multline*}
\sum_{\bf{d}=(d_1,...,d_{m+n})\in \mathbb{N}^{m+n}}(-q)^{S\cdot\textbf{d}}q^{\textbf{d}\cdot Q\cdot \textbf{d}^t}a^{A\cdot \textbf{d}} (q^2;q^2)_{K\cdot \textbf{d}}{d_1+...+d_m\brack d_1,...,d_m}{d_{m+1}+...+d_{m+n}\brack d_{m+1},...,d_{m+n}}\\ \times X[d_1+...+d_{m+n},d_1+...+d_m],
\end{multline*}
where $m,n,K,S,A,$ and $Q$ are the same as in the definition of the almost quiver form for $\mathcal{P}(\tau_{u/v})$.
\end{definitionN}
We can express the data for $P(\tau_{u/v})$ in the same way as $\mathcal{P}(\tau_{u/v})$, but we no longer need the linear form $T$.

\section{The Main Theorem}
\label{mainthmsec}

\subsection{Setting Up the Theorem for Poincar\'e Polynomials}
\label{thmsetupsec}

In this section, we discuss Theorem \ref{introtanglethm} from Section \ref{introsec}. In terms of the terminology we have developed, the theorem tells us that $\mathcal{P}(\tau_{u/v})=\sum_{j\geq 0} \langle\langle \tau_{u/v}\rangle\rangle_j$ can be put in quiver form, meaning
\begin{multline*}
    \mathcal{P}(\tau_{u/v})=
        \sum_{\bf{d}=(d_1,...,d_{u+v})\in \mathbb{N}^{u+v}}q^{S\cdot\textbf{d}+\textbf{d}\cdot Q\cdot \textbf{d}^t}a^{A\cdot \textbf{d}} t^{T\cdot\textbf{d}} {d_1+...+d_u\brack d_1,...,d_u}{d_{u+1}+...+d_{u+v}\brack d_{u+1},...,d_{u+v}}\\ \times X[d_1+...+d_{u+v},d_1+...+d_u],
\end{multline*}
where the linear forms $S,A,T,$ and $Q$ are computed geometrically from the first and second configuration spaces of $M$, the 3-punctured plane. Before we can state the theorem precisely, we first need to establish conventions and prove some easy lemmata. Section \ref{pfsec} is dedicated to proving the theorem. 

Given a rational tangle $\tau_{u/v}$, recall that we have $\mathcal{D}(\tau_{u/v})$ consisting of an arc Lagrangian $\alpha_{u/v}$ and two vertical Lagrangians $l_A$ and $l_I$ in $M$, the 3-punctured plane $\mathbb{C}\setminus\{X^+,X^-,Y\}$. Let $\mathcal{G}_{u/v}^1$ be the $u+v$ intersection points coming from $\alpha_{u/v}\cap(l_A\cup l_I)$. We will see that $S,A,$ and $T$ can be thought of as linear forms and $Q$ as a quadratic form on the free $\mathbb{Z}$-module $\mathbb{Z}\mathcal{G}_{u/v}^1$ equipped with the preferred basis $\mathcal{G}_{u/v}^1$. 

Now, we need to establish some notation for the intersection points (or, equivalently, the basis for $\mathbb{Z}\mathcal{G}_{u/v}^1$).

\begin{definitionN}
    Let $\mathcal{G}_{u/v}^1=\{\xi_1,...,\xi_{u+v}\}$, where the points $\xi_1,...,\xi_u$ come from $\alpha_{u/v}\cap l_A$ and $\xi_{u+1},...,\xi_{u+v}$ come from $\alpha_{u/v}\cap l_I$. As a convention, we label the $\xi_i$'s with increasing subscripts as one moves ``up'' along $l_A$ or $l_I$ (i.e. in the imaginary direction). Call this the \textit{standard ordering} of the $\xi_i$'s.
\end{definitionN}

\begin{figure}
    \centering
    \includegraphics[height=5cm]{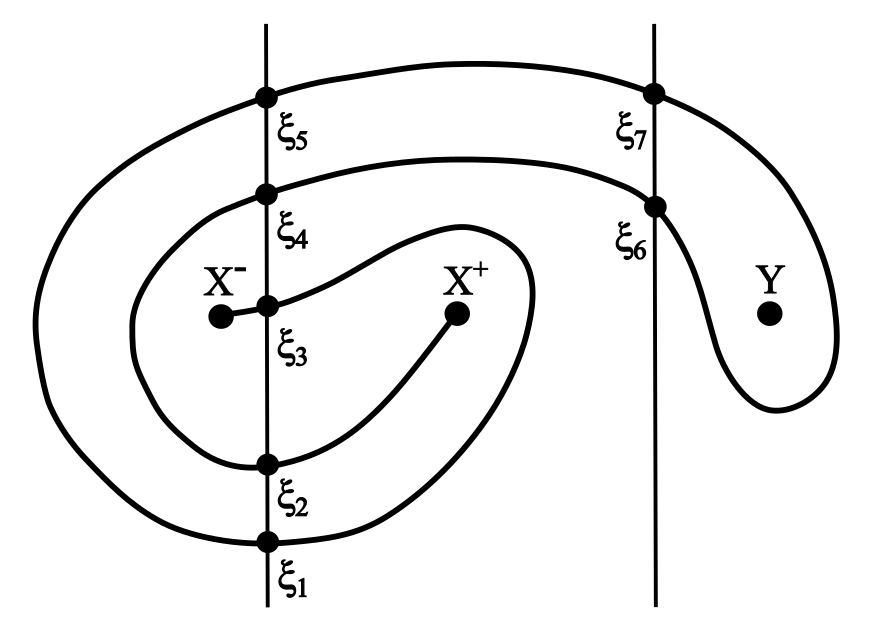}
    \caption{$\mathcal{D}(\tau_{5/2})$ drawn with the $\xi_i$'s labeled according to the standard convention.}
    \label{Dt52xilabel}
\end{figure}

In Section \ref{indreducesec}, we will use a different ordering of the intersection points, but it should be assumed that we are working with the standard ordering unless stated otherwise. We recall Definition \ref{in_activedef}, but now using our new notation.

\begin{definitionN}
    Let $\xi_1,...,\xi_u$, the intersections coming from $\alpha_{u/v}\cap l_A$, be called the \textit{active intersections}. Additionally, let $\xi_{u+1},...,\xi_{u+v}$, the intersections coming from $\alpha_{u/v}\cap l_I$, be called the \textit{inactive intersections}. 
\end{definitionN}

Recall that these names come from the fact that $l_A$ and $l_I$ are called the active and inactive axes, respectively.

Returning to the linear and quadratic forms, it turns out that we only need them on $\mathbb{N}$-linear combinations (where $0\in\mathbb{N}$) of the $\xi_i$'s. If $d_1\xi_1+...+d_{u+v}\xi_{u+v}\in \mathbb{Z}\mathcal{G}_{u/v}^1$ with all $d_i\in \mathbb{N}$, we will think of this as a point $\textbf{d}=(d_1,...,d_{u+v})\in \mathbb{N}^{u+v}$. As one might expect from the quiver form for $\mathcal{P}(\tau_{u/v})$, these $\textbf{d}$'s will correspond to the indices of the generating functions.

\begin{definitionN}
    Let $\mathbb{N}\mathcal{G}_{u/v}^1\subset \mathbb{Z}\mathcal{G}_{u/v}^1$ be the set of all $d_1\xi_1+...+d_{u+v}\xi_{u+v}\in \mathbb{Z}\mathcal{G}_{u/v}^1$ such that all $d_i\in \mathbb{N}$. 
\end{definitionN}



Since $S,A,$ and $T$ are linear forms on the free $\mathbb{Z}$-module $\mathbb{Z}\mathcal{G}_{u/v}^1$, it suffices to define them on the individual basis elements $\xi_i$. We will write $S_i=S(\xi_i)$, and similarly for $A$ and $T$. Thus, for $\textbf{d}\in \mathbb{N}^{u+v}$, when we write $S\cdot \textbf{d}$, we mean
\[
S\cdot \textbf{d}=S(\sum_{i=1}^{u+v}d_i\xi_i)=S_1d_1+...+S_{u+v}d_{u+v}.
\]
It also suffices to define $Q$ on $\xi_i\cdot\xi_j\in \text{Sym}^2(\mathbb{Z}\mathcal{G}_{u/v}^1)$ because quadratic forms on $\mathbb{Z}\mathcal{G}_{u/v}^1$ may be thought of as linear forms on $\text{Sym}^2(\mathbb{Z}\mathcal{G}_{u/v}^1)$. In particular, we have
\[
Q\left(\sum_{i=1}^{u+v} d_i\xi_i\right)=\sum_{1\leq i,j\leq u+v} Q_{ij} d_id_j= \sum_{i=1}^{u+v}Q_{ii}d_i^2+\sum_{1\leq i<j\leq u+v}2Q_{ij}d_id_j,
\]
where $Q_{ij}=Q_{ji}=Q'(\xi_i\cdot\xi_j)$, where $Q'$ is the corresponding linear form on $\text{Sym}^2(\mathbb{Z}\mathcal{G}_{u/v}^1)$. Additionally, this is what we mean by $\textbf{d}\cdot Q\cdot \textbf{d}^T$ in the quiver form.

Furthermore, we will see that all outputs for $S,A,T,$ and $Q$ are integers, so, for our purposes, we may think of them as functions $\mathbb{N}\mathcal{G}_{u/v}^1\to \mathbb{Z}$ that behave like linear and quadratic forms.

Before proceeding to define $S,A,T,$ and $Q$ in terms of homomorphisms and loops, we want a way to reduce generators of the colored HOMFLY-PT complexes to elements of $\mathbb{N}\mathcal{G}_{u/v}^1$. Recall that, in \cite{W16}, Wedrich gave a method to compare the gradings of generators in the $j-$colored HOMFLY-PT complex of a rational tangle, which were $j$-tuples
\[
\overline{x}=(x_1,...,x_j)\in (w_1\cap \alpha)\times...\times(w_k\cap \alpha)\times(u_1\cap \alpha)\times...\times(u_{j-k}\cap \alpha),
\]
where we have dropped the $u/v$ subscript for $\alpha_{u/v}$. Equivalently, these $\overline{x}$ are intersection points of the Lagrangians in $\mathcal{D}^{k,j-k}(\tau_{u/v})$. The method for computing the grading differences between generators $\overline{x}$ and $\overline{x}'$ was described in Section \ref{compmethods}.

\begin{definitionN}
    Given a generator $\overline{x}=(x_1,...,x_j)$ for the $j$-colored HOMFLY-PT complex for a rational tangle, as just described, let $c(x_i)$ be the $\xi_l$ that we get by viewing the $w_i$ or $u_i$ as $l_A$ or $l_I$ itself, respectively. In other words, if we collapse $w_1,...,w_k$ into $l_A$ and $u_1,...,u_{j-k}$ into $l_I$, then each $x_i$ collapses to a point $\xi_l$ on $l_A\cap \alpha$ or $l_I\cap \alpha$, and $c(x_i)$ is this point.
\end{definitionN}

Now, we can use this notation to define a function that will be very helpful for this section.

\begin{definitionN}
    If we let $\mathcal{G}_{u/v}^j$ be the set of all generators $\overline{x}=(x_1,...,x_j)$ of the $j$-colored HOMFLY-PT complex for $\tau_{u/v}$ (so $k$ ranges from $0$ to $j$), then we define the function $\textbf{c}:\bigcup_{j\geq 0}\mathcal{G}_{u/v}^j\to \mathbb{N}\mathcal{G}_{u/v}^1$, where
\[
\textbf{c}(\overline{x})=\sum_{i=1}^j c(x_i)=\sum_{l=1}^{u+v} d_l \xi_l \in \mathbb{N}\mathcal{G}_{u/v}^1
\]
if $\overline{x}\in \mathcal{G}_{u/v}^j$, which means $d_1+...+d_{u+v}=j$.
\end{definitionN}

This function is clearly surjective. Additionally, recall that this point in $\mathbb{N}\mathcal{G}_{u/v}^1$ may be written as a $(u+v)$-tuple $\textbf{d}=(d_1,...,d_{u+v})\in\mathbb{N}^{u+v}$. 

\begin{note}
    The set $\mathcal{G}_{u/v}^j$ defined above corresponds to the obvious basis for $T^j(\mathbb{Z}\mathcal{G}_{u/v}^1)$, the $j$-fold tensor product of $\mathbb{Z}\mathcal{G}_{u/v}^1$. In particular, the basis element corresponding to $\overline{x}$ is $c(x_1)\otimes...\otimes c(x_j)$. The domain of $\textbf{c}$ is, thus, the natural (infinite) basis for $T(\mathbb{Z}\mathcal{G}_{u/v}^1)=\bigoplus_{j\geq 0}T^j(\mathbb{Z}\mathcal{G}_{u/v}^1)$.
\end{note}

Consider also the natural basis of $\text{Sym}^j(\mathbb{Z}\mathcal{G}_{u/v}^1)$, the $j$-fold symmetric product of $\mathbb{Z}\mathcal{G}_{u/v}^1$, given by the symmetric tensors of the $\xi_i$. We will denote this basis by $\text{Sym}^j(\mathcal{G}_{u/v}^1)$. Observe that there is an injective function $\iota:\text{Sym}^j(\mathcal{G}_{u/v}^1) \to \mathcal{G}_{u/v}^j$ and a bijection $\tilde{\textbf{c}}:\bigcup_{j\geq 0}\text{Sym}^j(\mathcal{G}_{u/v}^1) \to \mathbb{N}\mathcal{G}_{u/v}^1$ such that (taking the obvious restrictions of $\textbf{c}$ and $\tilde{\textbf{c}}$)

\begin{center}
\begin{tikzcd}
\text{Sym}^j(\mathcal{G}_{u/v}^1) \arrow[rd,"\tilde{\textbf{c}}"'] \arrow[rr, "\iota"] & & \mathcal{G}_{u/v}^j \arrow[ld,"\textbf{c}"]\\
& \mathbb{N}\mathcal{G}_{u/v}^1 &
\end{tikzcd}
\end{center}
commutes, where, for $d=\xi_{i_1}\cdot\xi_{i_2}\cdot...\cdot\xi_{i_j}\in \text{Sym}^j(\mathcal{G}_{u/v}^1)$ written so that $1\leq i_1\leq i_2\leq...\leq i_j\leq u+v$, we have $\iota(d)=(x_1,...,x_j)$ for $c(x_m)=\xi_{i_m}$ for $1\leq m \leq j$ (i.e. the subscripts of the $c(x_i)$ are non-decreasing), and 
\[
\tilde{\textbf{c}}(d)=\sum_{m=1}^j \xi_{i_m}=\sum_{l=1}^{u+v} d_l \xi_l \in \mathbb{N}\mathcal{G}_{u/v}^1.
\]
Once again, we may write $(d_1,...,d_{u+v})$ for $\sum_{l=1}^{u+v} d_l \xi_l \in \mathbb{N}\mathcal{G}_{u/v}^1$. Now, if we let $\overline{d}=\iota(d)\in \mathcal{G}_{u/v}^j$, we may define the following set.

\begin{definitionN}
Given $d\in \text{Sym}^j(\mathcal{G}_{u/v}^1)$, define $G_{\overline{d}}$ to be the set
\[
G_{\overline{d}}=\{\overline{x} \in \mathcal{G}_{u/v}^j : \textbf{c}(\overline{x})=\textbf{c}(\overline{d})\}.
\]
Another way to write this is that $G_{\overline{d}}$ is the set of $\overline{x}$ such that $\textbf{c}(\overline{x})=\tilde{\textbf{c}}(d)$.
\end{definitionN}

Next, we will prove some lemmas about $q$-grading differences for generators in $G_{\overline{d}}$. To do so, we will use the formula
\[
q(\overline{x})-q(\overline{y})=\left(2\Phi^j+(1-2j)\left(\Psi_{X^+}^j-\Phi^j\right)+\sum_{W\in\{X^-,Y\}}\left(\Psi_W^j-\Phi^j\right)\right)([\gamma_{\overline{x},\overline{y}}])
\]
from Section \ref{compmethods}. It turns out, though, that we will only need to use a small portion of the expression on the right for studying $G_{\overline{d}}$.

\begin{lemmaN}
\label{uniquegr}
   For each $d\in \text{Sym}^j(\mathcal{G}_{u/v}^1)$, $G_{\overline{d}}$ has a unique generator of lowest $q$-grading, given by $\overline{d}$. 
\end{lemmaN}

\begin{proof}
We first consider the case where $\overline{x}=(x_1,...,x_j)$ and $\overline{y}=(y_1,...,y_j)$ are such that $x_i=y_i$ for all $i\not\in \{m,m+1\}$ for $m+1\leq k$ or $m\geq k$. In other words, they look the same except in a neighborhood of two parallel copies of $l_A$ or $l_I$. Because $\overline{x},\overline{y}\in G_{\overline{d}}$, we necessarily have $x_m=y_{m+1}$ and $x_{m+1}=y_m$. Thus, this neighborhood looks like the image in Figure \ref{rectangle}, where we assume the red points are coming from $\overline{x}$ and the blue points are coming from $\overline{y}$ (without loss of generality), so if $c(x_m)=c(y_{m+1}) = \xi_i$ and $c(x_{m+1})=c(y_m) = \xi_j$, then $i>j$ with respect to the standard ordering. The loop $\gamma_{\overline{x},\overline{y}}\subset \text{Conf}^j(M)$ that we want for computing $q(\overline{x})-q(\overline{y})$ is then given by the constant loops at all $x_i=y_i$ for $i\not\in \{m,m+1\}$ and then the pair of green paths shown in Figure \ref{rectangle}. We can think of this as swapping the red and the blue points; we will call this a \textit{rectangle move}. Because the loop does not wind around any of the three punctures, the only non-zero contribution comes from $2\Phi^j$, and we get
\[
q(\overline{x})-q(\overline{y})=2\Phi^j([\gamma_{\overline{x},\overline{y}}])=2.
\]
because the braid corresponding to this loop is the braid generator $\sigma_m\in \text{Br}_j$. 

\begin{figure}
\centering
\includegraphics[height=4.5cm, angle=0]{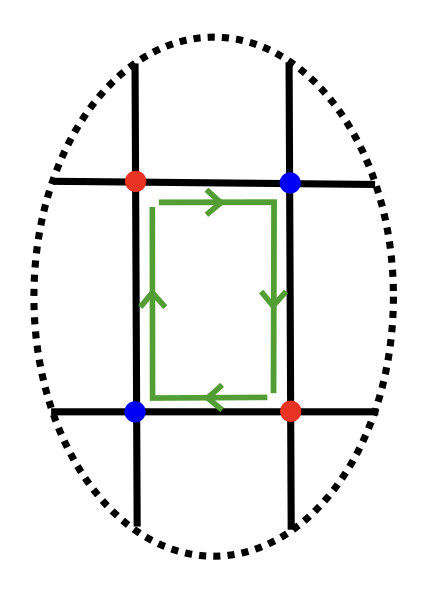}
\caption{A rectangle move, as used in the proof of Lemma \ref{uniquegr}}

\label{rectangle}
\end{figure}

Any two $\overline{x},\overline{y}\in G_{\overline{d}}$ are related by a finite sequence of these moves, so we see that there is a unique $\overline{x}\in G_{\overline{d}}$ of lowest $q$-grading obtained by sliding all intersection points on both the $l_A$ and $l_I$ sides to fill the states from bottom-left to top-right. In other words, using the standard ordering of the $\xi_i$, what we have shown is that the element of $G_{\overline{d}}$ with the lowest $q$-grading is given by the unique $(x_1,...,x_j)$ with the subscripts of $\text{c}(x_i)$ non-decreasing for $1\leq i \leq j$. This is precisely $\overline{d}$.
\end{proof}

The following illustrates what this $\overline{d}$ looks like in a concrete example. 

\begin{example}
    Consider $\tau_{4/3}$. The figure below shows how to label the intersection points $\xi_i$ according to the standard ordering. Additionally, it gives the lowest $q$-grading configuration $\overline{d}$ in $G_{\overline{d}}$, such that $\textbf{c}(\overline{d})=2\xi_1+\xi_2+2\xi_4+\xi_5+2\xi_7$, where the intersection points are given by the red points. Thus, $j=8$ and $\overline{d}=(x_1,...,x_8)$, where $c(x_1)=c(x_2)=\xi_1$, $c(x_3)=\xi_2$, $c(x_4)=c(x_5)=\xi_4$, $c(x_6)=\xi_5$, and $c(x_7)=c(x_8)=\xi_7$.

    \[
    \vcenter{\hbox{\includegraphics[height=5cm, angle=0]{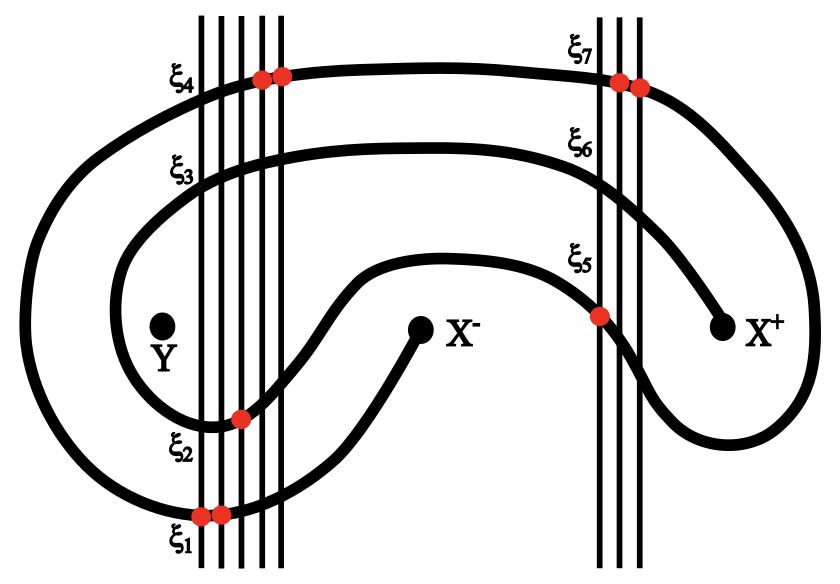}}}
    \]
\end{example}

Next, we consider how the $q$-gradings of all elements in $G_{\overline{d}}$ are related to each other.

\begin{lemmaN}
\label{genblocks}
    Given $G_{\overline{d}}$ where $\overline{d}=(\iota \circ \tilde{\textbf{c}}^{-1})(\sum_{i=1}^{u+v} d_i\xi_i)$ for $\tau_{u/v}$, we have
    \[
    \sum_{\overline{x}\in G_{\overline{d}}}q^{(q(\overline{x})-q(\overline{d}))}= {d_1+...+d_u\brack d_1,...,d_u}{d_{u+1}+...+d_{u+v}\brack d_{u+1},...,d_{u+v}}.
    \]
    Furthermore, the $a$- and $t$-gradings of all the generators in $G_{\overline{d}}$ are the same.
\end{lemmaN}

\begin{proof}
As we saw in the proof of Lemma \ref{uniquegr}, all elements of $G_{\overline{d}}$ are related by a finite sequence of rectangle moves, as shown in Figure \ref{rectangle}, which does not involve any of the three punctures, so for any $\overline{x},\overline{y}\in G_{\overline{d}}$, we get
\begin{align*}
    &q(\overline{x})-q(\overline{y})=2\Phi^j([\gamma_{\overline{x},\overline{y}}])\\
    &a(\overline{x})-a(\overline{y})=0\\
    &t(\overline{x})-t(\overline{y})=0
\end{align*}
 because $\left(\Psi_V^j-\Phi^j\right)([\gamma_{\overline{x},\overline{y}}])=0$ for each $V\in\{X^+,X^-,Y\}$. This proves the second part of the lemma.

    For the first part, this reduces to a simple combinatorial problem. The two q-multinomials come from the $u\times k$ and $v\times (j-k)$ grids that we get from the $k$ parallel copies of $l_A$ and the $j-k$ parallel copies of $l_I$ intersecting $\alpha$. Since our sum is over the generators in some $G_{\overline{d}}$, these are all generators of the same weight $k$, so each $\overline{x}$ can be split as $(\overline{x}_A,\overline{x}_I)$, where $\overline{x}_A=(x_1,...,x_k)$ and $\overline{x}_I=(x_{k+1},...,x_j)$, corresponding to the intersections on the active and inactive sides. We have
    \[
    q^{(q(\overline{x})-q(\overline{d}))}=q^{(q(\overline{x}_A)-q(\overline{d}_A))}q^{(q(\overline{x}_I)-q(\overline{d}_I))},
    \]
    since the rectangle moves only involve changes within the active or inactive sides (there is no mixing between them), so the grading differences coming from each side are independent of each other. Thus, it suffices to show
    \[
    \sum_{\overline{x}\in G_{\overline{d}}}q^{(q(\overline{x})-q(\overline{d}))}= {j\brack d_1,...,d_u}
    \]
    for the case where $k=j$. In this case, note that $\textbf{c}(\overline{d})=\sum_{i=1}^{u+v}d_i\xi_i$ with $d_i=0$ for $i>u$. Because of this, we get a bijection 
    
    \begin{align*}
    G_{\overline{d}} &\longleftrightarrow \mathcal{S}_j^u(d_1,...,d_u)\\
    \overline{x} &\longmapsto \sigma(\overline{x}),
    \end{align*}
    where $\mathcal{S}_j^u(d_1,...,d_u)$ is the set of sequences of length $j=k$ with $d_1$ 1's, $d_2$ 2's,..., and $d_u$ u's (as given by Definition \ref{seqdef}), and $\sigma(\overline{x})$ is the sequence with its $i$th term being $\sigma(\overline{x})_i=n$ if $c(x_i)=\xi_n$. 
    Recall that, in the proof of Lemma \ref{uniquegr}, we saw that the process of sliding from the red points to the blue points in a rectangle move results in a $q$-grading difference of 2; from here, it is not difficult to see
    \[
    q(\overline{x})-q(\overline{d})=2|\{(m,n): m<n \,\text{and}\, \sigma(\overline{x})_m>\sigma(\overline{x})_n\}|=: 2 \,\text{inv}(\sigma(\overline{x})) 
    \]
    because $\text{inv}(\sigma(\overline{x}))$ is the number of rectangle moves needed to get from $\overline{x}$ to $\overline{d}$. Thus, we get
    \[
    \sum_{\overline{x}\in G_{\overline{d}}}q^{(q(\overline{x})-q(\overline{d}))}= \sum_{\sigma\in \mathcal{S}_j^u(d_1,...,d_u)}q^{2\,\text{inv}(\sigma)}= {j\brack d_1,...,d_u},
    \]
    where we are using Proposition \ref{qalgcombprop} in the last equality. 
\end{proof}

Now, we define an ordering on $\mathcal{G}_{u/v}^1$, which will be helpful throughout the paper. In what follows, we will write $\alpha$ instead of $\alpha_{u/v}$ when the particular fraction $u/v$ does not matter (i.e. when defining things for arbitrary $\tau_{u/v}$).

\begin{definitionN}
    Let $\textbf{r}:[0,1]\to\alpha$ be the constant-speed parametrization of $\alpha$ with no critical points, such that $\mathbf{r}(0)=X^-$ and $\mathbf{r}(1)=X^+$. For $\xi_i,\xi_j\in \mathcal{G}_{u/v}^1$, we say $\xi_i\prec \xi_j$ if $\xi_i=\mathbf{r}(t_0)$ and $\xi_j=\mathbf{r}(t_1)$ with $t_0<t_1$.
\end{definitionN}

In other words, if we orient the arc $\alpha_{u/v}$ from $X^-$ to $X^+$, then $\xi_i\prec \xi_j$ if $\xi_i$ comes "before" $\xi_j$. 

Next, we will carefully define the homomorphisms and loops that we will need for this section. They are closely related to the ones from Section \ref{compmethods}, but they are considerably simpler. In fact, we will only need homomorphisms defined on $\pi_1(\text{Conf}^j(M))$ for $j\in\{1,2\}$. 

Recall that we defined homomorphisms $\Psi_V^j,\Phi^j:\pi_1(\text{Conf}^j(M))\to\mathbb{Z}$ in Section \ref{compmethods} (where $V\in \{X^+,X^-,Y\}$) given by compositions
\[
\Psi_V^j:\pi_1(\text{Conf}^j(M))\xrightarrow{\iota_*} \pi_1(\text{Conf}^j(\mathbb{C}\setminus V)) \to \pi_1(\text{Conf}^{j+1}(\mathbb{C}))\cong\text{Br}_{j+1}\xrightarrow{\text{ab}} \mathbb{Z},
\]
and
\[
\Phi^j:\pi_1(\text{Conf}^j(M))\xrightarrow{\iota_*} \pi_1(\text{Conf}^j(\mathbb{C}))\cong\text{Br}_j\xrightarrow{\text{ab}} \mathbb{Z}.
\]
For the first type of homomorphism that we want, let $W\subseteq \{X^+,X^-,Y\}$. Then, we define $\Psi_W:\pi_1(M)\to \mathbb{Z}$ to be 
\[
\Psi_W:=\frac{1}{2}\sum_{V\in W}\Psi_V^1.
\]
The second type of homomorphism is $\Phi:\pi_1(\text{Conf}^2(M))\to\mathbb{Z}$, which is defined just to be $\Phi:=\Phi^2$. 

Now, we must explicitly define the loops we want to consider for our homomorphisms. First, consider $\xi_i,\xi_j\in \mathcal{G}_{u/v}^1$ and suppose $\xi_i\in \alpha \cap l_i$ and $\xi_j\in \alpha \cap l_j$ for $l_i,l_j \in \{l_A,l_I\}$. We will define $\gamma_{i,j}:[0,1]\to M$ as a concatenation of paths. Let $\eta_{i,j}^1$ be the path along $\alpha$ from $\xi_i$ to $\xi_j$. Next, if $l_i=l_j$, i.e. if $\xi_i$ and $\xi_j$ are both active or inactive intersections, then define $\bar\eta_{j,i}^1$ to be the path from $\xi_j$ back to $\xi_i$ along $l_i$; if $l_i\neq l_j$, then take $\bar\eta_{j,i}^1$ to be the path from $\xi_j$ to the top intersection point of $l_j$ (either $\xi_u$ or $\xi_{u+v}$ by the standard ordering), followed by a simple slide to $l_i$ and then the path along $l_i$ back to $\xi_i$. 

\begin{definitionN}
    Given $\xi_i,\xi_j\in \mathcal{G}_{u/v}^1$, define $\gamma_{i,j}:[0,1]\to M$ by the concatenation $\gamma_{i,j}=\eta_{i,j}^1\cdot \bar\eta_{j,i}^1$.
\end{definitionN}

\begin{figure}
\begin{tikzpicture}
    \node at (0,0) {\includegraphics[height=5cm, angle=0]{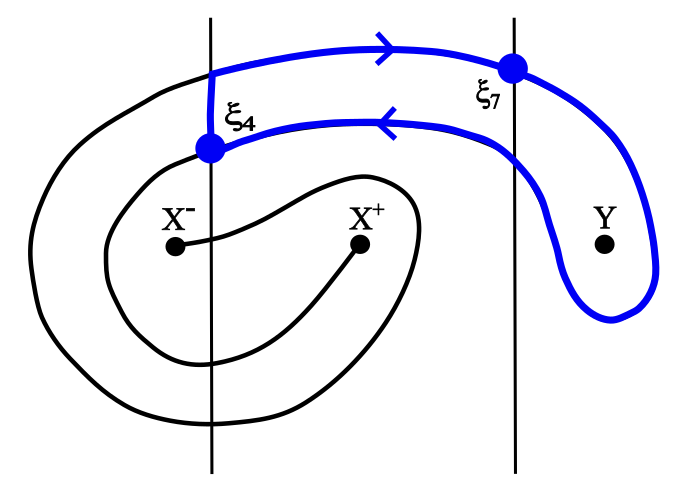}};
    \node at (6,0) {\includegraphics[height=4cm]{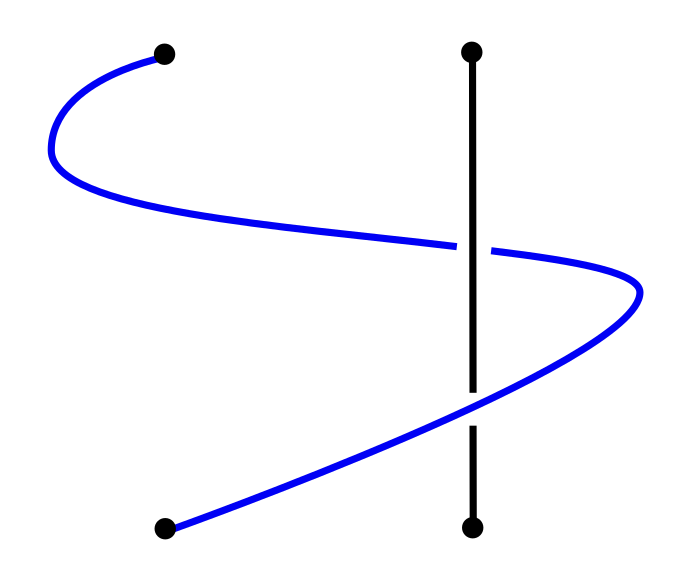}};
\end{tikzpicture}
\caption{Left: the loop $\gamma_{7,4}$ is drawn for $\tau_{5/2}$. Right: the braid image in $\text{Br}_2$ used for computing $\Psi_Y([\gamma_{4,7}]),$ where the blue strand comes from the loop and the black strand from $Y$. Notice that this gives $\Psi_Y([\gamma_{7,4}])=1$.}
\label{pathexample}
\end{figure}

See an example of one of these loops in Figure \ref{pathexample}. In the main theorem for this section, we will need to use loops of the form $\gamma_{i,j}$ where $\xi_j$ is the maximum element of $\mathcal{G}_{u/v}^1$ with respect to $\prec$. We will use special notation for designating this intersection point.

\begin{definitionN}
    Let $\xi_\omega$ denote the maximum element of $\mathcal{G}_{u/v}^1$ with respect to $\prec$.
\end{definitionN}

In other words, $\xi_\omega$ is the intersection point closest to the puncture $X^+$. Using this, we can introduce the simplified notation used in Theorem \ref{bigthm}, where we take the loops $\gamma_{i,j}$ where $j=\omega$.

\begin{definitionN}
    Define $\gamma_i:[0,1]\to M$ as the loop $\gamma_i:=\gamma_{i,\omega}$.
\end{definitionN}

It is slightly more tricky to define our desired loops in $\text{Conf}^2(M)$. For $\xi_i,\xi_j \in \mathcal{G}_{u/v}^1$, we have $\xi_i\cdot\xi_j\in\text{Sym}^2(\mathcal{G}_{u/v}^1)$, so we can apply $\iota$ to get $\iota(\xi_i\cdot\xi_j)=(x_1,x_2)\in \mathcal{G}_{u/v}^2$ such that the subscript of $c(x_1)$, which is $i$ or $j$, is less than or equal to the subscript of $c(x_2)$. Thus, $x_2$ is to the right of and possibly above $x_1$ if both are on the active or inactive side; otherwise, $x_1$ is on the active side and $x_2$ is on the inactive side. Now, $(x_1,x_2)$ is also a point in $\text{Conf}^2(M)$, and we will use these as the basepoints for our loops in $\text{Conf}^2(M)$. Additionally, recall from Section \ref{compmethods} that a loop in $\text{Conf}^2(M)$ is given by two disjoint loops in $M$ or two paths that swap starting and ending points; both situations will arise for us.

Given the set-up above, we now describe the paths used to define $\tilde{\gamma}_{i,j}=(\sigma_{i,j}^1,\sigma_{i,j}^2):[0,1]\to\text{Conf}^2(M)$, where $\sigma_{i,j}^1$ and $\sigma_{i,j}^1$ are the two component paths. Let $x_i\in\{x_1,x_2\}$ be the one such that $c(x_i)=\xi_i$ and similarly define $x_j$. First, we will define the path $\eta_{i,j;j}^2:[0,1]\to\text{Conf}^2(M)$ from $(x_1,x_2)$ to $(x_1',x_2')$ where $x_j=x_j'$ and $c(x_i')=\xi_j.$ Think of $x_j=x_j'$ as being a fixed point in $M$ (to the left or right of $\xi_j$) and consider the path $\eta_{i,j}^1$ along $\alpha$ that we have already defined, but thought of as going from $x_i$ to $x_i'$. If $\eta_{i,j}^1$ is disjoint from $x_j$, define  $\eta_{i,j;j}^2=(\eta_{i,j}^1,x_j)$, where $x_j$ is the constant path at $x_j$. Otherwise, if $\eta_{i,j}^1$ intersects $x_j$, we can think of it as a concatenation $\eta_{i,j}^1=\rho_{i,j}^1\cdot\rho_{j,i'}^1$, where $\rho_{i,j}^1$ is the path from $x_i$ to $x_j=x_j'$ and $\rho_{j,i'}^1$ is the path from $x_j=x_j'$ to $x_i'$. In this case, we take $\eta_{i,j;j}^2=(\rho_{i,j}^1,\rho_{j,i'}^1)$. These two cases are represented below, assuming $x_i$ and $x_j$ are both on the active or inactive side, but the situation is similar when one is active and the other is inactive. 
\[
\begin{tikzpicture}
    \node at (0,0) {\includegraphics[height=4cm]{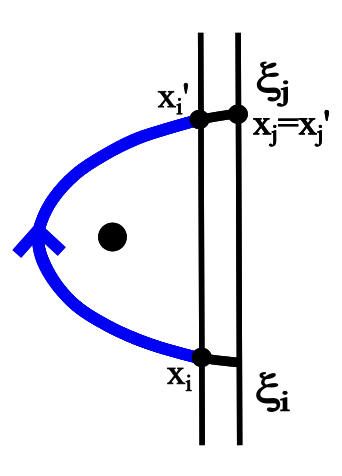}};
    \node at (5,0) {\includegraphics[height=4cm]{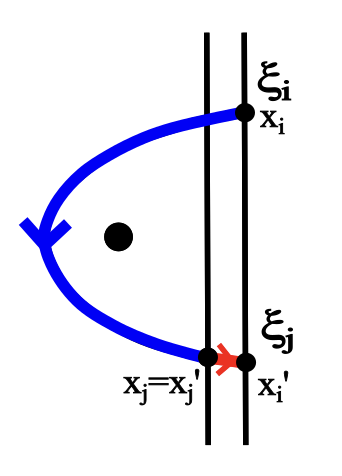}};
\end{tikzpicture}
\]
We will define $\tilde\gamma_{i,j}$ as a concatenation of $\eta_{i,j;j}^2$ with another path $\bar\eta_{j;i,j}^2=(\bar\sigma_{i,j}^1,\bar\sigma_{i,j}^2):[0,1]\to\text{Conf}^2(M)$ from $(x_1',x_2')$, where $c(x_1')=c(x_2')=\xi_j$, back to $(x_1,x_2)$. If $\eta_{i,j}^1$ was disjoint from $x_j$, then $\bar\eta_{j;i,j}=(\bar\eta_{j,i}^1,x_j)$, where $\bar\eta_{j,i}^1$ is the same path as before, but thought of as going from $x_i'$ to $x_i$. Otherwise, we take $\bar\eta_{j;i,j}=(x_j,\bar\eta_{j,i}^1)$ to get a well-defined loop in $\text{Conf}^2(M)$.

\begin{definitionN}
    Define $\tilde\gamma_{i,j}:[0,1]\to\text{Conf}^2(M)$ by the concatenation of paths $\tilde\gamma_{i,j}:=\eta_{i,j;j}^2\cdot \bar\eta_{j;i,j}^2$.
\end{definitionN}

\begin{figure}
    \raisebox{0pt}{\includegraphics[height=4cm, angle=0]{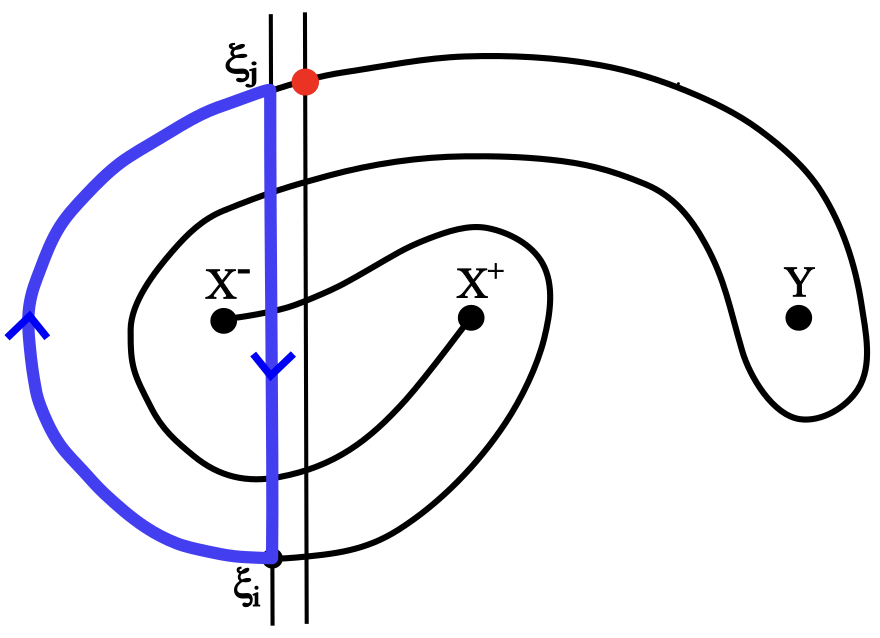}}\qquad \qquad \raisebox{0pt}{\includegraphics[height=4cm, angle=0]{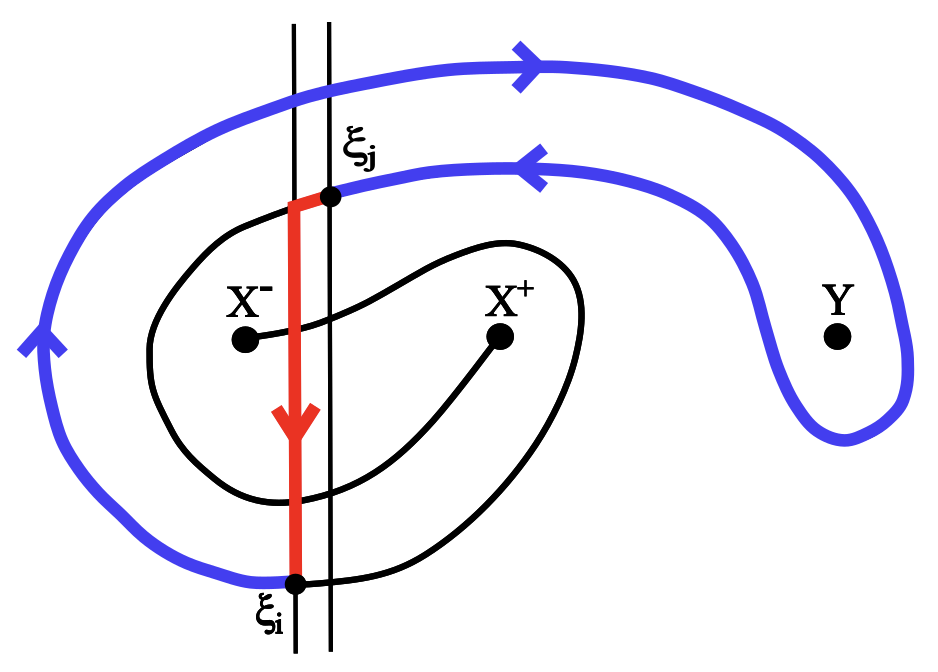}}
    \raisebox{0pt}{\includegraphics[height=4cm, angle=0]{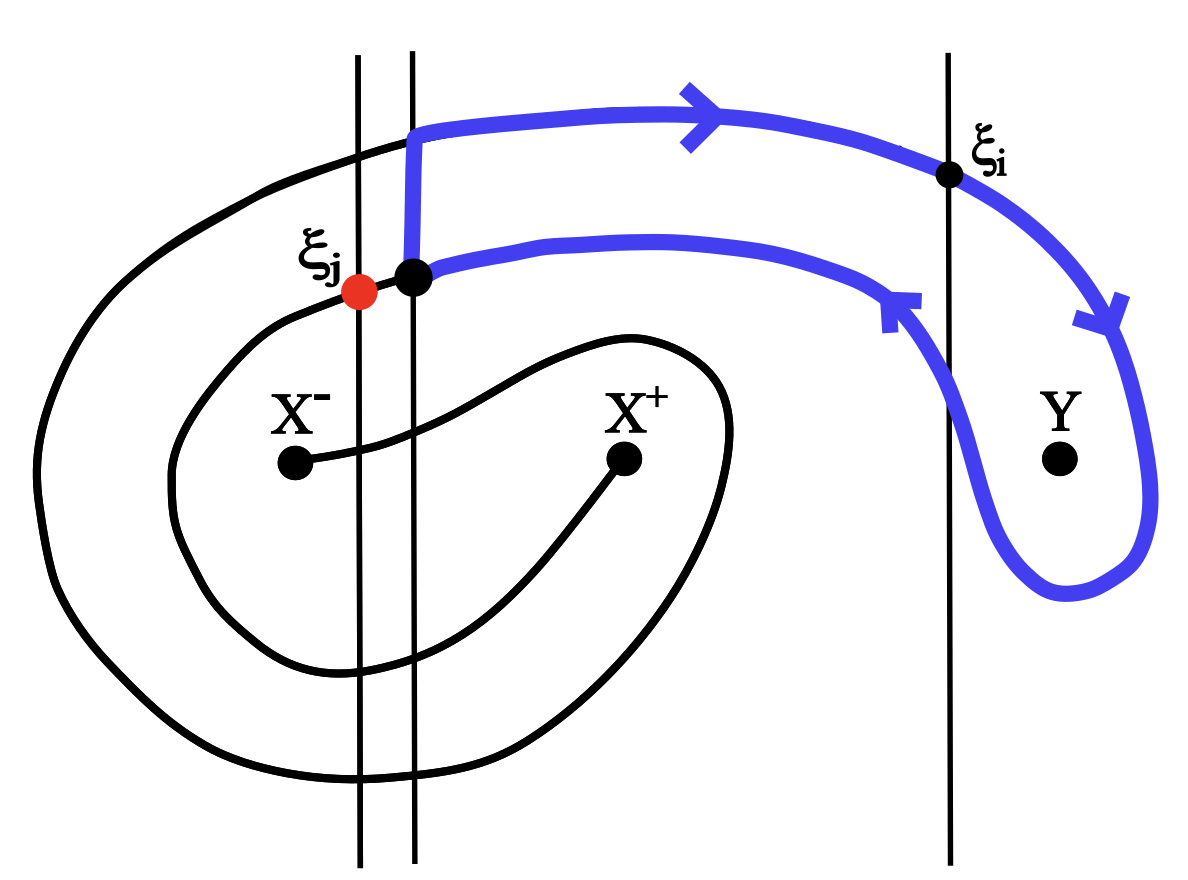}}\qquad \qquad \raisebox{0pt}{\includegraphics[height=4cm, angle=0]{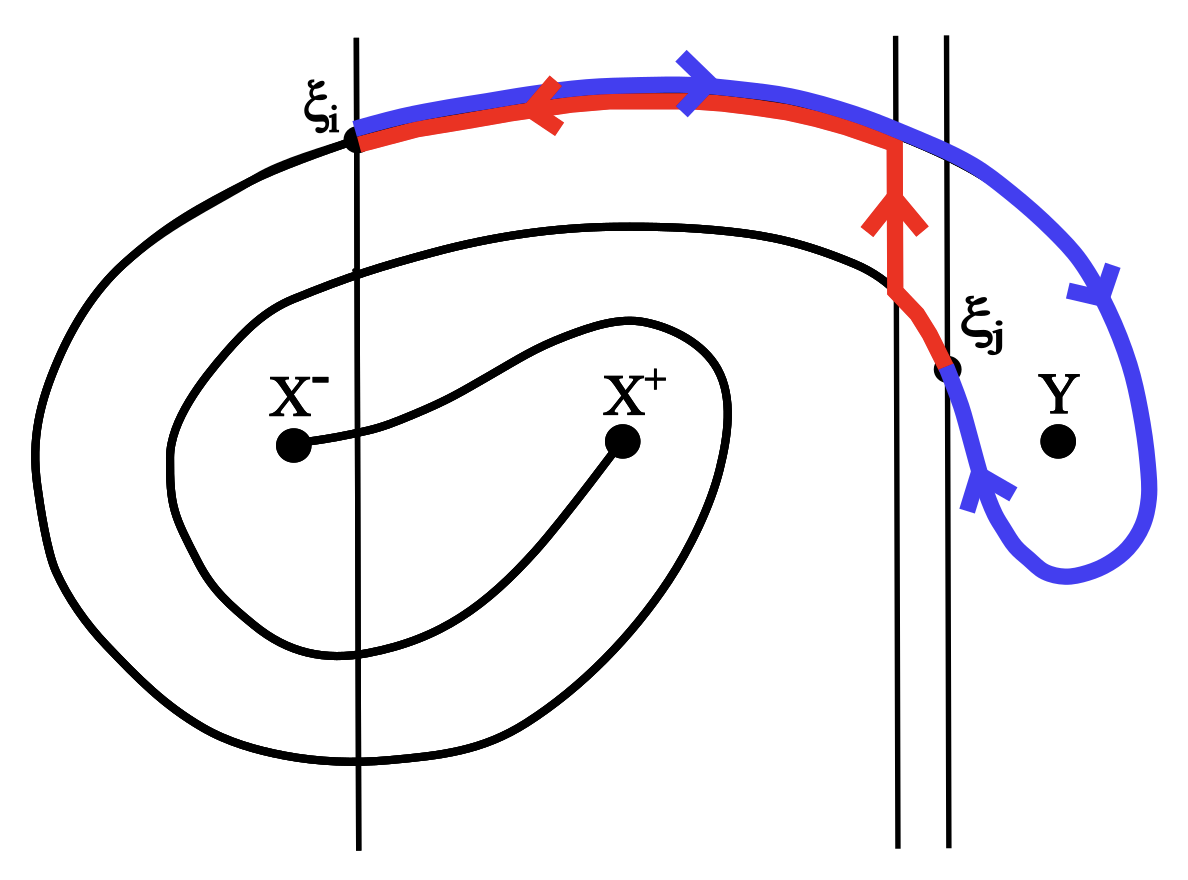}}
    \caption{Examples of $\tilde{\gamma}_{i,j}$ for $\tau_{5/2}$ demonstrating the four different cases described in the definition. Observe that $\Phi([\tilde{\gamma}_{i,j}])$ is 0 for the two pictures on the left, and it is 1 for the two on the right.}
    \label{gammacases}
\end{figure}

Refer to Figure \ref{gammacases} for examples of these paths, with the top two images giving loops for the case where $x_i$ and $x_j$ are both on the active side, and the bottom two images give loops where they are on opposite sides.

\begin{note}
We will abuse notation slightly and write $(\Psi_W+\Phi)([\gamma_{i,j}])$ instead of $\Psi_W([\gamma_{i,j}])+\Phi([\tilde{\gamma}_{i,j}])$.
\end{note}

Before proceeding, it is worth thinking about how $\tilde \gamma_{i,j}=(\sigma_{i,j}^1,\sigma_{i,j}^2)$ is related to $\gamma_{i,j}$. Particularly, $\sigma_{i,j}^1$ and $\sigma_{i,j}^2$ may be thought of as paths in $M$, and then it is easy to see from the definitions that $\sigma_{i,j}^1+\sigma_{i,j}^2=\gamma_{i,j}$ in $H_1(M)$. We also get this from the isomorphism $H_1(M)\cong H_1(\text{Sym}^2(M))$.

Now, we can state our main theorem for this section. In the formulas below, expressions like $\delta_{i,A}$ are to be read as 
\[
\delta_{i,A}=
\begin{cases}
    1, \qquad \xi_i \,\text{is an active intersection point}\\
    0, \qquad \text{otherwise}
\end{cases}
\]
and $\delta_{i,I}$ is similar for inactive intersections. Additionally, we will use $Z$ in the proof to denote the middle of the three punctures $X^+,X^-,Y$ in $\mathcal{D}(\tau_{u/v})$, so $\delta_{Z,X^+}$ should also be viewed as a Kronecker delta giving $1$ if $Z=X^+$ and 0 otherwise. It should be noted that $Z=X^+$ is equivalent to $\tau_{u/v}$ having $OP$ orientation.

\begin{theoremN}
\label{bigthm}
    We can express $\mathcal{P}(\tau_{u/v})$, up to some overall shift in the $(q,a,t)$-trigradings, as
    \begin{multline}
    \label{bigthmeq}
        \sum_{\bf{d}=(d_1,...,d_{u+v})\in \mathbb{N}^{u+v}}q^{S\cdot\textbf{d}+\textbf{d}\cdot Q\cdot \textbf{d}^t}a^{A\cdot \textbf{d}} t^{T\cdot\textbf{d}} {d_1+...+d_u\brack d_1,...,d_u}{d_{u+1}+...+d_{u+v}\brack d_{u+1},...,d_{u+v}}\\ \times X[d_1+...+d_{u+v},d_1+...+d_u],
    \end{multline}
    where $S,A,T,$ and $Q$ are computed by the following formulas:
    \newline
    \[
    \begin{cases}
        S_i=\Psi_{\{X^\pm,Y\}}([\gamma_i])+\delta_{i,A}\delta_{\omega,I}-\delta_{i,I}\delta_{\omega,A}\\
        A_i=2\Psi_{X^+}([\gamma_i])+\delta_{Z,X^+}(\delta_{i,A}\delta_{\omega,I}-\delta_{i,I}\delta_{\omega,A})\\
        T_i=-S_i=-\Psi_{\{X^\pm,Y\}}([\gamma_i])-\delta_{i,A}\delta_{\omega,I}+\delta_{i,I}\delta_{\omega,A}\\
        \begin{cases}
            Q_{ii}=(\Psi_{\{X^-,Y\}}-3\Psi_{X^+})([\gamma_i])+2\delta_{Z,X^+}(\delta_{i,I}\delta_{\omega,A}-\delta_{i,A}\delta_{\omega,I})\\
            Q_{ij}=Q_{ii}+(\Phi-2\Psi_{X^+})([\gamma_{j,i}])+\delta_{Z,X^+}(\delta_{i,A}\delta_{j,I}-\delta_{j,A}\delta_{i,I}).
        \end{cases}
    \end{cases}
    \]
\end{theoremN}

In other words, $\mathcal{P}(\tau_{u/v})$ can be put in quiver form with the linear and quadratic forms computed from the geometry of $\tau_{u/v}$. It should also be noted that the $\delta$ terms in Theorem \ref{bigthm} are unnecessarily complicated for computing $\mathcal{P}(\tau_{u/v})$ up to some arbitrary shift, but they have been written in this way to guarantee that $S_\omega=A_\omega=T_\omega=Q_{\omega\omega}=0$ for all $\tau_{u/v}$. Equivalently, the generators in the colored HOMFLY-PT complexes coming from Lagrangian intersections $\overline{x}$ with $\textbf{c}(\overline{x})=j\xi_\omega$ are in $(q,a,t)$-gradings $(0,0,0)$ for all $j$.

Before proving this theorem, we will first show that the formula for $Q$ results in a symmetric matrix.

\subsection{Symmetry of $Q$}
\label{symsec}

The formulas in Theorem \ref{bigthm} are presented in a way that makes computation as easy as possible. In particular, we can view the diagonal of $Q$ as a linear form that can be computed in terms of loops in $M$. However, we can also compute $Q$ purely in terms of loops in $\text{Conf}^2(M)$, and doing so will make it easier to show that $Q$ is symmetric, i.e. that $Q_{ij}=Q_{ji}$. 

The homomorphisms we need for this new method of computing $Q$ are $\Phi^2$ (or just $\Phi$) and $\Psi_{X^+}^2$. We will keep the superscripts to emphasize that everything in this subsection takes place in $\text{Conf}^2(M)$. 

We will ultimately provide a single formula for computing entries of $Q$, but it will help to first translate the diagonal part to $\text{Conf}^2(M)$, so we start by defining the loop used in these computations. To compute $Q_{ii}$, we will define a loop $\gamma_{i,i}^2:[0,1]\to\text{Conf}^2(M)$ with basepoint $(x_1,x_2)\in \mathcal{G}_{u/v}^2$ such that $c(x_1)=c(x_2)=\xi_i$. First, define the path $\eta_{i;j}^2$ from $(x_1,x_2)$ to $(y_1,y_2)$, with $c(y_1)=c(y_2)=\xi_j$ by $\eta_{i;j}^2=(\eta_{i,j}^1,\eta_{i,j}^1)$, thought of as two parallel copies of the path $\eta_{i,j}^1$ taken in the natural way that does not intersect the diagonal $\Delta$ of $\text{Conf}^2(M)$. This means that one of the paths is always in front of the other, so if they turn an odd number of times, they swap the verticals on which they start and end. This follows from the fact that the path is defined up to homotopy in the complement of $\Delta$, so the component paths never cross each other. To define $\gamma_{i,i}^2$, we will want $\eta_{i;\omega}^2$ in particular.  

Now, we define the path $\bar\eta_{\omega;i}^2$ by taking two parallel copies of $\bar\eta_{\omega,i}^1$, i.e. $\bar\eta_{\omega;i}^2=(\bar\eta_{\omega,i}^1,\bar\eta_{\omega,i}^1)$, where the paths are pushed off of each other to avoid the diagonal $\Delta$. 

\begin{definitionN}
  The loop $\gamma_{i,i}^2$ is defined to be the concatenation $\gamma_{i,i}^2=\eta_{i;\omega}^2\cdot\bar\eta_{\omega;i}^2$.  
\end{definitionN}

\begin{figure}
    \begin{tikzpicture}
        \node at (0,0) {\includegraphics[height=5cm]{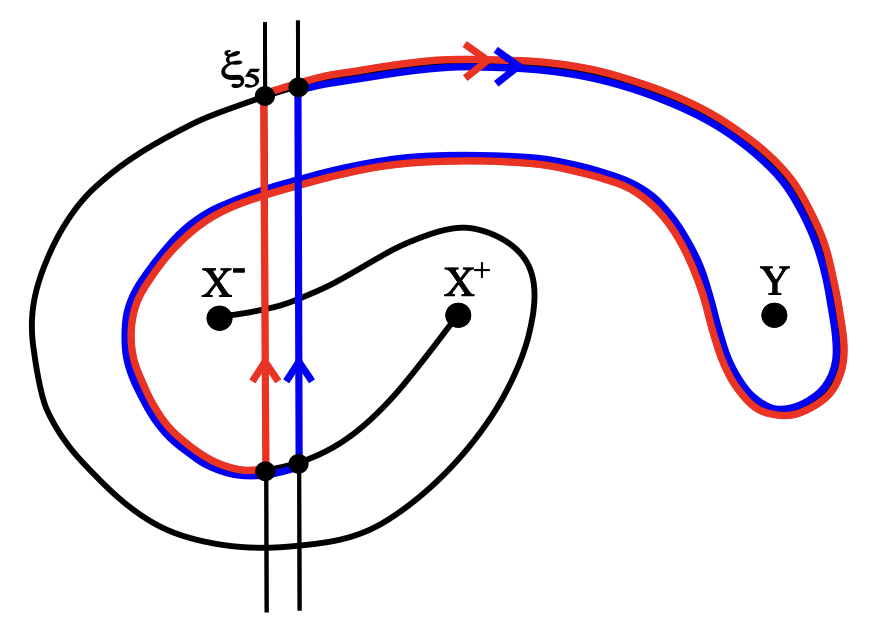}};
        \node at (6,0) {\includegraphics[height=4cm]{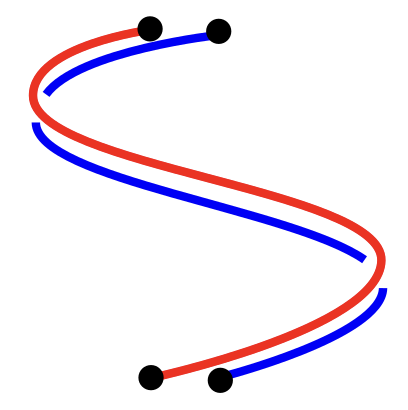}};
    \end{tikzpicture}
    \caption{Left: the loop $\gamma_{5,5}^2$ for $\tau_{5/2}$. Right: the corresponding braid in $\text{Br}_2$ for $\gamma_{5,5}^2$}
    \label{gamma552fig}
\end{figure}

The loop may be thought of, intuitively, as two parallel copies of $\gamma_i$, defined in the natural way to miss the diagonal. In Figure \ref{gamma552fig}, one can see the example of $\gamma_{5,5}^2$ for $\tau_{5/2}$.

Next, we need to define the general loops $\gamma_{j,i}^2$ that will be used in our new formula for $Q$. To do so, we will use $\eta_{j,i;i}^2$ (the same path that would be used to define $\tilde\gamma_{j,i}$) and $\eta_{i;\omega}^2$, which was just defined. The only additional path we need will be denoted $\bar\eta_{\omega;j,i}^2$, which is a slight generalization of $\bar\eta_{\omega;i}^2$. The difference is that, rather than taking two parallel copies of $\bar\eta_{\omega,i}^1$, we define the component paths to be copies of $\bar\eta_{\omega,i}^1$ and $\bar\eta_{\omega,j}^1$ taken the natural way that results in a well-defined loop in $\text{Conf}^2(M)$.


\begin{definitionN}
\label{conf2gammaloops}
    Define the loop $\gamma_{i,j}^2:[0,1]\to\text{Conf}^2(M)$ by the concatenation $\gamma_{i,j}^2:=\eta_{i,j;j}^2\cdot\eta_{j;\omega}^2\cdot\bar\eta_{\omega;i,j}^2$.
\end{definitionN}

An example with $i\neq j$ is provided in Figure \ref{gamma152fig}.

\begin{note}
    Definition \ref{conf2gammaloops} agrees with how we defined $\gamma_{i,i}^2$ (up to homotopy) because $\eta_{i,i;i}$ is the constant loop at $(x_1,x_2)$ with $c(x_1)=c(x_2)=\xi_i$ and we wrote $\bar\eta_{\omega;i}^2$ for the final path, but this is the same thing as $\bar\eta_{\omega;i,i}^2$.
\end{note}

Now, we are equipped to provide our new formula for computing $Q$ in terms of loops in $\text{Conf}^2(M)$. We will present the formula and then show that it works by proving the following proposition.

\begin{propositionN}
\label{Qsymprop}
    The quadratic form $Q$ may be computed by the formula
    \begin{equation}
     \label{conf2Qformula}
Q_{ij}=(2\Phi^2-\Psi_{X^+}^2)([\gamma_{j,i}^2])+\Delta,   
    \end{equation}
where 
\[
\Delta=\delta_{Z,X^+}(2\delta_{i,I}\delta_{j,I}\delta_{\omega,A}-2\delta_{i,A}\delta_{j,A}\delta_{\omega,I}+\delta_{i,A}\delta_{j,I}\delta_{\omega,A}+\delta_{i,I}\delta_{j,A}\delta_{\omega,A}-\delta_{i,I}\delta_{j,A}\delta_{\omega,I}-\delta_{i,A}\delta_{j,I}\delta_{\omega,I}),
\]
    in the sense that this agrees with Theorem \ref{bigthm}. Furthermore, Equation (\ref{conf2Qformula}) satisfies $Q_{ij}=Q_{ji}$.
\end{propositionN}

\begin{figure}
    \begin{tikzpicture}
        \node at (0,0) {\includegraphics[height=5cm]{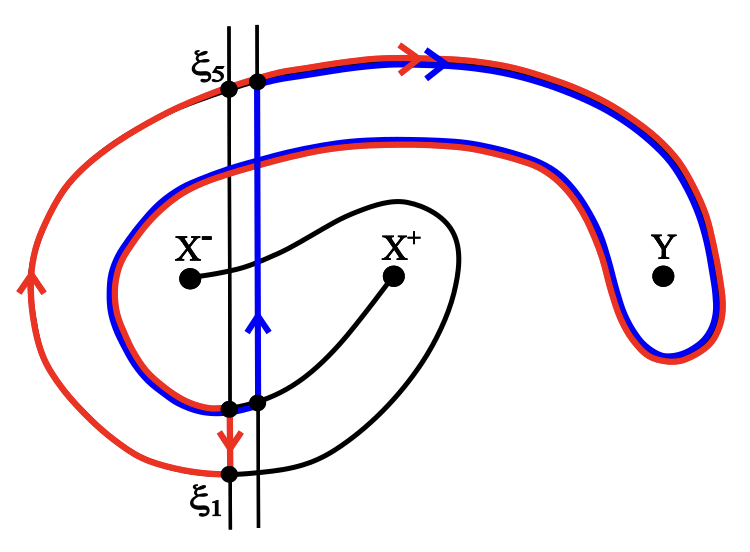}};
        \node at (6,0) {\includegraphics[height=4cm]{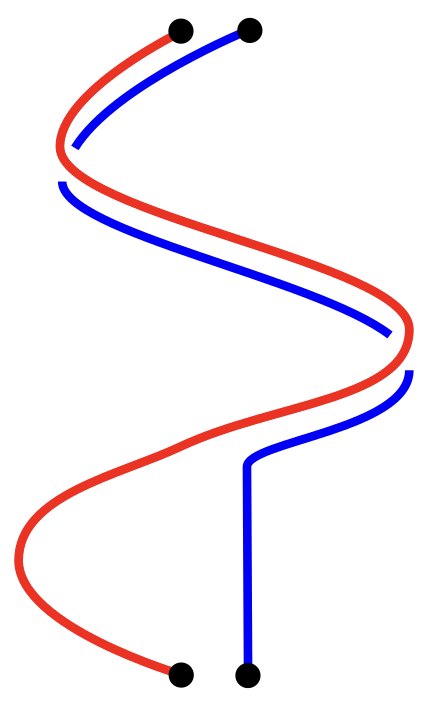}};
    \end{tikzpicture}
    \caption{Left: the loop $\gamma_{1,5}^2$. Right: the corresponding braid in $\text{Br}_2$ for $\gamma_{1,5}^2$}
    \label{gamma152fig}
\end{figure}

The proof of this proposition will occupy the remainder of this subsection and will be done in steps. First, we prove symmetry.

\begin{lemmaN}
Using Equation (\ref{conf2Qformula}), one computes $Q_{ij}=Q_{ji}$  
\end{lemmaN}

\begin{proof}
 First, we check that $\gamma_{j,i}^2$ and $\gamma_{i,j}^2$ are homologous. Without loss of generality, we may assume $\xi_i\prec \xi_j$, which means that $\gamma_{j,i}^2$ ``backtracks.'' More precisely, we can break $\eta_{i;\omega}^2$ further into the concatenation $\eta_{i;j}^2\cdot \eta_{j;\omega}^2$; then the backtracking refers to the easy observation that $\eta_{j,i;i}^2\cdot \eta_{i;j}^2\sim \eta_{i,j;j}^2$, where $\sim$ denotes that the paths are homologous in $\text{Conf}^2(M)$ (their difference forms a boundary cycle). This means
 \[
 \gamma_{j,i}^2=\eta_{j,i;i}^2\cdot \eta_{i;j}^2\cdot \eta_{j;\omega}^2\cdot  \bar\eta_{\omega;j,i}^2\sim \eta_{i,j;j}^2\cdot \eta_{j;\omega}^2\cdot  \bar\eta_{\omega;j,i}^2= \eta_{i,j;j}^2\cdot \eta_{j;\omega}^2\cdot  \bar\eta_{\omega;i,j}^2=\gamma_{i,j}^2,
 \]
 where we have also used $\bar\eta_{\omega;j,i}^2=\bar\eta_{\omega;i,j}^2$. Because $\Phi^2$ and $\Psi_{X^+}^2$ are $\mathbb{Z}$-valued homomorphisms, they factor through $H_1(\text{Conf}^2(M))$, so $\gamma_{j,i}^2$ and $\gamma_{i,j}^2$ being homologous implies 
 \[
 (2\Phi^2-\Psi_{X^+}^2)([\gamma_{j,i}^2])=(2\Phi^2-\Psi_{X^+}^2)([\gamma_{i,j}^2]).
 \]
 The proof is concluded by observing that $\Delta$ is symmetric in $i$ and $j$.
\end{proof}

Before proceeding to the last step in the proof of Proposition \ref{Qsymprop}, we need to establish the following lemma. This lemma will ultimately enable us to rewrite Equation (\ref{conf2Qformula}) in a form comparable to the formula for $Q_{ij}$ in Theorem \ref{bigthm}.

\begin{lemmaN}
\label{1cycleslem}
    As $1$-cycles in $H_1(\text{Conf}\,^2(M)),$
    \begin{equation}
        \gamma_{j,i}^2=\tilde\gamma_{j,i}+\gamma_{i,i}^2.
    \end{equation}
\end{lemmaN}

\begin{proof}
    First, we prove
    \begin{equation}
    \label{n2cH}
(\bar\eta_{i;j,i}^2)^*\cdot\eta_{i;\omega}^2\cdot\eta_{\omega;j,i}^2\sim \eta_{i;\omega}^2\cdot \bar\eta_{\omega;i,i}^2=\gamma_{i,i}^2,
    \end{equation}
    where $(\bar\eta_{i;j,i}^2)^*$ is the reverse of $\bar\eta_{i;j,i}^2$. In other words, as 1-chains in 
$H_1(\text{Conf}^2(M)),$ we have $(\bar\eta_{i;j,i}^2)^*=-\bar\eta_{i;j,i}^2$. To see that (\ref{n2cH}) holds, it suffices to show that $(\bar\eta_{i;j,i}^2)^*+\eta_{\omega;j,i}^2-\bar\eta_{\omega;i,i}^2=0$ in $H_1(\text{Conf}^2(M))$.

    If $\xi_i, \xi_j,$ and $\xi_\omega$ are all active or inactive intersections, these three paths only live on the vertical Lagrangians, and if they are not all active or inactive, it can be checked that the components of the paths not on the vertical Lagrangians cancel. Now, we explicitly consider what happens when $\xi_i, \xi_j$ and $\xi_\omega$ are all active or inactive. Below, we show the six cases for how they might be arranged with respect to each other along the appropriate vertical Lagrangians (the two intersection points $x_1,x_2$ along the horizontal labeled $k$ are such that $c(x_1)=c(x_2)=\xi_k$). For each of the six cases, we have drawn $(\bar\eta_{i;j,i}^2)^*$ in red (excluding the constant component), $\eta_{\omega;j,i}^2$ in blue, and $\bar\eta_{\omega;i,i}^2$ in green.
    \[
    \begin{tikzpicture}
        \node at (0,0) {\includegraphics[height=3cm]{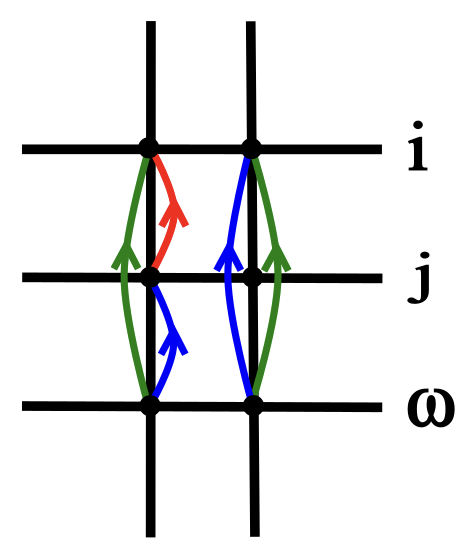}};
        \node at (5,0) {\includegraphics[height=3cm]{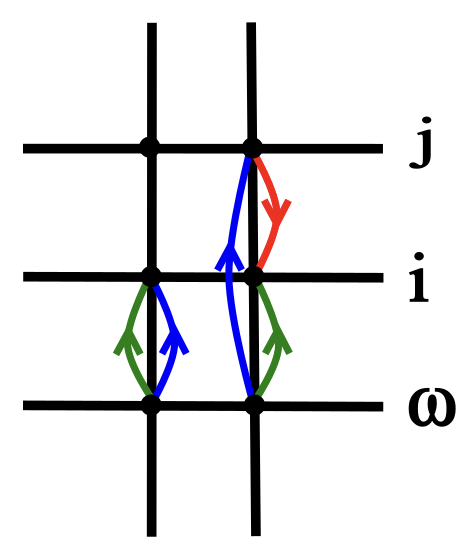}};
        \node at (10,0) {\includegraphics[height=3cm]{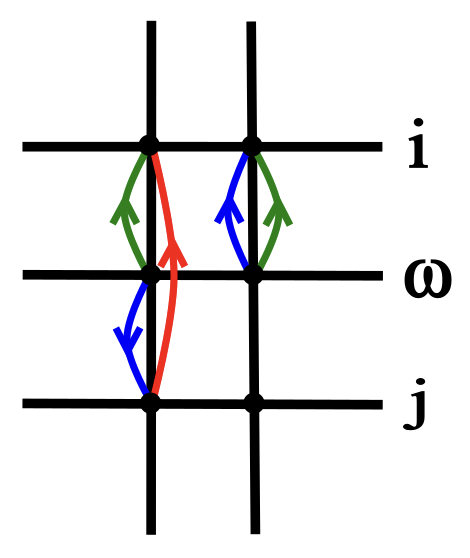}};
        \node at (0,-3.5) {\includegraphics[height=3cm]{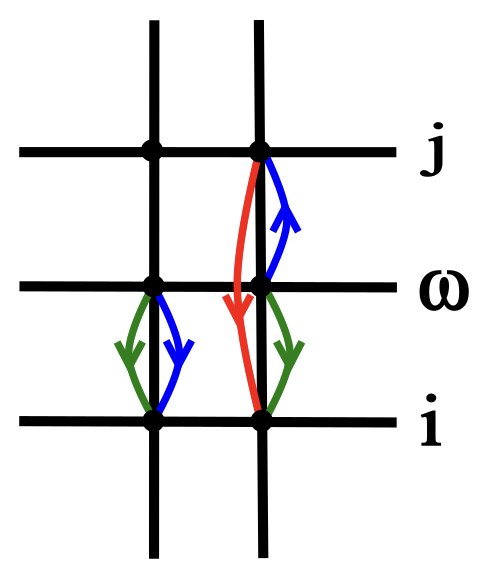}};
        \node at (5,-3.5) {\includegraphics[height=3cm]{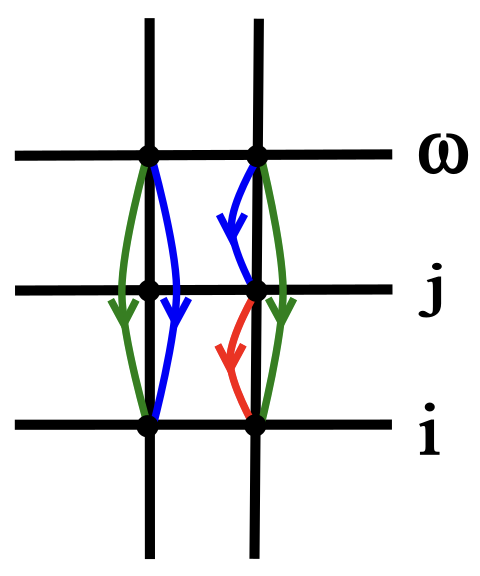}};
        \node at (10,-3.5) {\includegraphics[height=3cm]{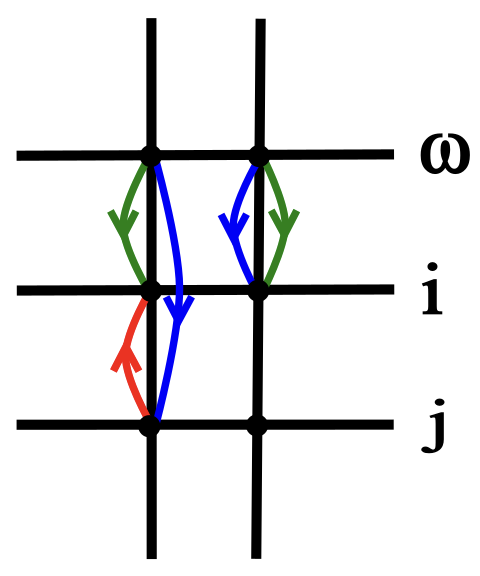}};
    \end{tikzpicture}
    \]
    In each case, we see that flipping the green arrows (taking $-\bar\eta_{\omega;i,i}^2$) results in a complete cancellation, which proves $(\bar\eta_{i;j,i}^2)^*+\eta_{\omega;j,i}^2-\bar\eta_{\omega;i,i}^2=0$. Now, we have
    \[
    \gamma_{j,i}^2=\eta_{j,i;i}^2\cdot \eta_{i;\omega}^2\cdot \eta_{\omega;j,i}^2 \sim \eta_{j,i;i}^2\cdot (\bar\eta_{i;j,i}^2\cdot (\bar\eta_{i;j,i}^2)^*)\cdot\eta_{i;\omega}^2\cdot\eta_{\omega;j,i}^2=(\eta_{j,i;i}^2\cdot \bar\eta_{i;j,i}^2)\cdot ((\bar\eta_{i;j,i}^2)^*\cdot\eta_{i;\omega}^2\cdot\eta_{\omega;j,i}^2),
    \]
    and in $H_1(\text{Conf}^2(M))$, the last term above splits into a sum. In particular, 
    \[
    (\eta_{j,i;i}^2\cdot \bar\eta_{i;j,i}^2)\cdot ((\bar\eta_{i;j,i}^2)^*\cdot\eta_{i;\omega}^2\cdot\eta_{\omega;j,i}^2)=(\eta_{j,i;i}^2\cdot \bar\eta_{i;j,i}^2)+((\bar\eta_{i;j,i}^2)^*\cdot\eta_{i;\omega}^2\cdot\eta_{\omega;j,i}^2),
    \]
    where $\eta_{j,i;i}^2\cdot \bar\eta_{i;j,i}^2=\tilde\gamma_{j,i}$ and $(\bar\eta_{i;j,i}^2)^*\cdot\eta_{i;\omega}^2\cdot\eta_{\omega;j,i}^2\sim \gamma_{i,i}^2$ by (\ref{n2cH}), so the lemma follows.
\end{proof}

Finally, we need to check that Equation (\ref{conf2Qformula}) computes the same $Q$ as Theorem \ref{bigthm}. Then, we will have completed the proof of Proposition \ref{Qsymprop}, which shows that the formula for $Q$ in Theorem \ref{bigthm} gives a symmetric matrix.

\begin{lemmaN}
    Equation (\ref{conf2Qformula}) agrees with the formula for $Q$ in Theorem \ref{bigthm}.
\end{lemmaN}

\begin{proof}
    First, observe that the formula for $Q_{ii}$ in Theorem \ref{bigthm} may be written as 
    \[
    Q_{ii}=(\Psi_{\{X^+,X^-,Y\}}-4\Psi_{X^+})([\gamma_{i}])+2\delta_{Z,X^+}(\delta_{i,I}\delta_{\omega,A}-\delta_{i,A}\delta_{\omega,I}).
    \]
    We claim that $\Psi_{\{X^+,X^-,Y\}}([\gamma_i])=\Phi^2([\gamma_{i,i}^2])$. Observe that $\Psi_{\{X^+,X^-,Y\}}([\gamma_i])$ computes the number of times $\gamma_i$ wraps around any of the punctures (with sign), and for the loop $\gamma_{i,i}^2$ in $\text{Conf}^2(M)$, wrapping around a puncture has the effect of adding a single crossing in the braid image of $\gamma_{i,i}^2$ in $\text{Br}_2$. For example, the figures below show how $\gamma_i$ wrapping around a puncture in $M$ translates to the 2-strand braid corresponding to $\gamma_{i,i}^2$ having a single crossing.
    \[
    \begin{tikzpicture}
        \node at (0,0) {\includegraphics[height=3cm]{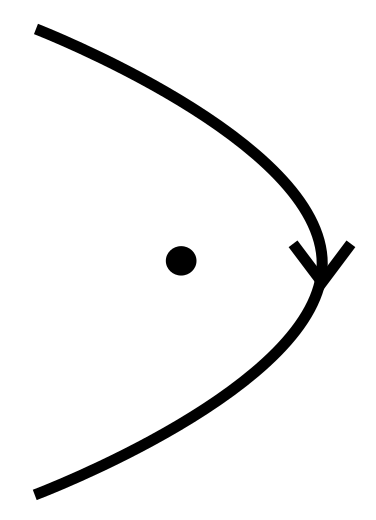}};
        \node at (2,0) {$\longrightarrow$};
        \node at (3.5,0) {\includegraphics[height=3cm]{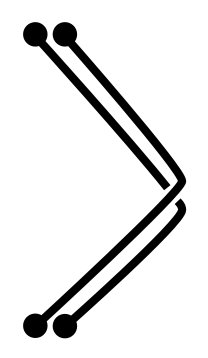}};
        \node at (7,0) {\includegraphics[height=3cm]{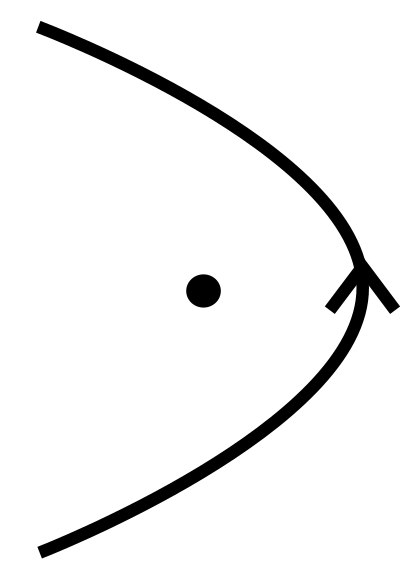}};
        \node at (9,0) {$\longrightarrow$};
        \node at (10.5,0) {\includegraphics[height=3cm]{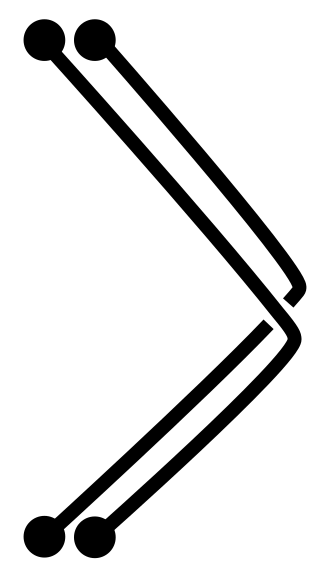}};
    \end{tikzpicture}
    \]
    We can also consider the two cases where $\alpha_{u/v}$ bends the other way around the puncture. In each case, we see that the signs contributed by these path segments to $\Psi_{\{X^+,X^-,Y\}}([\gamma_i])$ and $\Phi^2([\gamma_{i,i}^2])$ agree. It follows that $\Psi_{\{X^+,X^-,Y\}}([\gamma_i])=\Phi^2([\gamma_{i,i}^2]),$ where we are also using the assumption that $\alpha_{u/v}$ has been constructed to be as ``tight'' as possible, which means in this case that these are the only crossings in the braid for $\gamma_{i,i}^2$.
    
    Additionally, each time $\gamma_i$ wraps around $X^+$ in $M$, the loop $\gamma_{i,i}^2$ wraps around it twice. We know from Section \ref{compmethods} that $\frac{1}{2}\left(\Psi_{X^+}^2-\Phi^2\right)$ computes how many times a loop in $\text{Conf}^2(M)$ wraps around $X^+$, which means
    \[
    2\Psi_{X^+}([\gamma_i])=\frac{1}{2}\left(\Psi_{X^+}^2-\Phi^2\right)([\gamma_{i,i}^2]),
    \]
    so putting all this together gives 
    \[
    (\Psi_{\{X^+,X^-,Y\}}-4\Psi_{X^+})([\gamma_{i}])=(2\Phi^2-\Psi_{X^+}^2)([\gamma_{i,i}^2]).
    \]
    Now, it is easy to see that Equation (\ref{conf2Qformula}) and Theorem \ref{bigthm} agree on the diagonal of $Q$ by noticing that the $\delta$ terms also match.

    Next, we must check for agreement on the off-diagonal entries of $Q$. To show this, we need to use Lemma \ref{1cycleslem}, which gives
    \[
    (2\Phi^2-\Psi_{X^+}^2)([\gamma_{j,i}^2])=(2\Phi^2-\Psi_{X^+}^2)([\tilde\gamma_{j,i}])+(2\Phi^2-\Psi_{X^+}^2)([\gamma_{i,i}^2]).
    \]
    Consequently,
    \[
    Q_{ij}=Q_{ii}+(2\Phi^2-\Psi_{X^+}^2)([\tilde\gamma_{j,i}])+\delta_{Z,X^+}(\delta_{i,A}\delta_{j,I}-\delta_{j,A}\delta_{i,I}),
    \]
    where the $\delta$ terms can, once again, be worked out by considering the eight possible cases for whether $\xi_i,\xi_j,$ and $\xi_\omega$ are active or inactive intersections. Now, we only need to check that 
    \[
    (2\Phi^2-\Psi_{X^+}^2)([\tilde\gamma_{j,i}])=\Phi([\tilde\gamma_{j,i}])-2\Psi_{X^+}([\gamma_{j,i}]).
    \]
    However, this follows immediately from the fact that $\tilde\gamma_{j,i}$ and $\gamma_{j,i}$ wind around $X^+$ the same number of times. This is an easy consequence of the fact that $\sigma_{j,i}^1+\sigma_{j,i}^2=\gamma_{j,i}$ in $H_1(M)$, as we discussed earlier, where $\tilde\gamma_{j,i}=(\sigma_{j,i}^1,\sigma_{j,i}^2)$. This means
    \[
    \Psi_{X^+}([\gamma_{j,i}])=\frac{1}{2}\left(\Psi_{X^+}^2-\Phi^2\right)([\tilde\gamma_{j,i}]).
    \]
\end{proof}

This concludes the proof of Proposition \ref{Qsymprop}, which tells us that $Q_{ij}=Q_{ji}$ whether we compute $Q$ by Theorem \ref{bigthm} or by Equation (\ref{conf2Qformula}).

\subsection{Specializing to the HOMFLY-PT polynomial}
\label{specializationsec}

In Section \ref{quiversec}, we briefly discussed what it means for $P(\tau_{u/v})=\sum_{j\geq 0}\langle\tau_{u/v}\rangle_j$ to be in quiver form. We will need to use this type of generating function in \cite{JHpII26}, where our main results will be phrased in terms of the fully decategorified HOMFLY-PT invariants. Consequently, we state the equivalent version of Theorem \ref{bigthm} for $P(\tau_{u/v})$ below. 


\begin{propositionN}
\label{tangleHOMFLYpoly}
    We can express $P(\tau_{u/v})$ in quiver form, i.e.
    \begin{multline}
    \label{HOMFLYthmeq}
        \sum_{\bf{d}=(d_1,...,d_{u+v})\in \mathbb{N}^{u+v}}(-q)^{S\cdot\bf{d}}a^{A\cdot \bf{d}}q^{\bf{d}\cdot Q\cdot \bf{d}^t}{d_1+...+d_u\brack d_1,...,d_u}{d_{u+1}+...+d_{u+v}\brack d_{u+1},...,d_{u+v}}\\ \times X[d_1+...+d_{u+v},d_1+...+d_u],
    \end{multline}
    where $S,A,$ and $Q$ are given by the following formulas:
    \newline
    \[
    \begin{cases}
        S_i=\Psi_{\{X^\pm,Y\}}([\gamma_i])+\delta_{i,A}\delta_{\omega,I}-\delta_{i,I}\delta_{\omega,A}\\
        A_i=2\Psi_{X^+}([\gamma_i])+\delta_{Z,X^+}(\delta_{i,A}\delta_{\omega,I}-\delta_{i,I}\delta_{\omega,A})\\
        \begin{cases}
            Q_{ii}=(\Psi_{\{X^-,Y\}}-3\Psi_{X^+})([\gamma_i])+2\delta_{Z,X^+}(\delta_{i,I}\delta_{\omega,A}-\delta_{i,A}\delta_{\omega,I})\\
            Q_{ij}=Q_{ii}+(\Phi-2\Psi_{X^+})([\gamma_{j,i}])+\delta_{Z,X^+}(\delta_{i,A}\delta_{j,I}-\delta_{j,A}\delta_{i,I}).
        \end{cases}
    \end{cases}
    \]
\end{propositionN}
    
\begin{proof}
This theorem follows trivially from Theorem \ref{bigthm} because 
\[
\langle\tau_{u/v}\rangle_j=\langle\langle\tau_{u/v}\rangle\rangle_j\bigg|_{t\mapsto-1},
\] 
which means $P(\tau_{u/v})$ is obtained from $\mathcal{P}(\tau_{u/v})$ by setting $t=-1$. The only other thing we need to observe is that, as vectors, $T=-S$, so when we set $t=-1$, the relevant part of $\mathcal{P}(\tau_{u/v})$ becomes $q^{S\cdot \textbf{d}}(-1)^{T\cdot \textbf{d}}=(-q)^{S\cdot \textbf{d}}$.
\end{proof}

As we saw in Section \ref{symsec}, we could also write a single formula for $Q$ as 
\[
Q_{ij}=(2\Phi^2-\Psi_{X^+}^2)([\gamma_{j,i}^2])+\Delta.
\]

Now, we proceed to the proof of Theorem \ref{bigthm}.

\section{Proof of Main Theorem}
\label{pfsec}


The proof of Theorem \ref{bigthm} consists of showing that (\ref{calPtuvformula}), computed by the formulas in the theorem statement, gives the correct $\langle\langle\tau_{u/v}\rangle\rangle_j$ for each $j$. Since we are only care about the Poincar\'e polynomials $\langle\langle\tau_{u/v}\rangle\rangle_j$ up to some overall shift in the gradings, it suffices to show that, once we pair the generators of the $j$-colored complex with the terms in the sum
\begin{multline}
\label{Poincarej}
\sum_{\substack{\bf{d}=(d_1,...,d_{u+v})\in \mathbb{N}^{u+v}\\ d_1+...+d_{u+v}=j}}q^{S\cdot\textbf{d}+\textbf{d}\cdot Q\cdot \textbf{d}^t}a^{A\cdot \textbf{d}} t^{T\cdot\textbf{d}} {d_1+...+d_u\brack d_1,...,d_u}{d_{u+1}+...+d_{u+v}\brack d_{u+1},...,d_{u+v}}\\ \times X[d_1+...+d_{u+v},d_1+...+d_u],
\end{multline}
 we have that (\ref{Poincarej}) gives the correct grading differences between all generators, which can be computed using the methods of Wedrich in \cite{W16}.

First, recall that if $\overline d\in G_{\overline d}$ is the unique generator of lowest $q$-grading in $G_{\overline d}$ and $\textbf{c}(\overline{d})=\sum_{i=1}^{u+v}d_i\xi_i$, then we can represent $\overline d$ by $\textbf{d}=(d_1,...,d_{u+v})\in\mathbb{N}^{u+v}.$ Consequently, the individual term in (\ref{Poincarej}) corresponding to $\overline d$ is given by
\[
q^{S\cdot \textbf{d}+\textbf{d}\cdot Q\cdot\textbf{d}^t}a^{A\cdot \textbf{d}}t^{T\cdot \textbf{d}}\times X[j,d_1+...+d_u],
\]
and the sum of over all terms in $G_{\overline d}$ is given by
\[
q^{S\cdot \textbf{d}+\textbf{d}\cdot Q\cdot\textbf{d}^t}a^{A\cdot \textbf{d}}t^{T\cdot \textbf{d}}{d_1+...+d_u\brack d_1,...,d_u}{d_{u+1}+...+d_{u+v}\brack d_{u+1},...,d_{u+v}}\times X[j,d_1+...+d_u],
\]
by Lemma \ref{genblocks}. Thus, we can think of  (\ref{Poincarej}) as summing over the $\overline{d}=\iota(d)$ for all $d\in\text{Sym}^j(\mathcal{G}_{u/v}^1)$. Additionally, it follows now from Lemma \ref{genblocks} that we only need to compare the $q$-, $a$-, and $t$-grading differences between any two of these generators $\overline{d}$ and $\overline{d'}$ representing the lowest $q$-grading generators of sets $G_{\overline{d}}$  and $G_{\overline{d'}}$, as points within these sets have grading differences encoded by the quantum multinomials.


Now, suppose $\overline{d}\in G_{\overline{d}}$ and $\overline{d'}\in G_{\overline{d'}}$ are the unique generators of lowest $q$-grading in their respective sets of generators. As we just explained, although $\overline{d}$ and $\overline{d'}$ are $j$-tuples, they are uniquely determined by $(u+v)$-tuples given by the $d_i$'s in $\textbf{c}(\overline{d})=\sum_{i=1}^{u+v}d_i\xi_i$ (and similarly for $\overline{d'}$). Now, we think of them as vectors 
    \begin{align*}
    \overline{d} &\mapsto \vec{d}=\langle d_1,...,d_{u+v}\rangle\\
    \overline{d'} & \mapsto \vec{d'} = \langle d_1',...,d_{u+v}' \rangle.
    \end{align*}
It suffices to show that the formulas for $S,A,T,$ and $Q$ in Theorem \ref{bigthm} give the appropriate grading differences for $\overline{d}$ and $\overline{d'}$ when $\vec{d}-\vec{d'}=\pm\langle 0,...0,1,0,...0,-1,0,...,0\rangle$, i.e. where the vectors $\vec{d}$ and $\vec{d'}$ differ by $e_i-e_j$ for $1\leq i,j\leq u+v$ ($i\neq j)$, where $e_k$ is the vector of length $u+v$ with a $1$ in the $k$th coordinate and 0's elsewhere. Arbitrary $\overline{d}$ and $\overline{d'}$ are related by a finite sequence of these differences.

    Now, a simple calculation shows that, if $\vec{d}-\vec{d'}=e_i-e_j$, then
    \begin{align}
    \label{qdiffeq}
        &q(\overline{d})-q(\overline{d'}) = S_i-S_j+Q_{ii}+Q_{jj}-2Q_{ij}+2\sum_{k=1}^{u+v} d_k' (Q_{ik}-Q_{jk})\\
        \label{adiffeq}
        &a(\overline{d})-a(\overline{d'}) =A_i-A_j \\
        \label{tdiffeq}
        &t(\overline{d})-t(\overline{d'}) =T_i-T_j.
    \end{align}
    In fact, it suffices to consider the situation where the $i$ and $j$ correspond to $\xi_i, \xi_j \in \mathcal{G}_{u/v}^1$ with $\xi_i$ and $\xi_j$ consecutive with respect to the $\prec$ ordering. In other words, we are going to compare the gradings of $\overline{d}$ and $\overline{d'}$ by sliding points one at a time to consecutive intersection points. This results in three cases to consider. In particular, consecutive intersection points in $\mathcal{D}(\tau_{u/v})$ are connected to each other by $\alpha$ in one of the three ways shown in Figure \ref{3cases}; in each case, we need to show that if $\overline d$ and $\overline{d'}$ are such that $\vec{d}-\vec{d'}=e_i-e_j$ for $\xi_i$ and $\xi_j$ the intersection points represented by the case, then Equations (\ref{qdiffeq})-(\ref{tdiffeq}) give the same same grading differences as \cite{W16}.

   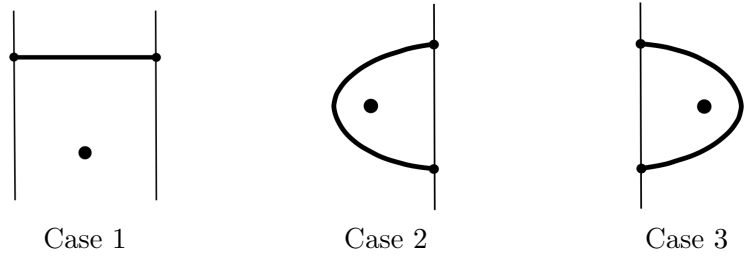
\begin{figure}
       \begin{tikzpicture}
           \node at (0,0) {\includegraphics[height=3cm]{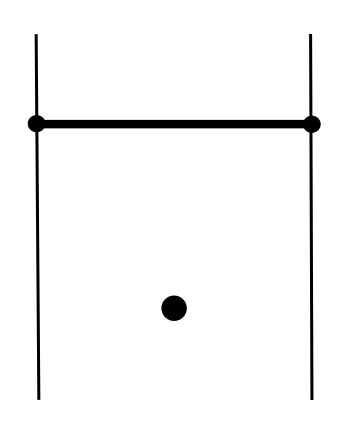}};
           \node at (4,0) {\includegraphics[height=3cm, angle=0]{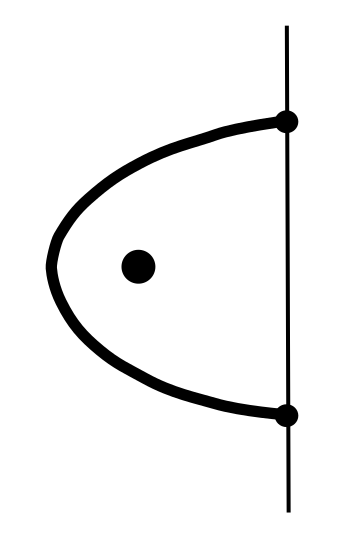}};
           \node at (8,0) {\includegraphics[height=3cm, angle=0]{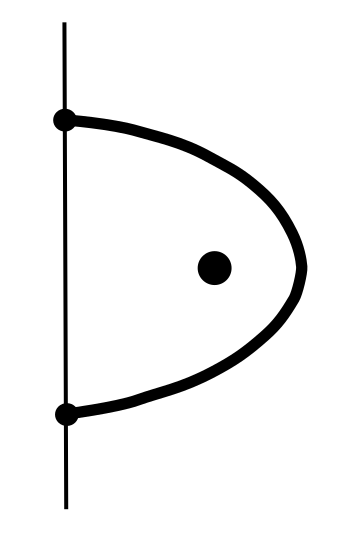}};
           \node at (0,-1.75) {Case $1$};
           \node at (4,-1.75) {Case $2$};
           \node at (8,-1.75) {Case $3$};
       \end{tikzpicture}
       \caption{The three main cases for the proof of Theorem \ref{bigthm}}
   \label{3cases}
   \end{figure}

   The first picture represents the case where we are sliding one of the points from the active side to the consecutive point on the inactive side (or vice versa) by a simple slide, thus changing the weight of the generators. Recall that we use $Z$ to denote the middle of the three punctures for $\tau_{u/v}$, represented by the unlabeled point in the image for the first case. The other two pictures correspond to the cases where we are sliding a point to a consecutive point on the same (active or inactive) side, thus preserving the weight. The additional unlabeled point can be $X^+,X^-,$ or $Y$, and it will be labeled $W$.

\subsection{Some More Loops in $\text{Conf}^2(M)$}

Recall that we defined the loop $\gamma_{i,j}:[0,1]\to\text{Conf}^2(M)$ (dropping the tilde) to be a loop at $\iota(\xi_i\xi_j)$, i.e. at $(x_1,x_2)$ where $c(x_1)=\xi_i$ and $c(x_2)=\xi_j$ if $i\leq j$ and vice versa if $i>j$. Alternatively, if $\xi_i$ and $\xi_j$ are both active or inactive intersections, we may define the loop $\widehat{\gamma_{i,j}}$ at $(x_2,x_1)$ under the above conditions, which is given by the same rule as $\gamma_{i,j}$. Two examples are shown in Figure \ref{T52hatgammafig}. These loops should be compared to the top two shown in Figure \ref{gammacases}.

\begin{figure}
\begin{tikzpicture}
    \node at (0,0) {\includegraphics[height=4cm]{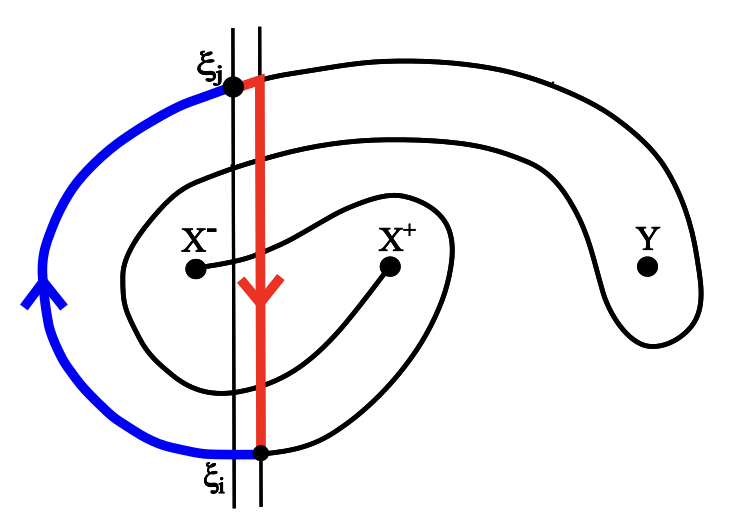}};
    \node at (7,0) {\includegraphics[height=4cm]{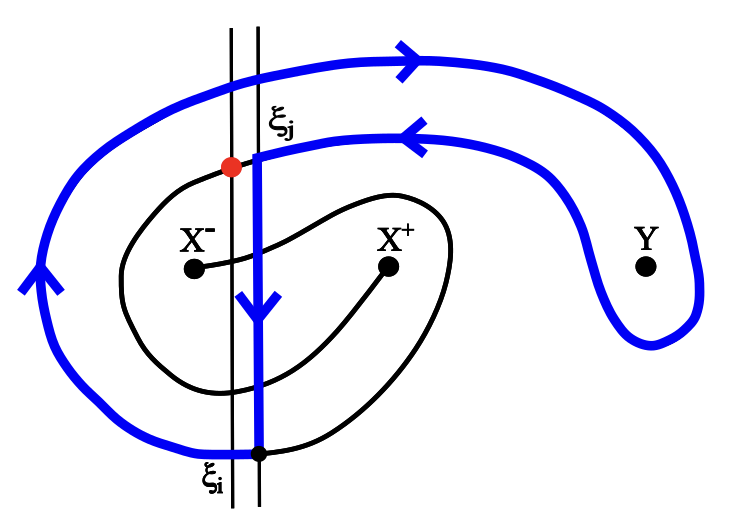}};
\end{tikzpicture}
\caption{Two example $\widehat{\gamma_{i,j}}$ loops for $\tau_{5/2}$}
\label{T52hatgammafig}
\end{figure}

We will define two more types of loops that will be helpful for proving our theorem. Define $s_{i,j}$ to be the loop at the same point as $\gamma_{i,j}$, shown in the image on the left in Figure \ref{sij}, and similarly, let $\widehat{s_{i,j}}$ be the loop at the same point as $\widehat{\gamma_{i,j}}$ shown in the image on the right in the same figure.

\begin{figure}
  \raisebox{0pt}{\includegraphics[height=3cm, angle=0]{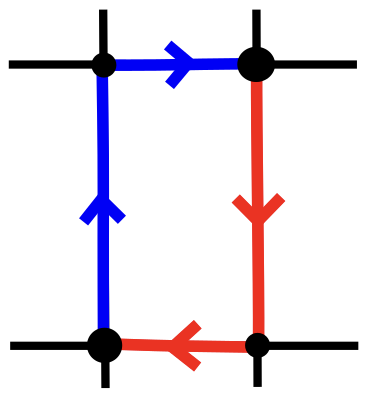}}  \qquad \qquad \raisebox{0pt}{\includegraphics[height=3cm, angle=0]{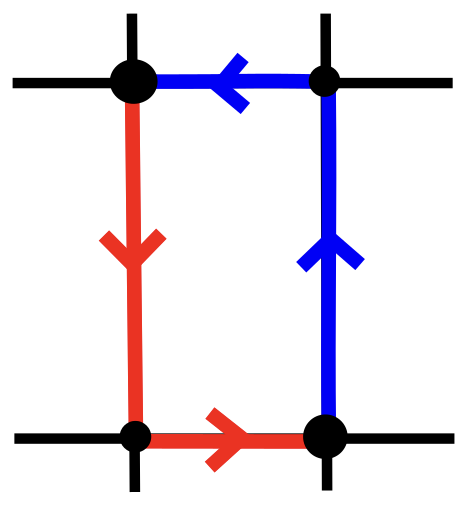}}
   \caption{The loops $s_{i,j}$ (left) and $\widehat{s_{i,j}}$ (right)}  
   \label{sij}
\end{figure}

\begin{lemmaN}
\label{phihat}
If $\xi_i$ and $\xi_j$ are both active or inactive intersection points, then
   \[
   \Phi([\gamma_{i,j}])= \Phi([\widehat{\gamma_{i,j}}])-1.
   \]
\end{lemmaN}

\begin{proof}
We will prove the lemma by showing that $s_{i,j}\cdot\gamma_{i,j} \sim \widehat{\gamma_{i,j}}$ or $\widehat{s_{i,j}}\cdot\widehat{\gamma_{i,j}} \sim \gamma_{i,j}$, depending on the behavior of $\gamma_{i,j}$, where $\sim$ denotes being homologous in $\text{Conf}^2(M)$. It suffices to show this because $\Phi$ factors through homology, so the lemma will follow since $\Phi([s_{i,j}])=1$ and $\Phi([\widehat{s_{i,j}}])=-1$.

The loops $\widehat{\gamma_{i,j}}$ and $\gamma_{i,j}$ only differ in a neighborhood of the two verticals and the portions of $\alpha$ that intersect them, and $s_{i,j}$ and $\widehat{s_{i,j}}$ are fully contained in such a neighborhood, so we only need to look at what happens there. We will consider all possible orientations of $\gamma_{i,j}$ at the (locally horizontal) parts of $\alpha$ where it intersects the two verticals, and we also need to consider whether $i<j$ or $j<i$. Thus, there are eight possible cases; four of them have $s_{i,j}\cdot\gamma_{i,j} \sim \widehat{\gamma_{i,j}}$ and the other four have $\widehat{s_{i,j}}\cdot\widehat{\gamma_{i,j}} \sim \gamma_{i,j}$.

First, we consider the four cases where $s_{i,j}\cdot\gamma_{i,j} \sim \widehat{\gamma_{i,j}}$. These are the cases where $\gamma_{i,j}$ splits into two non-trivial paths. The path components of $\gamma_{i,j}$ are colored in yellow and green, and the components of $s_{i,j}$ are colored the same way as in Figure \ref{sij}. In each case, it can easily be seen that $s_{i,j}\cdot\gamma_{i,j} \sim \widehat{\gamma_{i,j}}$. 

\[
\vcenter{\hbox{\includegraphics[height=3cm,angle=0]{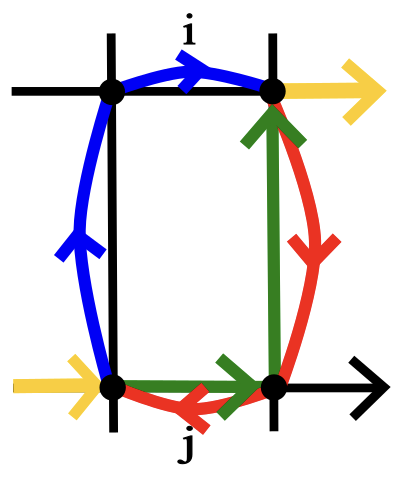}}} \qquad \qquad 
\vcenter{\hbox{\includegraphics[height=3cm,angle=0]{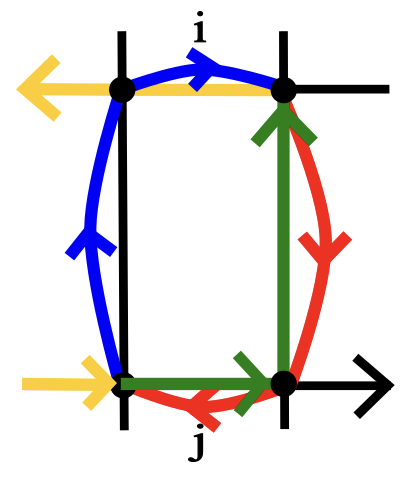}}} \qquad \qquad
\vcenter{\hbox{\includegraphics[height=3cm,angle=0]{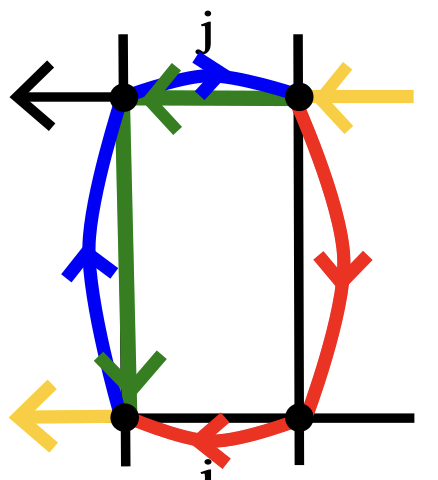}}} \qquad \qquad
\vcenter{\hbox{\includegraphics[height=3cm,angle=0]{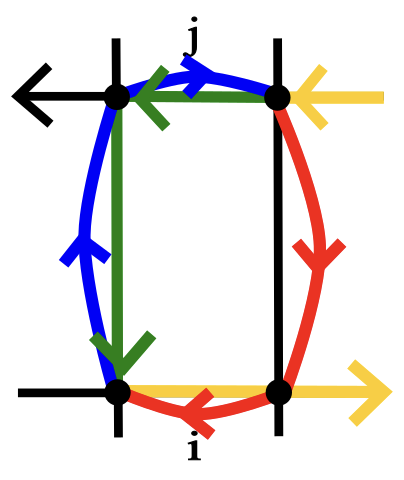}}}
\]

The remaining four cases satisfy $\widehat{s_{i,j}}\cdot\widehat{\gamma_{i,j}} \sim \gamma_{i,j}$, and they are shown below. We use yellow and green to denote the components of $\widehat{\gamma_{i,j}}$ and red and green to denote the components of $\widehat{s_{i,j}}$, similar to what we did before.

\[
\vcenter{\hbox{\includegraphics[height=3cm,angle=0]{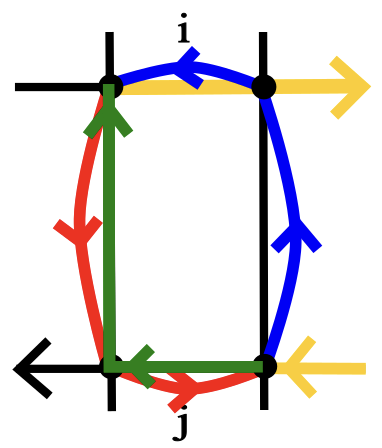}}} \qquad \qquad 
\vcenter{\hbox{\includegraphics[height=3cm,angle=0]{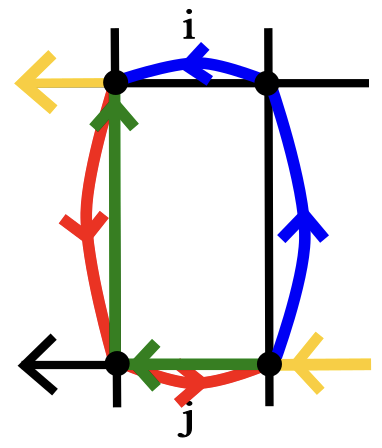}}} \qquad \qquad
\vcenter{\hbox{\includegraphics[height=3cm,angle=0]{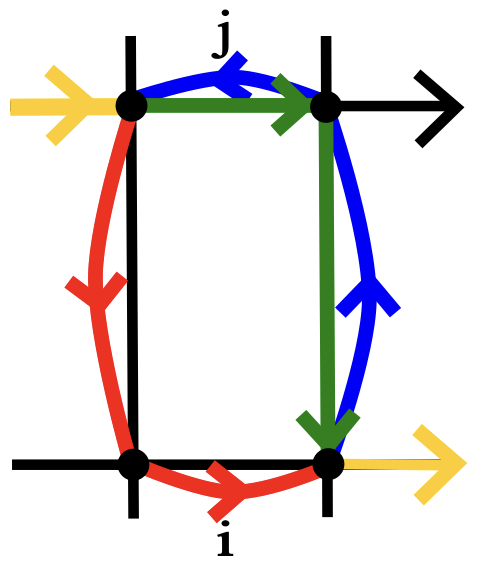}}} \qquad \qquad
\vcenter{\hbox{\includegraphics[height=3cm,angle=0]{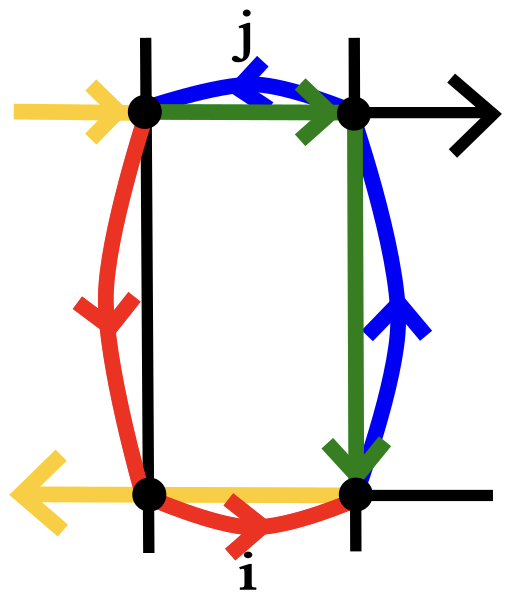}}}
\]

Since all cases satisfy $\Phi([\gamma_{i,j}])=\Phi([\widehat{\gamma_{i,j}}])-1$, the lemma follows.
\end{proof}

 By definition of the homomorphism $\Phi$, it factors through $H_1(\text{Conf}^2(\mathbb{C}))$, so $\Phi$ must be equal on loops that are homologous in $\text{Conf}^2(\mathbb{C})$. This is, in fact, tautological because the image group of $\Phi$ is isomorphic to $H_1(\text{Conf}^2(\mathbb{C}))$. This observation will be useful for computing values such as $\Phi([\gamma_{i,k}])-\Phi([\gamma_{j,k}])$ and $\Phi([\gamma_{k,i}])-\Phi([\gamma_{k,j}])$ when proving the three cases.

 Now, we proceed to prove Theorem \ref{bigthm} by considering the three cases shown in Figure \ref{3cases}, one at a time.

 \subsection{Proof of Case 1}

Now, we consider the case shown on the left in Figure \ref{3cases}, where we want to compare the grading differences between to generators that are related by a simple slide. In particular, we want to show that the grading differences predicted by Theorem \ref{bigthm} agree with the grading differences from \cite{W16}.

\begin{lemmaN}
\label{case1bigpf}
    Let $\xi_i$ and $\xi_j$ be consecutive with respect to the $\prec$ ordering with one on the active side and the other on the inactive side, as in the left-most image in Figure \ref{3cases}. Then, if $\vec{d}-\vec{d}'=e_i-e_j$, then the formulas in Theorem \ref{bigthm} give the correct grading differences between $\overline{d}$ and $\overline{d'}$.
\end{lemmaN}

\begin{proof}
Suppose $\overline{d}=(x_1,...,x_k,x_{k+1},...,x_{u+v})$ is the generator of weight $k$ and we want to compare its gradings with $\overline{d'}=(x_1',...,x_{k-1}',x_k',...,x_{u+v}')$, the generator of weight $k-1$ such that $\vec{d}-\vec{d}'= e_i-e_j$ with $i\leq u$, $j>u$, and $\xi_i$ and $\xi_j$ are consecutive with respect to the $\prec$ ordering. Thus, $x_l'=x_l$ if $l\leq\sum_{k\leq i} d_k'$ or $l>\sum_{k\leq {j-1}} d_k'+1$, $x_l'=x_{l+1}$ if $\sum_{k\leq i}d_k' < l < \sum_{k\leq {j-1}} d_k'+1$, and, if $l=\sum_{k\leq {j-1}} d_k'+1$, then $x_l$ is such that $c(x_l)=\xi_j$. 

       We can think of this geometrically as sliding the point $x_l$ of $\overline{d}$ for $l=\sum_{k\leq i} d_k$ from the active side to the consecutive point on the inactive side to get $\overline{d'}$. First, we need to think about what the grading differences should be according to \cite{W16}. The rule for sliding a point from the active to the inactive side requires that the point be on the $k$th vertical, for $k$ the weight, so first need to slide $x_l$ (as given at the beginning of the paragraph) to the $k$th vertical, which we can do one step at a time by sequentially sliding the $x_m$ for $l<m\leq k$ to the left as we slide our point $x_l$ to the right; each time this contributes a $-2$ to the $q$-grading difference and 0 to the $a$- and $t$-grading differences (see Figure \ref{rectangle} and the argument using it in the proof of Lemma \ref{uniquegr}). Then, we can slide this point from the active side to the inactive side with a simple slide, which provides grading shifts as given in Section \ref{compmethods}. Then we must slide this point to the appropriate position on the inactive side to get $\overline{d'}$, which contributes $+2$ to the $q$-grading difference at each step and nothing to the other grading differences. Recall that we use $Z$ to denote the point between the active and inactive sides. 

       Figure \ref{bigthms1fig} can help with visualizing what was described in the previous paragraph. The labels on the sides denote the subscripts of the active and inactive intersections, with the portion of $\alpha$ included that relates $\xi_i$ and $\xi_j$ by a simple slide. The red and blue points represent the types of points that contribute a $-2$ or a $+2$, respectively, to the $q$-grading difference in the procedure just described.
       
       \begin{figure}
           \centering
           \includegraphics[height=5cm]{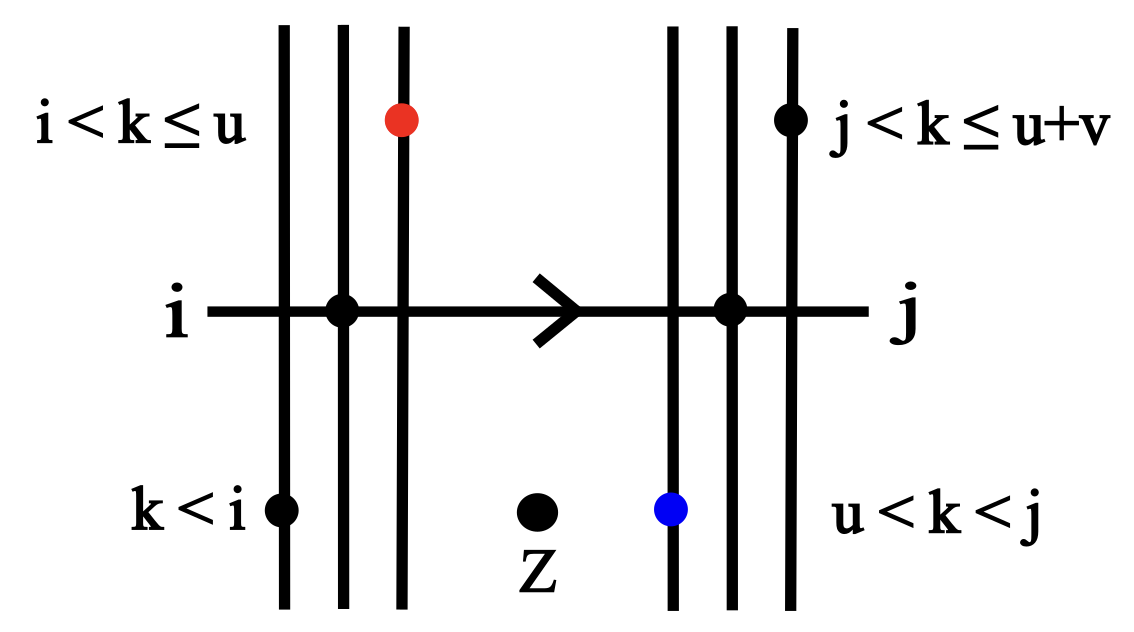}
           \caption{Geometric set-up for main case considered in proof of Lemma \ref{case1bigpf}}
           \label{bigthms1fig}
       \end{figure}
       
       Algebraically, the procedure described above demonstrates that \cite{W16} gives
       \[
       \begin{cases}
           q(\overline{d})-q(\overline{d'})=1-2\sum_{i<k\leq u}d_k'+2\sum_{u+1\leq k<j}d_k'\\
           a(\overline{d})-a(\overline{d'})=0 \qquad \qquad\qquad\qquad\qquad\qquad\qquad\qquad\qquad \text{if} \,Z\in\{X^-,Y\} \\
           t(\overline{d})-t(\overline{d'})=-1
       \end{cases}
       \]
       and 
       \[
       \begin{cases}
           q(\overline{d})-q(\overline{d'})=1-2j-2\sum_{i<k\leq u}d_k'+2\sum_{u+1\leq k<j}d_k'\\
           a(\overline{d})-a(\overline{d'})=1 \qquad \qquad\qquad\qquad\qquad\qquad\qquad\qquad\qquad \text{if} \,Z=X^+. \\
           t(\overline{d})-t(\overline{d'})=-1
       \end{cases}
       \]

Now, we need to check that our formulas from Theorem \ref{bigthm} give these grading differences as well. We will first check the case where $Z \in \{X^-,Y\}$.

According to Theorem \ref{bigthm}, we clearly get $S_i-S_j=1$, $A_i-A_j=0$, and $T_i-T_j=-1$, since $\Psi_W([\gamma_i])=\Psi_W([\gamma_j])$ for any $W\subset \{X^+,X^-,Y\}$. This also gives $Q_{ii}=Q_{jj}$. Additionally, $\Phi([\gamma_{i,j}])=0$, so $Q_{ij}=Q_{ji}=Q_{jj}=Q_{ii}$. Thus, we already have the correct $a$- and $t$-grading differences according to Equations (\ref{adiffeq}) and (\ref{tdiffeq}), and Equation (\ref{qdiffeq}) gives
\[
q(\overline{d})-q(\overline{d'})=1+2\sum_{k=1}^{u+v} d_k' (Q_{ik}-Q_{jk}).
\]
Thus, it suffices to show
\[
Q_{ik}-Q_{jk} = \begin{cases}
    0 \quad &\text{if}\, k\leq i \,\text{or}\, k\geq j \\
-1 \quad &\text{if}\, i <k \leq u \\
1 \quad &\text{if}\, u+1 \leq k <j
\end{cases}
\]
Now, we show $Q_{ik}-Q_{jk}=0$ if $k\leq i$. The argument is similar if $k\geq j$. The formula for $Q$ in Theorem \ref{bigthm} gives
\[
Q_{ik}-Q_{jk}=Q_{ii}-Q_{jj}+(\Phi-2\Psi_{X^+})([\gamma_{k,i}])-(\Phi-2\Psi_{X^+})([\gamma_{k,j}]).
\]
It is not difficult to see that $\gamma_{k,i}\sim\gamma_{k,j}$ for how they are defined as loops in both $M$ and $\text{Conf}^2(M)$, which means $\Psi_{X^+}([\gamma_{k,i}])=\Psi_{X^+}([\gamma_{k,j}])$ and $\Phi([\gamma_{k,i}])=\Phi([\gamma_{k,j}])$. Then, since $Q_{ii}=Q_{jj}$, we get $Q_{ik}-Q_{jk}=0$.

 Now, if $i < k \leq u$, we want to show $Q_{ik}-Q_{jk}=-1$. As before, it is true that in this case and in the one we will next consider, the contributions of $\Psi_{X^+}$ are the same for both loops, so we will drop these terms since they cancel. This means  
 \[
 Q_{ik}-Q_{jk}=\Phi([\gamma_{k,i}])-\Phi([\gamma_{k,j}]).
 \]
 For this case, one can check that $\gamma_{k,j}\sim \widehat{\gamma_{k,i}}$, so $\Phi([\gamma_{k,j}])=\Phi([\gamma_{k,i}])+1$, which gives $\Phi([\gamma_{k,i}])-\Phi([\gamma_{k,j}])=-1$. Thus, we get $Q_{ik}-Q_{jk}=-1$. If $\xi_i \prec \xi_j \prec \xi_k$, then the same holds true. 

 If $u+1 \leq k <j$, then we can do something similar. Here, one gets $\gamma_{k,i}\sim \widehat{\gamma_{k,j}}$, so $\Phi([\gamma_{k,i}])-\Phi([\gamma_{k,j}])=1$, so we indeed get $Q_{ik}-Q_{jk}=1$.

Next, we need to address what happens when $Z=X^+$. Much of the previous work still applies here, as everything about the differences of $\Phi$ and $\Psi_{X^+}$ values still hold true, but there are also some slight changes. In particular, $Q_{ii}$ and $Q_{jj}$ are no longer equal, and we will need to utilize the $\delta_{Z,X^+}(\delta_{i,A}\delta_{j,I}-\delta_{j,A}\delta_{i,I})$ term for off-diagonal entries of $Q$. 

One can easily check that we have $S_i-S_j=1$, $A_i-A_j=1$, and $T_i-T_j=-1$ in this case, which already accounts for the $a$- and $t$-gradings. It is also easy to see that $Q_{jj}=Q_{ii}+2$ and $Q_{ij}=Q_{jj}-1$ due to the $\delta$ terms. Thus, Equation (\ref{qdiffeq}) once again gives
\[
q(\overline{d})-q(\overline{d'})=1+2\sum_{k=1}^{u+v} d_k' (Q_{ik}-Q_{jk}),
\]
so, by our previous calculation using the methods of \cite{W16}, it suffices to show
\[
Q_{ik}-Q_{jk} = \begin{cases}
    -1 \quad &\text{if}\, k\leq i \,\text{or}\, k\geq j \\
-2 \quad &\text{if}\, i <k \leq u \\
0 \quad &\text{if}\, u+1 \leq k <j
\end{cases}
\]

Theorem \ref{bigthm} gives
\[
Q_{ik}-Q_{jk}=-2+(\Phi-2\Psi_{X^+})([\gamma_{k,i}])-(\Phi-2\Psi_{X^+})([\gamma_{k,j}]) +\delta,
\]
where $\delta$ indicates the difference of the $\delta_{Z,X^+}(\delta_{i,A}\delta_{j,I}-\delta_{j,A}\delta_{i,I})$ terms. Now, we need to show that this gives the right value for each $k$.

If $k\leq i$, then we already know $\Phi([\gamma_{k,i}])-\Phi([\gamma_{k,j}])=0$ and $\Psi_{X^+}([\gamma_{k,i}])-\Psi_{X^+}([\gamma_{k,j}])=0$. Additionally, $\delta=1$ since $\xi_i$ and $\xi_k$ are active intersections and $\xi_j$ is inactive, so
$Q_{ik}-Q_{jk}=-1$. Similarly, we can calculate $Q_{ik}-Q_{jk}$ for all other possible $k$'s, by using the same differences between $\Phi$ values that we calculated for when $Z\in\{X^-,Y\}$, and by noting that $\delta=1$ for all $k$.

Finally, observe that if we were instead to consider the case where $\vec{d}-\vec{d'}=e_i-e_j$ with $j\leq p$ and $i>p$ (i.e. sliding a point from the inactive to the active side), this just flips the signs of everything, which is what we should expect.
\end{proof}

Thus, we have completed the proof of the first case.

\subsection{Proof of Case 2}

Now, we move on to consider the second main case for the theorem, given by the middle picture in Figure \ref{3cases}.

\begin{lemmaN}
\label{case2bgpf}
    Let $\xi_i$ and $\xi_j$ be consecutive with respect to the $\prec$ ordering with both on the active or inactive side, as in the middle image in Figure \ref{3cases}. Then, if $\vec{d}-\vec{d}'=e_i-e_j$, then the formulas in Theorem \ref{bigthm} give the correct grading differences between $\overline{d}$ and $\overline{d'}$.
\end{lemmaN}

\begin{proof}
    We will look in depth at what happens when $\vec{d}-\vec{d'}=e_i-e_j$ with $1\leq i <j\leq u$. The situation where $j>i$ is the same, except with signs switched (as before), and if $u+1\leq i,j\leq u+v$, then the arguments are the exact same, but with the points on the active/inactive sides swapped.

    Once again, we may consider the sub-cases where the special point given by a dot in Figure \ref{3cases}, which we will call $W$, is $X^-$ or $Y$ at the same time, but then we must treat $W=X^+$ separately. As before, we check what the grading differences should be for $\overline{d}$ and $\overline{d'}$ according to \cite{W16}, and then we show that our formulas give this as well. Note that if $\overline{d}=(x_1,...,x_k,x_{k+1},...,x_{u+v})$, then $\overline{d'}$ satisfies $\overline{d'}=(x_1',...,x_k',x_{k+1}',...,x_{u+v}')$ with $x_l'$ related to the $x_l$ in the same way that we saw at the beginning of case 1. 

    Geometrically, we can think of going from $\overline{d}$ to $\overline{d'}$ by first sliding $x_l$, for $l=\sum_{k\leq j-1} d_k'+1$, along $\alpha$ to the point above on the same vertical Lagrangian, which gives a grading difference of $q^2/t$ (written multiplicatively) if $W\in \{X^-,Y\}$, or $a^2/tq^{4j-2}$ if $W=X^+$. After this, we need to slide the point further to the right to get it on the correct vertical for $\overline{d'}$, which contributes a $+2$ to the $q$-grading for each of the $\sum_{i\leq k < j}d_k'$ points we must sequentially slide to the left in order to do so. 

    As in the proof of Lemma \ref{case1bigpf}, one can use Figure \ref{bigthms2fig} to help visualize the procedure just described. The blue point represents the $\sum_{i<k<j}d_k'$ points between $\xi_i$ and $\xi_j$ that contribute a $+2$ to the grading difference.
    
    \begin{figure}
        \centering
        \includegraphics[height=6cm]{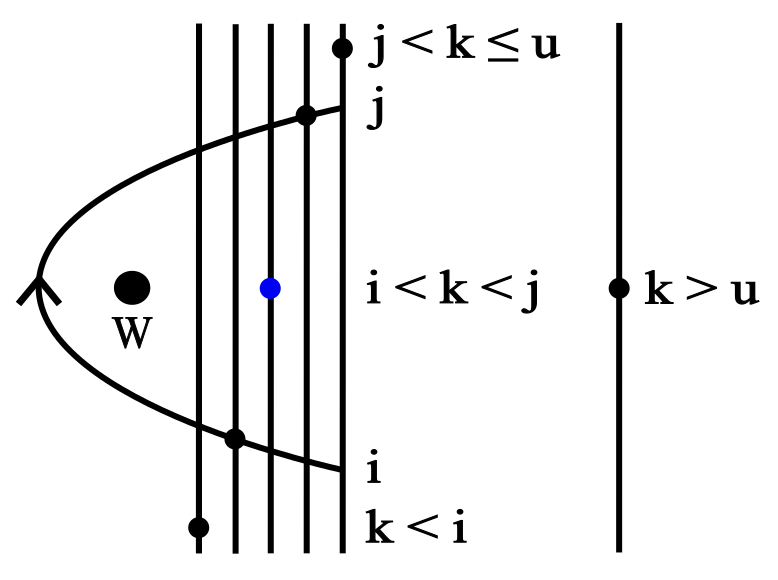}
        \caption{Geometric set-up for main case considered in proof of Lemma \ref{case2bgpf}}
        \label{bigthms2fig}
    \end{figure}
    
    This time, we get

    \[
       \begin{cases}
           q(\overline{d})-q(\overline{d'})=2+2\sum_{i\leq k<j}d_k'\\
           a(\overline{d})-a(\overline{d'})=0 \qquad \qquad\qquad\qquad\qquad\qquad\qquad\qquad\qquad \text{if}\, \,W\in\{X^-,Y\} \\
           t(\overline{d})-t(\overline{d'})=-1
       \end{cases}
       \]
       and 
       \[
       \begin{cases}
           q(\overline{d})-q(\overline{d'})=2-4j+2\sum_{i\leq k<j}d_k'\\
           a(\overline{d})-a(\overline{d'})=2 \qquad \qquad\qquad\qquad\qquad\qquad\qquad\qquad\qquad \text{if} \,\,W=X^+ \\
           t(\overline{d})-t(\overline{d'})=-1
       \end{cases}
       \]
       by following the rules from \cite{W16} for computing grading differences.
       
    Now, suppose $W\in\{X^-,Y\}$. Theorem \ref{bigthm} clearly gives $S_i-S_j=1$, $A_i-A_j=0$, and $T_i-T_j=-1$. This tells us that the $a$- and $t$-grading differences are already correct, so we need only consider the $q$-grading. We also get $Q_{ii}=Q_{jj}+1$, and then $Q_{ij}=Q_{jj}$ since $(\Phi-2\Psi_{X^+})([\gamma_{i,j}])=0$, which can easily be checked. Thus, we get
    \[
    q(\overline{d})-q(\overline{d'})= 2+2\sum_{k=1}^{u+v}d_k'(Q_{ik}-Q_{jk})
    \]
    from Equation (\ref{qdiffeq}), so what we need to show is 
    \[
Q_{ik}-Q_{jk} = \begin{cases}
    1 \quad &\text{if}\, i\leq k <j \\
0 \quad &\text{otherwise} \\
\end{cases}
\]

We have $Q_{ii}-Q_{jj}=1$ because $\Psi_W([\gamma_{i,j}])=1$, which means
\[
Q_{ik}-Q_{jk}=1+(\Phi-2\Psi_{X^+})([\gamma_{k,i}])-(\Phi-2\Psi_{X^+})([\gamma_{k,j}]).
\]
Clearly, $\Psi_{X^+}([\gamma_{k,i}])=\Psi_{X^+}([\gamma_{k,j}])$ since $W\neq X^+$, so the right-hand side reduces to $1+\Phi([\gamma_{k,i}])-\Phi([\gamma_{k,j}])$. Now, for comparing the contributions from $\Phi$, we first consider the case where $i \leq k <j$. In this case, one finds $\gamma_{k,i}\sim \gamma_{k,j}$ as 1-chains in $\text{Conf}^2(\mathbb{C})$, so $\Phi([\gamma_{k,i}])-\Phi([\gamma_{k,j}])=0$ because of how $\Phi$ factors through $\text{H}_1(\text{Conf}^2(\mathbb{C}))$. For any other $k$, it can be checked that $\gamma_{k,i} \sim \widehat{\gamma_{k,j}}$, so Lemma \ref{phihat} gives $\Phi([\gamma_{k,i}])-\Phi([\gamma_{k,j}])=-1$. Thus, we get all the correct $Q_{ik}-Q_{jk}$ values when $W \in \{X^-,Y\}$.

If $W=X^+$, all that we proved already about the contributions to $Q_{ik}-Q_{jk}$ coming from $\Phi$ still hold, but we need to check the differences coming from $\Psi_{X^+}$, which are no longer 0. First, note that Theorem \ref{bigthm} gives $S_i-S_j=1$, $A_i-A_j=2$, and $T_i-T_j=-1$. As usual, this gives us the $a$- and $t$-gradings already. However, the $\Psi_{X^+}$ contributions give $Q_{ii}=Q_{jj}-3$ and $Q_{ij}=Q_{jj}-2$, so we have
 \[
 S_i-S_j+Q_{ii}+Q_{jj}-2Q_{ij}=2.
 \]
 Thus, setting Equation (\ref{qdiffeq}) equal to the predicted value for $q(\overline{d})-q(\overline{d'})$ from \cite{W16}, we get 
\[
q(\overline{d})-q(\overline{d'})=2+2(-2j+\sum_{i\leq k<j}d_k')=2+2\sum_{k=1}^{u+v}d_k'(Q_{ik}-Q_{jk}),
\]
which means we must have 
 \[
Q_{ik}-Q_{jk} = \begin{cases}
    -1 \quad &\text{if}\, i\leq k <j \\
-2 \quad &\text{otherwise}. \\
\end{cases}
\]
if Theorem \ref{bigthm} is true. This shift by $-2$ is accounted for by the obvious fact that $\Psi_{X^+}([\gamma_{k,i}])-\Psi_{X^+}([\gamma_{k,j}])=-\Psi_{X^+}([\gamma_{i,j}])=-1$, which contributes a $+2$ to $Q_{ik}-Q_{jk}$, but $Q_{ii}-Q_{jj}$ is four less than in the case where $W\in\{X^-,Y\}$, so this still gives a shift in $Q_{ik}-Q_{jk}$ by $-2$.  
\end{proof}

\subsection{Proof of Case 3}

Finally, we address the final case, shown on the right in Figure \ref{3cases}, which follows trivially from Case 2.

\begin{lemmaN}
    Let $\xi_i$ and $\xi_j$ be consecutive with respect to the $\prec$ ordering with both on the active or inactive side, as in the right-most image in Figure \ref{3cases}. Then, if $\vec{d}-\vec{d}'=e_i-e_j$, then the formulas in Theorem \ref{bigthm} give the grading differences between $\overline{d}$ and $\overline{d'}$.
\end{lemmaN}

\begin{proof}
    The proof for this case is essentially the same as the one for Lemma \ref{case2bgpf}.
\end{proof}

We have now proved Theorem \ref{bigthm} since we have shown that the formulas correctly predict grading differences. Next, we will see how to reduce the number of indices in the sum. 

\section{Reducing the Number of Indices}
\label{indreducesec}

\subsection{Partitions of $\mathcal{G}_{u/v}^1$}
The previous two sections were devoted to giving a geometric way to compute $\mathcal{P}(\tau_{u/v})$ in terms of a sum defined over $u+v$ indices. However, it can be advantageous to reduce this number of indices, which can be done under certain conditions at the expense of adding Pochhammer symbols. Using the language of Section \ref{quiversec}, reducing the number of indices in $\mathcal{P}(\tau_{u/v})$ corresponds to writing it in almost quiver form. In \cite{JHpII26}, the sequel to this paper, we will need to use a version of $P(\tau_{u/v})$ written in almost quiver form to prove the main theorem for rational knots, following the general proof structure of Sto\v si\'c and Wedrich in \cite{SW21}, but we will still need the geometric interpretation of the linear and quadratic forms. In this section, we will study how to reduce the number of indices in the generating function while operating within our geometric framework. We begin with the following lemma.

\begin{figure}
   
 \raisebox{0pt}{\includegraphics[height=3.5cm, angle=0]{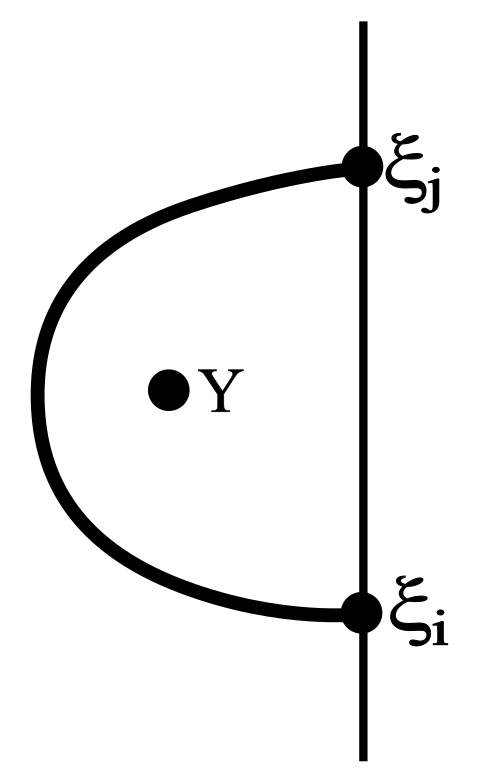}}\qquad \qquad \raisebox{0pt}{\includegraphics[height=3.5cm, angle=0]{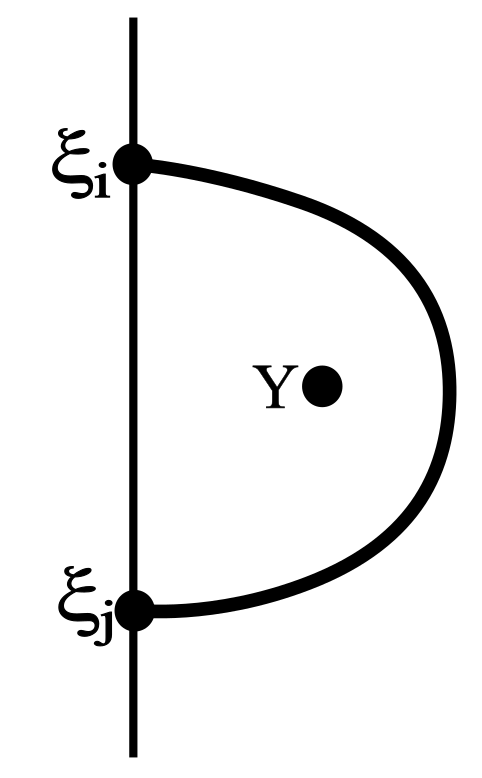}}
\caption{The two cases where consecutive $\xi_i,\xi_j\in \alpha\cap l$ (for $l\in\{l_
A,l_I\}$) wrap around $Y$.}
\label{Ydisk}
\end{figure}

\begin{lemmaN}
\label{singlepair}
    Suppose that $\xi_i,\xi_j\in \alpha \cap l$, where $l\in\{l_A,l_I\}$, such that $\xi_i$ and $\xi_j$ are related as in Figure \ref{Ydisk} (on the left or right). Then, $S_i=S_j+1$, $A_i=A_j$, $T_i=T_j-1$, and $Q_{ii}=Q_{jj}+1$.
\end{lemmaN}

This lemma follows immediately from the work in the proof of Lemma \ref{case2bgpf} for the case involving the picture on the left in Figure \ref{Ydisk}, and it is similar for the picture on the right.




Now, we will extend this lemma to apply to collections of pairs of points that satisfy the properties of Lemma \ref{singlepair}. To avoid double subscripts, we will use $x$'s and $y$'s, with the $x$'s corresponding to points like $\xi_i$ and the $y$'s corresponding to points like $\xi_j$ as given above.

\begin{lemmaN}
\label{indpartition}
    Partition the set of intersection points of $\alpha$ with $l_A$ and $l_I$ as $\mathcal{G}_{u/v}^1=\mathfrak{X}\sqcup \mathfrak{Y}\sqcup \mathfrak{Z}$, where $\mathfrak{X}=\{x_1,...,x_k\}$ and $\mathfrak{Y}=\{y_1,...,y_k\}$ are such that for each $i\in\{1,...,k\}$, $x_i$ and $y_i$ satisfy the conditions of Lemma \ref{singlepair}, and $k$ is as large as possible. Then, there is an ordering of the index points in $\mathfrak{X}$ and $\mathfrak{Y}$ such that the data for $\mathcal{P}(\tau_{u/v})$ is given by
    \[X, 
    \left[\begin{array}{ c | c |c }
    S_1+1 & A_1 & T_1-1 \\
    \hline
    S_1 & A_1 & T_1 \\
    \hline
    S_2 & A_2 & T_2
  \end{array}\right],  \left[\begin{array}{ c | c |c }
    Q_{11}+1 & Q_{11}+L & Q_{12} \\
    \hline
    Q_{11}+U & Q_{11} & Q_{12} \\
    \hline
    Q_{21} & Q_{21} & Q_{22}
  \end{array}\right],
    \]
    where $U$ and $L$ denote the strictly upper and lower triangular matrices of $1$'s of the appropriate size.
\end{lemmaN} 

It is worth noting, as suggested by Lemma \ref{singlepair}, that the blocks shown in Lemma \ref{indpartition} are ordered by $\mathfrak{X}$, then $\mathfrak{Y}$, and then $\mathfrak{Z}$.

\begin{proof}
    First, observe that some of the work is already done by Lemma \ref{singlepair}. As long as we consistently order the blocks $\mathfrak{X}$ and $\mathfrak{Y}$ (meaning $x_i$ and $y_i$ must always satisfy the property from Lemma \ref{singlepair}), then the linear forms $S, A,$ and $T$ are already accounted for. Some care must be taken, however, for dealing with $Q$. In fact, we will fix a particular ordering (distinct from the standard ordering) to ensure that we get the blocks with $L$ and $U$.

 \begin{figure}
 \raisebox{0pt}{\includegraphics[height=5cm, angle=0]{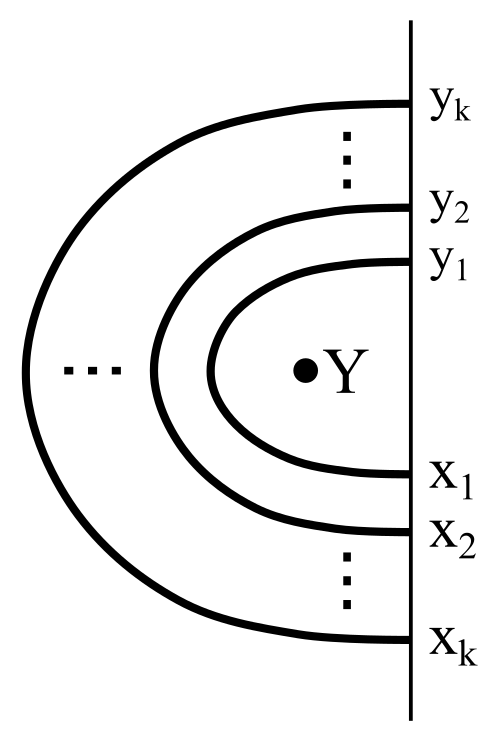}}\qquad \qquad \raisebox{0pt}{\includegraphics[height=5cm, angle=0]{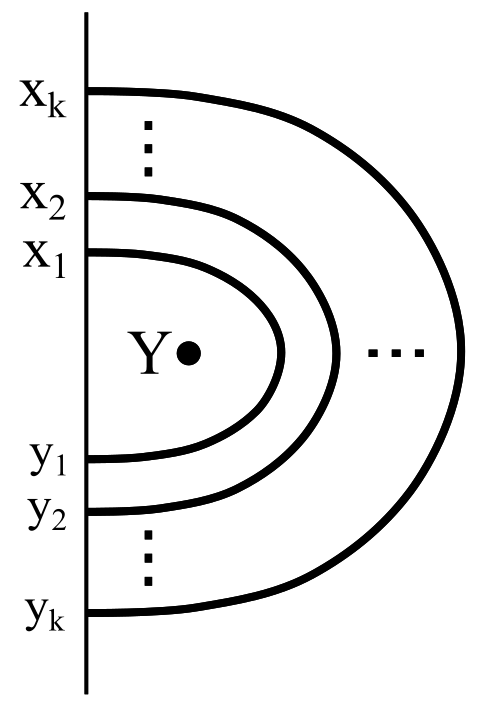}}
\caption{The desired ordering of $x_i$'s and $y_i$'s for the proof of Lemma \ref{indpartition}}
\label{xyorderfig}
\end{figure} 

As figure \ref{xyorderfig} indicates, we want to order the points from the inside-out; in the case of the puncture $Y$ being to the left of the vertical $l$, we have the $x_i's$ on the bottom, and when $Y$ is to the right of $l$, the $y_i's$ are on the bottom. We will prove the lemma explicitly for the case where $Y$ is to the left of $l$; the argument for the other case is the same.

In order to show that $Q$ has the desired structure with respect to this ordering in the top left $2\times 2$ blocks, it suffices to show that every $4\times 4$ sub-matrix coming from rows and columns indexed by some $x_i,x_j\in \mathfrak{X}$ and $y_i,y_j\in \mathfrak{Y}$ for $i<j$ have this same structure. To simplify notation, we will write $i$ instead of $x_i$, $i'$ instead of $y_i$, and similarly use $j$ and $j'$ instead of $x_j$ and $y_j$, respectively. Using this notation, the $4\times 4$ block we get from these indices should have the form 
\[
\begin{blockarray}{ccccc}
    & i & j & i' & j' \\
    \begin{block}{c[cccc]}
        i & A+1 & C+1 & A & C \\
        j & C+1 & B+1 & C+1 & B\\
        i' & A & C+1 & A & C\\
        j' & C & B & C & B\\
    \end{block}
\end{blockarray}
\]
for some $A,B,C\in \mathbb{Z}$. For the remainder of the proof, we will often use without stating explicitly that $Q$ is symmetric.

Setting $A=Q_{i'i'}$, $B=Q_{j'j'}$, and $C=Q_{i'j'}=Q_{j'i'}$, we already know $Q_{ii}=A+1$ and $Q_{jj}=B+1$ from Lemma \ref{singlepair}. Next, we confirm $Q_{ij}=Q_{ji}=C+1$. One can easily check that $\gamma_{i,j}\sim\gamma_{i',j'}$ as loops in $\text{Conf}^2(\mathbb{C})$ by checking this for the four possible ways that the intersection points can be connected by $\alpha$, which is shown in Figure \ref{connectivityfig}. In particular, $\gamma_{i,j}$ and $\gamma_{i',j'}$ agree on the dotted portions, so this part cancels in $\gamma_{i,j}-\gamma_{i',j'}$, when we think of them as 1-chains in $\text{Conf}^2(\mathbb{C})$, and then one easily checks that the remaining part is null-homologous in each case. Consequently, we have $\Phi([\gamma_{i,j}])=\Phi([\gamma_{i',j'}])$. Also, it is clear that $\Psi_{X^+}([\gamma_{i,j}])=\Psi_{X^+}([\gamma_{i',j'}])$. From here, it follows that $Q_{ij}=Q_{ji}=C+1$, or $Q_{ij}-Q_{i'j'}=1$, since $Q_{ii}-Q_{i'i'}=1$. 

\begin{figure}
    \begin{tikzpicture}
        \node at (0,0) {\includegraphics[height=3.8cm]{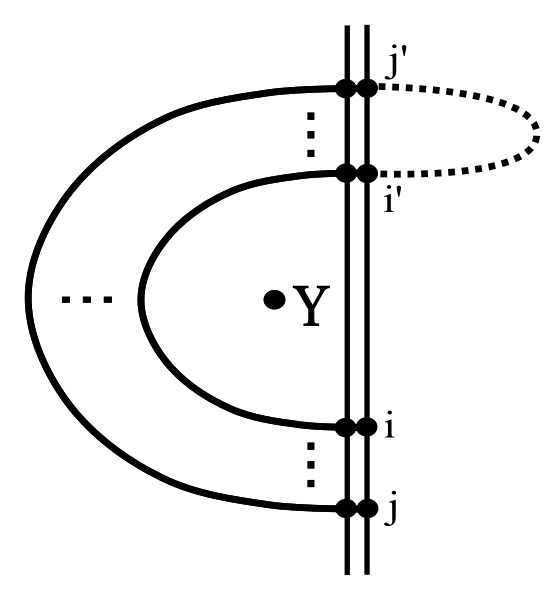}};
        \node at (4,0) {\includegraphics[height=3.8cm]{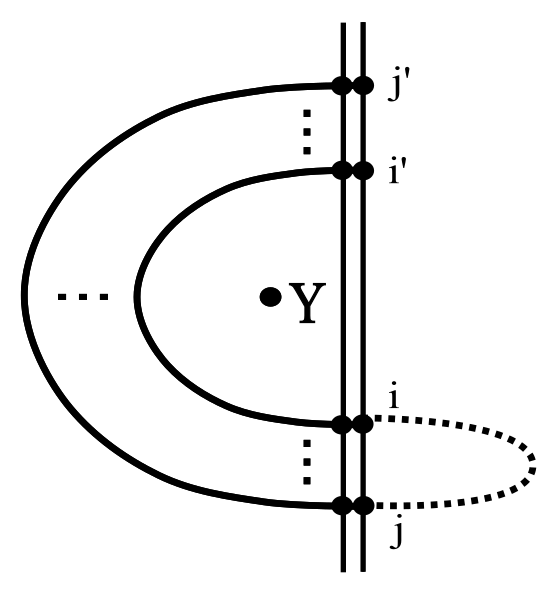}};
        \node at (8,0) {\includegraphics[height=3.8cm]{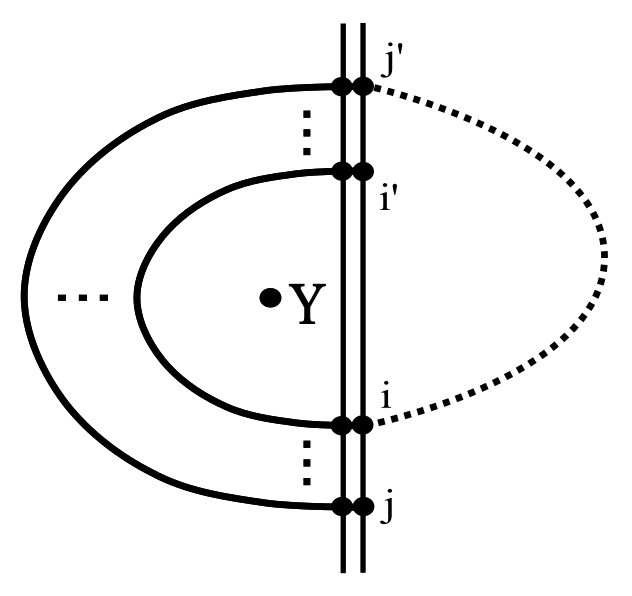}};
        \node at (12,0) {\includegraphics[height=3.8cm]{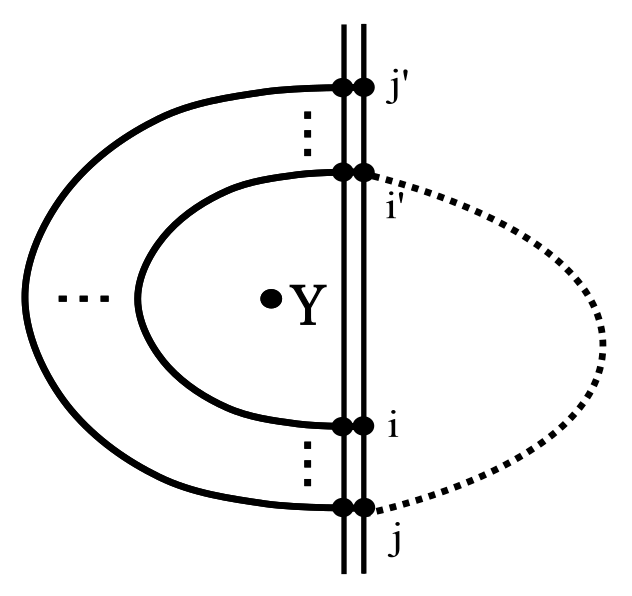}};
    \end{tikzpicture}
    \caption{The four possible ways that $\alpha$ can connect the four intersection points labeled by $i,i',j,$ and $j'$}
    \label{connectivityfig}
\end{figure}

Next, we compute $\Phi([\gamma_{i,i'}])=\Phi([\gamma_{j,j'}])=0$ and $\Psi_{X^+}([\gamma_{i,i'}])=\Psi_{X^+}([\gamma_{j,j'}])=0$, so 
\[
Q_{ii'}=Q_{i'i}=Q_{i'i'}+(\Phi-2\Psi_{X^+})([\gamma_{i,i'}])=Q_{i'i'}=A.
\]
Similarly, $Q_{jj'}=Q_{j'j}=B$. It remains to show that $Q_{ij'}=Q_{j'i}=C$ and $Q_{ji'}=Q_{i'j}=C+1$. By Theorem \ref{bigthm}, we have
\[
C=Q_{j'i'}=Q_{j'j'}+(\Phi-2\Psi_{X^+})([\gamma_{i',j'}]).
\]
Clearly, $\Psi_{X^+}([\gamma_{i,j'}])=\Psi_{X^+}([\gamma_{i',j'}])$ and $\Phi([\gamma_{i,j'}])=\Phi([\gamma_{i',j'}])$ by looking at the cases in Figure \ref{connectivityfig} or by using homology arguments. As a result, we get
\[
Q_{ij'}=Q_{j'i}=Q_{j'j'}+(\Phi-2\Psi_{X^+})([\gamma_{i,j'}])=Q_{j'i'}=C.
\]
To show $Q_{ji'}=Q_{i'j}=Q_{ji'}=C+1$, we can use a similar argument. We still have $\Psi_{X^+}([\gamma_{j,i'}])=\Psi_{X^+}([\gamma_{j',i'}])$, but now $\gamma_{j,i'}\sim \widehat{\gamma_{j',i'}}$, so Lemma \ref{phihat} gives $\Phi([\gamma_{j,i'}])=\Phi([\gamma_{j',i'}])+1$. From here, one computes
\begin{align*}
Q_{ji'}=Q_{i'j}&=Q_{i'i'}+(\Phi-2\Psi_{X^+})([\gamma_{j,i'}])\\
&= Q_{i'i'}+(\Phi-2\Psi_{X^+})([\gamma_{j',i'}])+1=Q_{j'i'}+1=C+1,
\end{align*}
which concludes the proof for the upper $2\times 2$ block of $Q$.

 To finish the proof, we need to show $Q_{iz}=Q_{i'z}$ for any $z\in \mathfrak{Z}$ and $1\leq i\leq k$. If $x_i,y_i,$ and $z$ are all on the same vertical $l\in\{l_A,l_I\}$, then we first claim that $z$ lies above or below both $x_i$ and $y_i$ in $\mathcal{D}(\tau_{u/v})$. Otherwise, since $\alpha$ is an embedded arc in $M$ and $Y$ is not an endpoint, we would get $z\in \mathfrak{X}$ or $z\in \mathfrak{Y}$, which contradicts $z\in \mathfrak{Z}$ (we assumed $k=|\mathfrak{X}|=|\mathfrak{Y}|$ was as large as possible). Now, it is trivial to see that $\Phi([\gamma_{i,z}])=\Phi([\gamma_{i',z}])$ and  $\Phi_{X^+}([\gamma_{i,z}])=\Phi_{X^+}([\gamma_{i',z}])$, so $Q_{iz}=Q_{i'z}$.

If $z$ is on $l'\neq l\in\{l_A,l_B\}$, then it is also clear that $\Phi$ and $\Psi_{X^+}$ agree on the the loops $\gamma_{i,z}$ and $\gamma_{i',z}$. Furthermore, $(\delta_{i,A}\delta_{z,I}-\delta_{i,A}\delta_{z,I})\delta_{Z,X^+}=(\delta_{i',A}\delta_{z,I}-\delta_{i',A}\delta_{z,I})\delta_{Z,X^+}$, so $Q_{iz}=Q_{i'z}$. This completes the proof that $S,A,T,$ and $Q$ can be broken into the blocks as stated in the lemma.
\end{proof}

What we want to do now is reduce the number of indices in our generating function by throwing out some of the points in $\mathcal{G}_{u/v}^1$. To do so, though, we will need to introduce the Pochhammer symbols of the form $(-t^{-1}q^2;q^2)_i$ introduced in Section \ref{qalgsec}. Now, we extend the data for the generating function by adding another linear form $K$ such that $X, [K|S|A|T], Q$ corresponds to the generating function
\begin{multline}
\label{genfuncompr}
    \sum_{\bf{d}=(d_1,...,d_{m+n})\in \mathbb{N}^{m+n}}q^{S\cdot\bf{d}}a^{A\cdot \bf{d}}q^{\bf{d}\cdot Q\cdot \bf{d}^t} t^{T\cdot\bf{d}} (-t^{-1}q^2;q^2)_{K\cdot \bf{d}}{d_1+...+d_m\brack d_1,...,d_m}{d_{m+1}+...+d_{m+n}\brack d_{m+1},...,d_{m+n}}\\ \times X[d_1+...+d_{m+n},d_1+...+d_m],
\end{multline}
where there are, of course, $m$ intersections in $\alpha\cap l_A$ contributing to the index and $n$ such intersections in $\alpha\cap l_I$, so $m\leq u$ and $n\leq v$. 

Now, we need to consider some quantum algebra to show how to get generating functions of the form shown in (\ref{genfuncompr}) from Lemma \ref{indpartition}. The following lemma is similar to Lemma 4.11 from \cite{SW21}, but adjusted to our particular situation.

\begin{lemmaN}
\label{collapsing}
    The data from
\[X, 
    \left[\begin{array}{ c| c | c |c }
    0 & S_1+1 & A_1 & T_1-1 \\
    \hline
    0 & S_1 & A_1 & T_1 \\
    \hline
    0 & S_2 & A_2 & T_2
  \end{array}\right],  \left[\begin{array}{ c | c |c }
    Q_{11}+1 & Q_{11}+L & Q_{12} \\
    \hline
    Q_{11}+U & Q_{11} & Q_{12} \\
    \hline
    Q_{21} & Q_{21} & Q_{22}
  \end{array}\right]
    \]
    and 
\[X, 
    \left[\begin{array}{ c| c | c |c }
    1 & S_1 & A_1 & T_1 \\
    \hline
    0 & S_2 & A_2 & T_2
  \end{array}\right],  \left[\begin{array}{ c |c }
    
    Q_{11} & Q_{12} \\
    \hline
    Q_{21} & Q_{22}
  \end{array}\right]
    \]
give the same generating function.
\end{lemmaN}

\begin{proof}
    We will start with the first set of data having three  blocks. We will adopt the notation of $\mathfrak{X}, \mathfrak{Y}$, and $\mathfrak{Z}$ for the three blocks, as in Lemma \ref{indpartition}. Assume that $|\mathfrak{X}|=|\mathfrak{Y}|=k$. Then, we can write $\vec{b}=(d_1,...,d_k)$ for the indices corresponding to $\mathfrak{X}$ and $\vec{c}=(d_{k+1},...,d_{2k})$ for the indices corresponding to $\mathfrak{Y}$ in the generating function. To make things easier to read, we will relabel $d_i$ as $b_i$ and $d_{k+i}$ as $c_i$, so that $\vec{b}=(b_1,...,b_k)$ and $\vec{c}=(c_1,...,c_k)$. Thus, the relevant part of the generating function that we need  is
    \[
    \sum_{\vec{b},\vec{c}\in \mathbb{N}^k} \frac{q^{S_1\cdot (\vec{b}+\vec{c})}a^{A_1\cdot(\vec{b}+\vec{c})}q^{(\vec{b}+\vec{c})\cdot Q_{11}\cdot(\vec{b}+\vec{c})^t}t^{T_1\cdot(\vec{b}+\vec{c})}(t^{-1}q)^{b_1+...+b_k} q^{b_1^2+...+b_k^2+2\sum_{i=1}^{k-1}b_{i+1}((b_1+c_1)+...+(b_i+c_i))}}{(q^2;q^2)_{b_1}...(q^2;q^2)_{b_k}(q^2; q^2)_{c_1}...(q^2;q^2)_{c_k}}.
    \]
    Now, Lemma \ref{algidentity} tells us that this expression is equal to 
    \[
    \sum_{\vec{d}\in \mathbb{N}^k} \frac{q^{S_1\cdot \vec{d}}a^{A_1\cdot \vec{d}}q^{\vec{d}\cdot Q_{11}\cdot \vec{d}^t}t^{T_1\cdot \vec{d}} (-t^{-1}q^2;q^2)_{d_1+...+d_k}}{(q^2;q^2)_{d_1}...(q^2;q^2)_{d_k}},
    \]
    where $\vec{d}=\vec{b}+\vec{c}$, so $d_i=b_i+c_i$. Finally, we observe that this is the corresponding part of the generating function for the second set of data, so we are done.
\end{proof}


\subsection{The Almost Quiver Form for $\mathcal{P}(\tau_{u/v})$ and $P(\tau_{u/v})$}

Combining Lemmas \ref{indpartition} and \ref{collapsing}, it follows that we can compute $S,A,T,$ and $Q$ for $\mathcal{P}(\tau_{u/v})$ geometrically by only performing the calculations from Theorem \ref{bigthm} for intersection points coming from $\mathfrak{Y}$ and $\mathfrak{Z}$. The only difference is that we have the linear form $K$ now. We codify this below with a theorem.

\begin{theoremN}
\label{aqfthm}
    Given $\tau_{u/v}$, let $\mathcal{G}_{u/v}^1=\mathfrak{X}\sqcup\mathfrak{Y}\sqcup\mathfrak{Z}$, where $\mathfrak{X},\mathfrak{Y},$ and $\mathfrak{Z}$ are defined the same way as in Lemma \ref{indpartition}. Then $\mathcal{P}(\tau_{u/v})$ can be written in almost quiver form, i.e.
    \begin{multline*}
        \mathcal{P}(\tau_{u/v})=\sum_{\textbf{d}\in \mathbb{N}^{m+n}}q^{S\cdot\textbf{d}+\textbf{d}\cdot Q\cdot \textbf{d}^t}a^{A\cdot \textbf{d}} t^{T\cdot\textbf{d}} (-t^{-1}q^2;q^2)_{K\cdot \textbf{d}}{d_1+...+d_m\brack d_1,...,d_m}{d_{m+1}+...+d_{m+n}\brack d_{m+1},...,d_{m+n}}\\ \times X[d_1+...+d_{m+n},d_1+...+d_m],
    \end{multline*}
    where $S,A,T,$ and $Q$ are computed by the same formulas as in Theorem \ref{bigthm}, but restricted to the free submodule of $\mathbb{Z}\mathcal{G}_{u/v}^1$ with basis $\mathfrak{Y}\sqcup\mathfrak{Z}$. The linear form $K$ is given by 
    \[
    K_i=
    \begin{cases}
        1,\qquad  \xi_i\in \mathfrak{Y}\\
        0, \qquad \xi_i\in \mathfrak{Z}.
    \end{cases}
    \]
\end{theoremN}

In \cite{JHpII26}, we will primarily be working with $P(\tau_{u/v})$ rather than $\mathcal{P}(\tau_{u/v})$, so we will briefly address the corresponding result for the fully decategorified generating function. In particular, we want to consider the almost quiver form of $P(\tau_{u/v})$ from the geometric perspective.

\begin{propositionN}
    Given $\tau_{u/v}$ and $\mathcal{G}_{u/v}^1=\mathfrak{X}\sqcup\mathfrak{Y}\sqcup\mathfrak{Z}$ as in Theorem \ref{aqfthm}, $P(\tau_{u/v})$ can be written in almost quiver form, i.e.
    \begin{multline*}
        P(\tau_{u/v})=\sum_{\textbf{d}\in \mathbb{N}^{m+n}}(-q)^{S\cdot\textbf{d}}q^{\textbf{d}\cdot Q\cdot \textbf{d}^t}a^{A\cdot \textbf{d}} (q^2;q^2)_{K\cdot \textbf{d}}{d_1+...+d_m\brack d_1,...,d_m}{d_{m+1}+...+d_{m+n}\brack d_{m+1},...,d_{m+n}}\\ \times X[d_1+...+d_{m+n},d_1+...+d_m],
    \end{multline*}
    where $S,A,K,$ and $Q$ are the same as in Theorem \ref{aqfthm}.
\end{propositionN}

Finally, we will provide some examples to illustrate how Theorems \ref{bigthm} and \ref{aqfthm} can be applied to particular rational tangles.

\section{Examples}
\label{exsec}

\subsection{Tangles of the Form $\tau_{n/1}$}
\label{torustanglesex}

Now, we determine what $\mathcal{P}(\tau_{u/v})$ looks like for a particular infinite family of tangles. The tangles $\tau_{n/1}$ are the tangles whose closures give the $(2,n)$-torus knots and links. This is a particularly easy class to study, as there is a nice structure to the linear and quadratic forms. 

We will use a different labeling of the indices than the one used in Section 3. In this case, we will label them according to the $\prec$ ordering, meaning that $\xi_i\preceq \xi_j$ if and only if $i\leq j$. There is a particular symmetry in $Q$ with respect to this ordering, so we define the following notation to make this structure easier to describe symbolically.

\begin{definitionN}
    Given $k\in\mathbb{N}$, let
    \[
    \mathcal{I}_k=\{(i,j)\in\mathbb{N}^2 : i=k \, \text{and} \, j\leq k, \text{or} \, j=k \, \text{and} \, i\leq k\}.
    \]
    Thus, if $(i,j)\in \mathcal{I}_k$, then the $(i,j)$th entry of a matrix is in the $k$th row or $k$th column, but it is not in a $l$th row or column for $l>k$.
\end{definitionN}

Given this notation, we have the following way to express the data for $\mathcal{P}(\tau_{u/v})$.

\begin{exampleN}
\label{torustangleprop}
    For the tangle $\tau_{n/1}$, given the labeling of the $\xi_i$ described above, the data for $\mathcal{P}(\tau_{u/v})$ is given by
   \[
    \begin{cases}
        S_i=n-i+1\\
        A_i=0\\
        T_i=i-n-1\\
        \begin{cases}
            Q_{ij}=n-k, \qquad & \text{if} \, (i,j)\in\mathcal{I}_k, 1 \leq k \leq n\\
            Q_{ij}=0,\qquad & \text{if} \, (i,j)\in\mathcal{I}_k, k=n+1.
        \end{cases}
    \end{cases}
    \] 
\end{exampleN}

\begin{proof}
    First, observe that $\tau_{n/1}$ has state $(UP,Y|X^-|X^+)$ if $n$ is even and $(UP,X^-|Y|X^+)$ if $n$ is odd. For $\tau_{n/1}$, it is not difficult to check that $\Phi([\gamma_{i,j}])=0$ whenever $\xi_i\prec \xi_j$ and $Z$ is never $X^+$. Since $Q_{ij}=Q_{ji}$ and none of the loops used in Theorem \ref{bigthm} wind around $X^+$, it follows that if $(i,j),(i',j')\in \mathcal{I}_k$, then $Q_{ij}=Q_{i'j'}=Q_{kk}$. Thus, we only need to check what happens for the linear forms, including the diagonal of $Q$. The fact that $A_i=0$ for all $i$ follows from the observation that $\Psi_{X^+}$ is zero on all our loops; then the theorem follows from the easy observations that
    \[
    \Psi_{\{X^-,Y\}}([\gamma_i])=
    \begin{cases}
        n-i, \qquad &i\leq n\\
        0, \qquad &i=n+1
    \end{cases}
    \]
    and that the $\pm 1$ parts above come from the $\delta_{i,A}$ terms in $S$ and $T$.  
\end{proof}

\begin{example}
    Figure \ref{tau81ex} shows $\mathcal{D}(\tau_{8/1})$ with the intersection points labeled according to the convention described in this section. Then, we have 

\begin{figure}[h]
\begin{minipage}[b]{.45\textwidth}

$\left[\begin{array}{ c | c |c }
    S & A & T\end{array}\right] =
  \left[\begin{array}{ c | c |c }
    8 & 0 & -8 \\
    7 & 0 & -7 \\
    6 & 0 & -6 \\
    5 & 0 & -5 \\
    4 & 0 & -4 \\
    3 & 0 & -3 \\
    2 & 0 & -2 \\
    1 & 0 & -1 \\
    0 & 0 & 0
    \end{array}\right]$
    \end{minipage}
    \begin{minipage}[b]{.45\textwidth}
    $ Q=\begin{bmatrix}
    7 & 6 & 5 & 4 & 3 & 2 & 1 & 0 & 0 \\
    6 & 6 & 5 & 4 & 3 & 2 & 1 & 0 & 0 \\
    5 & 5 & 5 & 4 & 3 & 2 & 1 & 0 & 0 \\
    4 & 4 & 4 & 4 & 3 & 2 & 1 & 0 & 0 \\
    3 & 3 & 3 & 3 & 3 & 2 & 1 & 0 & 0 \\
    2 & 2 & 2 & 2 & 2 & 2 & 1 & 0 & 0 \\
    1 & 1 & 1 & 1 & 1 & 1 & 1 & 0 & 0 \\
    0 & 0 & 0 & 0 & 0 & 0 & 0 & 0 & 0 \\
    0 & 0 & 0 & 0 & 0 & 0 & 0 & 0 & 0 
    \end{bmatrix}$
    \end{minipage}
\end{figure}

\begin{figure}
\begin{tikzpicture}
    \node at (0,0) {$\mathcal{D}(\tau_{8/1})=$};
    \node at (3.2,0) {\includegraphics[height=4.5cm, angle=0]{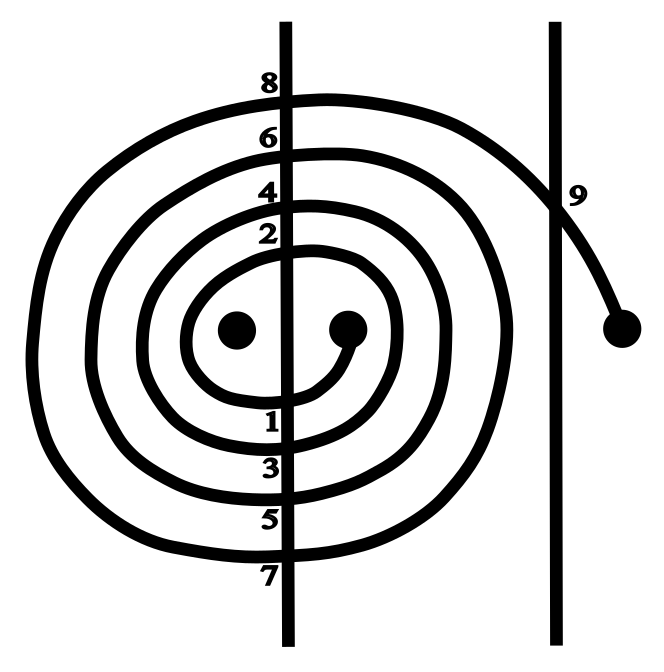}};
\end{tikzpicture}
\caption{$\mathcal{D}(\tau_{8/1})$ with $\mathcal{G}_{u/v}^1$ ordered by $\prec$}
\label{tau81ex}
\end{figure}
by Example \ref{torustangleprop}.
\end{example}

If we carefully relabel the intersection points, we can apply Theorem \ref{aqfthm} to reduce the number of indices at the expense of adding some Pochhammer symbols. It will be helpful to consider the cases where $n$ is even or odd separately. 

If $n$ is even, note that the top $n/2$ intersection points are labeled by the even numbers $2,4,...,n$ from the ``inside-out'', and the bottom points are labeled by $1,3,...,n-1$ from the inside-out as well. Our new convention is that we label the bottom $n/2$ points $1,2,...,n/2$ and the top $n/2$ points $n/2+1,n/2+2,...,n$, both from the inside-out. The bottom $n/2$ points give $\mathfrak{X}$, the top $n/2$ points give $\mathfrak{Y}$, and $\mathfrak{Z}$ is the one inactive intersection point. Then, applying Theorem \ref{aqfthm} gives the following.

\begin{exampleN}
\label{torustangleeven}
Consider $\tau_{n/1}$ with $n$ even. Then, $\mathcal{P}(\tau_{n/1})$ is given by the data below, where there are $n/2+1$ indices coming from $\mathfrak{Y}\sqcup\mathfrak{Z}$:
 \[
    \begin{cases}
        K_i = \delta_{i\leq n/2}\\
        S_i=-T_i=n-2i+1\\
        A_i=0\\
        \begin{cases}
            Q_{ij}=n-2k, \qquad & \text{if} \,\, (i,j)\in\mathcal{I}_k, \,1 \leq k \leq n/2\\
            Q_{ij}=0,\qquad & \text{if} \,\, (i,j)\in\mathcal{I}_k,\, k=n/2+1,
        \end{cases}
    \end{cases}
    \]
    and the labels for $\mathfrak{Y}\sqcup\mathfrak{Z}$ are shifted back down by $n/2$ from their original labels. 
\end{exampleN}

It should be noted that, applying the work in section 4 of \cite{SW21}, one can extract an algorithm to calculate $K,S,A,$ and $Q$ inductively for $\tau_{2k/1}$ by applying $T^2$ to $\tau_{0/1}$ in $k$ steps. Using these conventions of Sto\v si\'c and Wedrich, one would get the same formulas for $K,A,$ and $Q$, but this would lead to the formula $S_i=-2i+1$ for $S$. This has the effect of scaling each $\langle\langle\tau_{n/1}\rangle\rangle_j$ in $\mathcal{P}(\tau_{n/1})$ down by $q^{nj}$, but all relative gradings are the same.

\begin{example}
    We now consider the generating function data for the tangle $\tau_{8/1}$ again, but this time with the $K$ linear form. In this case, the active indices in $\mathfrak{Y}$ are coming from the intersection points labeled with even numbers in Figure \ref{tau81ex}. The linear and quadratic forms are given below:

    \begin{figure}[H]
    \centering
\begin{minipage}[b]{.45\textwidth}

$\left[\begin{array}{ c| c | c |c }
   K&  S & A & T\end{array}\right] =
  \left[\begin{array}{ c | c | c |c }
    1 & 7 & 0 & -7 \\
    1 & 5 & 0 & -5 \\
    1 & 3 & 0 & -3 \\
    1 & 1 & 0 & -1 \\
    0 & 0 & 0 & 0 \\
   
    \end{array}\right]$
    \end{minipage}
    \begin{minipage}[b]{.45\textwidth}
    $ Q=\begin{bmatrix}
    6 & 4 & 2 & 0 & 0 \\
    4 & 4 & 2 & 0 & 0 \\
    2 & 2 & 2 & 0 & 0 \\
    0 & 0 & 0 & 0 & 0 \\
    0 & 0 & 0 & 0 & 0 
    \end{bmatrix}$.
    \end{minipage}
\end{figure} 

\end{example}

We can do something similar for $\tau_{n/1}$ if $n$ is odd. In this case, there will be $\frac{n+1}{2}$ active indices coming from the top points on the active side. Label the intersection points as follows: assign the top $\frac{n+1}{2}$ points $1,2,...,\frac{n-1}{2},n$ and the bottom $\frac{n-1}{2}$ points $\frac{n+1}{2},\frac{n+3}{2},...,n-1$, both working from the inside-out again. Then, $\mathfrak{X}$ consists of the points labeled $1,...,\frac{n-1}{2}$, $\mathfrak{Y}$ consists of the points labeled $\frac{n+1}{2},\frac{n+3}{2},...,n-1$, and $\mathfrak{Z}$ includes the point labeled $n$ and the single inactive intersection point. Applying Theorem \ref{aqfthm} gives the following.

\begin{exampleN}
\label{taun_odd}
    Consider $\tau_{n/1}$ with $n$ odd. Then, $\mathcal{P}(\tau_{n/1})$ is given by the data below, where there are $\frac{n+3}{2}$ indices coming from $\mathfrak{Y}\sqcup\mathfrak{Z}$:
    \[
    \begin{cases}
    K_i = \delta_{i\leq \frac{n-1}{2}}\\
        \begin{cases}
        S_i=-T_i=n-2i+1, \qquad & \text{if}\, 1\leq i \leq \frac{n-1}{2}\\
        S_i= -T_i= 1, \qquad & \text{if}\, i=\frac{n+1}{2} \\
        S_i=-T_i=0, \qquad & \text{if}\, i=\frac{n+3}{2}
        \end{cases}\\
        A_i=0\\
        \begin{cases}
            Q_{ij}=n-2k, \qquad & \text{if} \, (i,j)\in\mathcal{I}_k, 1 \leq k \leq \frac{n-1}{2}\\
            Q_{ij}=0,\qquad & \text{if} \, (i,j)\in\mathcal{I}_k, k\in \{\frac{n+1}{2},\frac{n+3}{2}\},
        \end{cases}
    \end{cases}
    \]
    and the labels for $\mathfrak{Y}\sqcup\mathfrak{Z}$ are shifted back down by $\frac{n-1}{2}$ from their original labels. 
    \end{exampleN}

\begin{example}
    Next, we consider $\tau_{7/1}$ to provide a concrete example for the case where $n$ is odd. The intersection points are labeled below according to the prescribed method, so that $\mathfrak{Y}$ includes the points labeled $4,5,$ and $6$, and $\mathfrak{Z}$ includes $7$ and $8$.  
    \[
    \centering
\includegraphics[height=3.5cm, angle=0]{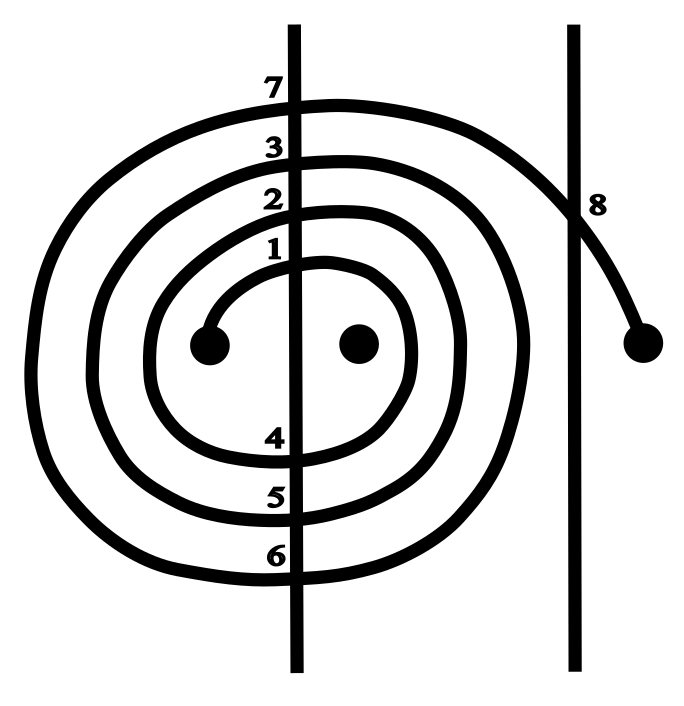}
    \]
    In this case, the generating function data are given by
   \begin{figure}[H]
    \centering
\begin{minipage}[b]{.45\textwidth}

$\left[\begin{array}{ c| c | c |c }
   K&  S & A & T\end{array}\right] =
  \left[\begin{array}{ c | c | c |c }
    1 & 6 & 0 & -6 \\
    1 & 4 & 0 & -4 \\
    1 & 2 & 0 & -2 \\
    0 & 1 & 0 & -1 \\
    0 & 0 & 0 & 0 \\
   
    \end{array}\right]$
    \end{minipage}
    \begin{minipage}[b]{.45\textwidth}
    $ Q=\begin{bmatrix}
    5 & 3 & 1 & 0 & 0 \\
    3 & 3 & 1 & 0 & 0 \\
    1 & 1 & 1 & 0 & 0 \\
    0 & 0 & 0 & 0 & 0 \\
    0 & 0 & 0 & 0 & 0 
    \end{bmatrix}.$
    \end{minipage}
\end{figure}      
\end{example}

\subsection{The Tangle $\tau_{10/3}$}
\label{T103sec}

Now, we consider the tangle $\tau_{10/3}$. Just as in the case of $\tau_{n/1}$ with $n$ even, all intersection points $\alpha \cap l_A$ pair up and we can reduce the number of indices in $\mathcal{P}(\tau_{10/3})$ by Theorem \ref{aqfthm}. In particular, we can partition the active intersections by $\mathfrak{X}$ and $\mathfrak{Y}$, where $\mathfrak{X}$ is the bottom five intersection points and $\mathfrak{Y}$ is the top five intersection points in $\alpha_{10/3}\cap l_A$. Theorem \ref{aqfthm} tells us then that we only need to consider $\mathfrak{Y}$ and $\mathfrak{Z}$, which is the set of three inactive intersections. 

Of course, there are many ways to order these eight points, but we have chosen to do so by restricting the standard ordering to $\mathfrak{Y}\sqcup\mathfrak{Z}$. See Figure \ref{Dt103fig}. We will first present the final result and then share the work behind a couple of entries. 

Here is the data for the generating function, given the specified ordering of indices:

\begin{figure}[h]
    \centering
\begin{minipage}[b]{.45\textwidth}

$\left[\begin{array}{ c| c | c |c }
   K&  S & A & T\end{array}\right] =
  \left[\begin{array}{ c | c | c |c }
    1 & 5 & 2 & -5 \\
    1 & 3 & 2 & -3 \\
    1 & 2 & 0 & -2 \\
    1 & 3 & 0 & -3 \\
    1 & 1 & 0 & -1 \\
    0 & 2 & 2 & -2 \\
    0 & 1 & 0 & -1 \\
    0 & 0 & 0 & 0
   
    \end{array}\right]$
    \end{minipage}
    \begin{minipage}[b]{.45\textwidth}
    $ Q=\begin{bmatrix}
    0 & -2 & 0 & 1 & 0 & -2 & 0 & 0 \\
    -2 & -2 & -1 & -1 & -1 & -2 & 0 & 0 \\
    0 & -1 & 1 & 1 & 0 & -1 & 1 & 0\\
    1 & -1 & 1 & 2 & 0 & -1 & 1 & 0\\
    0 & -1 & 0 & 0 & 0 & -1 & 1 & 0 \\
    -2 & -2 & -1 & -1 & 1 & -2 & -1 & -1 \\
    0 & 0 & 1 & 1 & 1 & -1 & 1 & 0 \\
    0 & 0 & 0 & 0 & 0 & -1 & 0 & 0
    \end{bmatrix}.$
    \end{minipage}
\end{figure}    

 \begin{figure}
     \begin{tikzpicture}
         \node at (0,0) {$\mathcal{D}(\tau_{10/3})=$};
         \node at (4,0) {\includegraphics[height=5cm, angle=0]{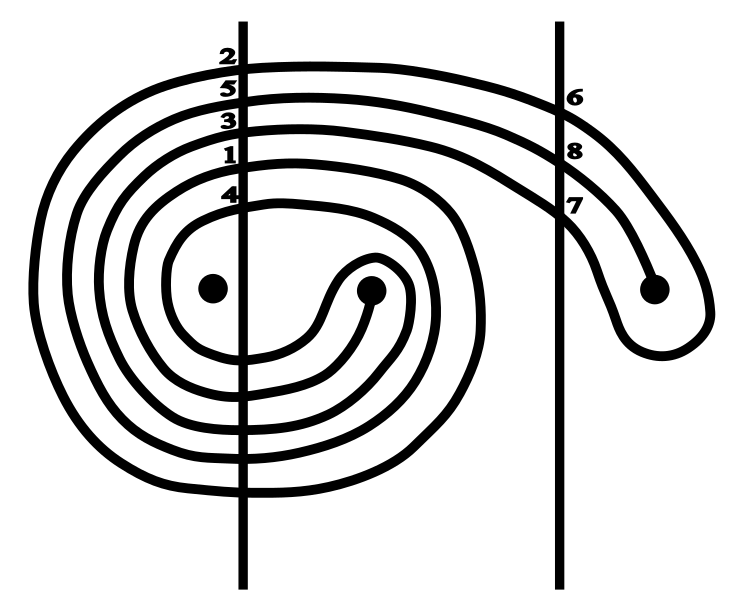}};
     \end{tikzpicture}
     \caption{$\mathcal{D}(\tau_{10/3})$ along with the labeling of the intersection points giving the active and inactive indices}
     \label{Dt103fig}
 \end{figure}

It should also be noted that $\tau_{10/3}$ has state $(UP,Y|X^-|X^+)$, which we need to know for making computations. Now, we show explicitly how to compute $Q_{44}$ and $Q_{24}$. For $Q_{44}$, we have drawn the path $\gamma_4$, which we need to compute $\Psi_{\{X^-,Y\}}([\gamma_4])$. Next to it is the braid we get from $\gamma_4$ by adding constant loops for $Y$ and $X^-$ in $\text{Conf}^3(\mathbb{C})$. We only include the strands for $X^-$ and $Y$ since it should be clear that the strand we would get from $X^+$ would be completely untangled from the others, which means $\Psi_{X^+}([\gamma_4])=0$. Also, since $Z=X^-$, it follows that $Q_{44}=\Phi_{\{X^-,Y\}}([\gamma_4])$, which we easily determine to be 2 from the braid picture.

\[
\vcenter{\hbox{\includegraphics[height=5cm,angle=0]{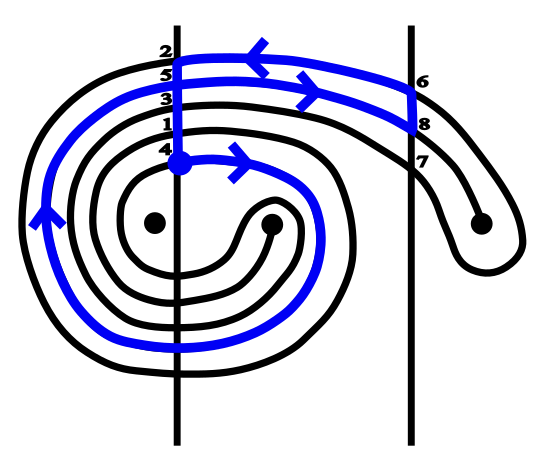}}} \qquad \qquad 
\vcenter{\hbox{\includegraphics[height=4cm,angle=0]{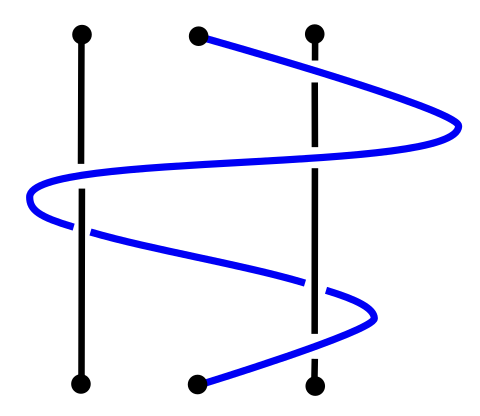}}}
\]

For computing $Q_{42}$, we have drawn the loop $\gamma_{2,4} \subset \text{Conf}^2(M)$ used to compute $\Phi([\gamma_{2,4}])$, along with the braid we get from the path in $\text{Conf}^2(M)$. From this braid, we can easily tell that $\Phi([\gamma_{2,4}])=-1$, as there are two negative crossings and 1 positive crossing.

\[
\vcenter{\hbox{\includegraphics[height=5cm,angle=0]{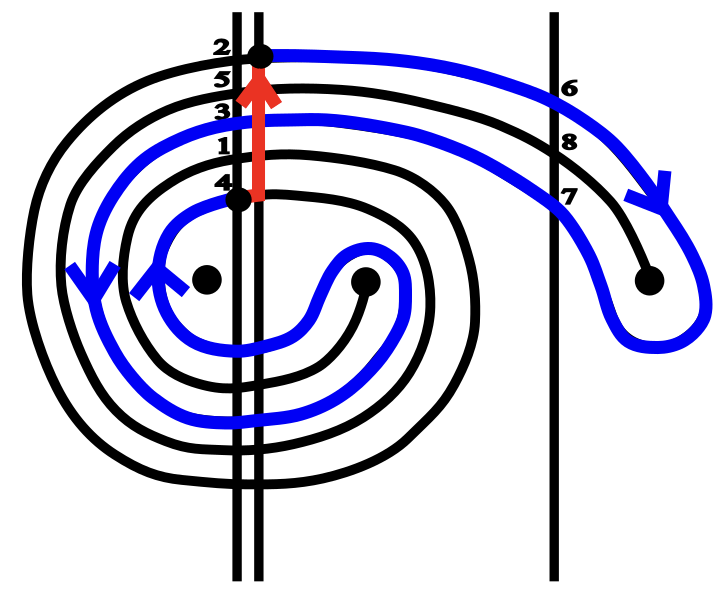}}} \qquad \qquad 
\vcenter{\hbox{\includegraphics[height=3.5cm,angle=0]{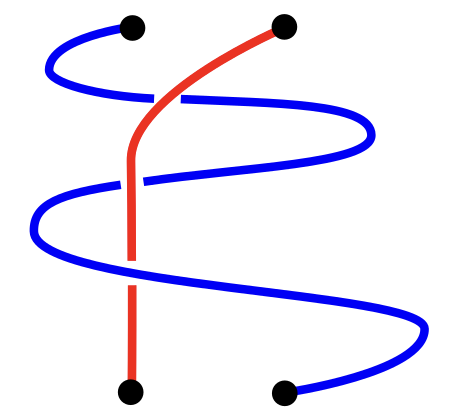}}}
\]

It is also not difficult to see that $\Psi_{X^+}([\gamma_{2,4}])=1$, so we get
\[
Q_{42}=Q_{44}+(\Phi-2\Psi_{X^+})([\gamma_{2,4}])= Q_{44}-3=-1.
\]

\printbibliography
\end{document}